%% file: flat.tex
\documentclass[10pt, a4paper]{amsart}

\pdfoutput=1

\usepackage[T1]{fontenc}

\usepackage{mathrsfs, amsmath, amsthm, amssymb, stmaryrd, enumerate, bbm}%
\usepackage{tikz}%
\usetikzlibrary{knots}
\usepackage[linktocpage]{hyperref}%
\usepackage[all]{xy}%
\xyoption{2cell}%
\UseAllTwocells%

\usepackage[centerboxes]{ytableau}

\newdir{t>}{{\UseTips\dir{>}}}
\newdir{ >}{{}*!/-5pt/@{>}}

\DeclareMathOperator{\id}{id}
\DeclareMathOperator{\el}{el}
\DeclareMathOperator{\op}{op}
\DeclareMathOperator{\Lan}{Lan}

\DeclareMathOperator{\Mod}{\mathbf{Mod}}

\DeclareMathOperator{\fp}{fp}

\DeclareMathOperator{\End}{End}
\DeclareMathOperator{\Spec}{Spec}

\DeclareMathOperator{\Sym}{Sym}
\DeclareMathOperator{\sgn}{sgn}
\DeclareMathOperator{\coev}{coev}
\DeclareMathOperator{\ev}{ev}

\DeclareMathOperator{\Ind}{Ind}

\DeclareMathOperator{\slm}{\mathrm{SymLaxMon}}
\DeclareMathOperator{\ssm}{\mathrm{SymStrMon}}
\DeclareMathOperator{\cell}{\mathrm{Cell}}

\DeclareMathOperator{\sh}{Sh}

\DeclareMathOperator{\colim}{colim}

\DeclareMathOperator{\Lex}{\mathbf{Lex}}

\DeclareMathOperator{\CAT}{\mathbf{CAT}}
\DeclareMathOperator{\Set}{\mathbf{Set}}
\DeclareMathOperator{\Ab}{\mathbf{Ab}}

\DeclareMathOperator{\Comod}{\mathbf{Comod}}

\DeclareMathOperator{\CAlg}{\mathrm{CAlg}}

\newcommand{\ca}[1]{\mathscr{#1}}

\DeclareMathOperator{\U}{\mathbbm{1}}

\newcommand{\dt}{\mathop{\bar{\otimes}}}

\newcommand{\ten}[1]{\mathop{{\otimes}_{#1}}}

\newcommand{\defl}{\mathrel{\mathop:}=}

\theoremstyle{plain}
\newtheorem{thm}{Theorem}[subsection]
\newtheorem*{thm*}{Theorem}
\newtheorem{prop}[thm]{Proposition}
\newtheorem{lemma}[thm]{Lemma}
\newtheorem{cor}[thm]{Corollary}

\theoremstyle{definition}
\newtheorem{example}[thm]{Example}
\newtheorem{rmk}[thm]{Remark}
\newtheorem{dfn}[thm]{Definition}

\newtheoremstyle{citing}{}{}{\itshape}{}{\bfseries}{.}{ }{\thmnote{#3}}
\theoremstyle{citing}

\newtheoremstyle{citingdfn}{}{}{}{}{\bfseries}{.}{ }{\thmnote{#3}}
\theoremstyle{citingdfn}

\numberwithin{equation}{section}


\author{Daniel Sch\"appi}
\thanks{This research was supported by the DFG grant: SFB 1085 ``Higher invariants.''}
\address{Fakult{\"a}t f{\"u}r Mathematik,
Universit{\"a}t Regensburg,
93040 Regensburg,
Germany}
\email{daniel.schaeppi@ur.de}

\title{Flat replacements of homology theories}

\begin{document}

\begin{abstract}
 
 To a homology theory one can associate an additive site and a new homological functor with values in the category of additive sheaves on that site. If this category of sheaves can be shown to be equivalent to a category of comodules of a Hopf algebroid, then we obtain a new homology theory by composing with the underlying module functor. This new homology theory is always flat and we call it a flat replacement of the original theory. For example, Pstr\k{a}gowski has shown that complex cobordism is a flat replacement of singular homology. In this article we study the basic properties of the sites associated to homology theories and we prove an existence theorem for flat replacements.

\end{abstract}

\maketitle

\tableofcontents

\input{introduction}
\input{sites}

\input{recognition}
\input{duals}

\bibliographystyle{amsalpha}
\bibliography{flat}

\end{document}

%% file: introduction.tex
\section{Introduction}

   Let $\mathrm{SH}$ denote the stable homotopy category. If $E$ is a commutative ring object in $\mathrm{SH}$, then $E_{\ast} X \defl \pi_{\ast} (E \wedge X)$ defines a homological functor with values in modules over the graded ring $E_{\ast} \defl \pi_{\ast} (E)$. In \cite{PSTRAGOWSKI}, Pstr\k{a}gowski observed that we can associate an additive site to $E$ as follows. The objects of $\ca{A}$ are the finite spectra $X$ such that $E_{\ast} X$ has a dual. The additive Grothendieck topology is induced by the singleton coverage consisting of $E_{\ast}$-epimorphisms: morphisms $p \colon X \rightarrow Y$ such that $E_{\ast} p$ is an epimorphism. If $E$ is an Adams-type homology theory, for example, $E=\mathrm{S}$, $\mathrm{MU}$, $\mathrm{MO}$, $\mathrm{MSp}$, $\mathrm{K}$, $\mathrm{KO}$, then the category $\sh(\ca{A})$ of additive sheaves on $\ca{A}$ is equivalent to the category of comodules of the Hopf algebroid $(E_{\ast},E_{\ast} E)$ (see \cite[Remark~3.26]{PSTRAGOWSKI}).
   
 Let $J \colon \ca{A} \rightarrow \mathrm{SH}$ denote the inclusion. By sending $X \in \mathrm{SH}$ to the additive sheaf associated to $\mathrm{SH}(J-,X)$, we obtain a new homological functor with values in the category $\sh(\ca{A})$. We write $[\ca{A}^{\op},\Ab]$ for the category of additive presheaves and $L \colon [\ca{A}^{\op},\Ab] \rightarrow \sh(\ca{A})$ for the associated sheaf functor. If $E_{\ast} \colon \ca{A} \rightarrow \Mod_{E_{\ast}}$ induces an \emph{exact} functor
 \[
 \sh(\ca{A}) \rightarrow \Mod_{E_{\ast}}
 \]
 via left Kan extension, then the composite
 \[
\xymatrix{ \mathrm{SH} \ar[r] & [\ca{A}^{\op},\Ab] \ar[r]^-{L} & \sh(\ca{A}) \ar[r] & \Mod_{E_{\ast}} }
 \]
 is naturally isomorphic to the original functor $E_{\ast}$. In this case, $E$ is \emph{topologically flat} in the sense of \cite[Definition~1.4.5]{HOVEY} (see Proposition~\ref{prop:Adams_quasi_ring_characterization}). If $E$ is topologically flat and $E^{\prime}$ is Landweber exact over $E$, then $E^{\prime}$ is also topologically flat by \cite[Theorem~1.4.9]{HOVEY}. Landweber exact homology theories have been widely studied, for example the elliptic cohomology theories of \cite{LANDWEBER_RAVENEL_STONG}, the Johnson--Wilson spectra $E(n)$ and Morava $E$-theory $E_n$ are all Landweber exact over $\mathrm{MU}$ or $\mathrm{BP}$ (see \cite[Example~0.1]{HOVEY_STRICKLAND}). The above condition also makes sense for other triangulated categories $\ca{T}$ in place of $\mathrm{SH}$, hence we say that the homology theory $E_{\ast}$ is \emph{$\mathrm{SH}$-flat} if the left Kan extension above is exact.

 If the left Kan extension is \emph{not} exact, it can still happen that $\sh(\ca{A})$ is equivalent to the category $\Comod(A,\Gamma)$ of comodules of a flat commutative Hopf algebroid $(A,\Gamma)$. In this case, we get a new homological functor with values in $A$-modules by composing with the underlying module functor. This new homology theory is always $\mathrm{SH}$-flat. We call it a \emph{flat replacement} of $E_{\ast}$. By Brown representability, the new homology theory is of the form $\pi_{\ast}(E^{\prime} \otimes -)$ for some spectrum $E^{\prime}$. In general, this spectrum will only be a quasi-ring spectrum (that is, the usual ring axioms hold only up to phantom maps). In specific cases, one can however find a ring spectrum which gives a flat replacement.
 
 For example, Pstr\k{a}gowski has shown that complex cobordism $\mathrm{MU}$ is a flat replacement of singular homology $\mathrm{H}\mathbb{Z}$. This is based on a result of Conner and Smith \cite[\S 3]{CONNER_SMITH} that $\mathrm{H}\mathbb{Z}_{\ast} X$ is finitely generated projective if and only $\mathrm{MU}_{\ast} X$ is finitely generated projective (see Theorem~\ref{thm:complex_cobordism}). Motivated by this example, Pstr\k{a}gowski asked if it is possible to prove a general existence theorem for flat replacements.
 
 In this article, we focus on two goals. Firstly, we define sites associated to homological functors whose domain is an arbitrary tensor triangulated category instead of $\mathrm{SH}$ and we study their basic properties in \S \ref{section:sites}. Examples of such sites have been used in \cite{SCHAEPPI_MOTIVES} to give unconditional constructions of Tannakian categories of motives. The categories of additive sheaves of these sites are symmetric monoidal finitely presentable abelian categories which are generated by duals and the dual objects are finitely presentable. We call such categories \emph{Grothendieck tensor categories}. We thus obtain a novel class of examples of Grothendieck tensor categories built from a tensor triangulated category equipped with a lax monoidal homological functor.
 
 The second half of this article is devoted to the existence theorem for flat replacements. We show that a $\mathrm{SH}$-flat replacement exists if the coefficient ring $E_{\ast}$ is connected (that is, it has no non-trivial idempotents) and $2$ is a unit in $E_{\ast}$ (see Theorem~\ref{thm:refined_existence_of_Adams_replacement}). This theorem shows for example that $\mathrm{H} \mathbb{Z}[\frac{1}{2}]$ and $\mathrm{H}\mathbb{Z}_{(p)}$ for $p$ an odd prime have $\mathrm{SH}$-flat replacements. In these cases, the arguments of Conner--Smith can be localized to show that $\mathrm{MU}[\frac{1}{2}]$ and $\mathrm{MU}_{(p)}$ give $\mathrm{SH}$-flat replacements of $\mathrm{H} \mathbb{Z}[\frac{1}{2}]$ and $\mathrm{H}\mathbb{Z}_{(p)}$. Since all flat replacements have by construction the same Bousfield class, the existence theorem gives a spectrum with the same Bousfield class as (localized) complex cobordism starting only from singular homology.
 
 The existence theorem applies also to other triangulated categories in place of $\mathrm{SH}$. For example, let $R$ be a countable $\mathbb{Z}[\frac{1}{2}]$-algebra. Then the unbounded derived category $\ca{D}(R)$ is a triangulated category which satisfies Brown representability. Examples of commutative ring objects in $\ca{D}(R)$ are given by commutative differential graded $R$-algebras $E$. The homology theory represented by $E$ is in this case given by $H_{\ast}(E \mathop{\otimes^{L}} -)$. The existence theorem for flat replacements shows that $E$ has a $\ca{D}(R)$-flat replacement if the coefficient ring $H_{\ast}(E)$ is connected.

 A specific case is $R=\mathbb{Q}[x,y] \slash (xy)$ and $E=R \slash y$, concentrated in degree zero. This is not $\ca{D}(R)$-flat and one can show that no flat $R$-algebra, concentrated in degree zero, gives a flat replacement of $E$ (see Proposition~\ref{prop:no_flat_algebra_gives_flat_replacement}). Is it possible to find a concrete description of a flat replacement of $R \slash y$? It would be interesting to compute some explicit examples of flat replacements in the context of derived categories.

 The proof of the existence theorem for flat replacements uses a version of the Tannakian recognition theorem of \cite{DELIGNE, DELIGNE_TENSORIELLES} and its generalizations \cite{SCHAEPPI_GEOMETRIC, SCHAEPPI_COLIMITS}. It is based on the following observation. If $B$ is a faithfully flat algebra in the symmetric monoidal category $\sh(\ca{A})$ with the property that $B$ is a projective generator of the category $\sh(\ca{A})_B$ of $B$-modules in $\sh(\ca{A})$, then the free $B$-module functor defines a faithful and exact strong monoidal left adjoint from $\sh(\ca{A})$ to some category of modules. In this situation, it follows that $\sh(\ca{A})$ is equivalent to the category of comodules of a flat commutative Hopf algebroid (cf.\ Remark~\ref{rmk:Hopf_algebroid}). We construct such an algebra $B$ using the notion of \emph{locally free objects}: these are objects $X$ for which there exists a faithfully flat algebra $A$ such that the free $A$-module on $X$ is isomorphic to a direct sum of shifts of the unit object. The key step in the proof of the existence theorem lies in establishing the fact that all sheaves associated to representable presheaves are locally free in this sense.
 
\section*{Acknowledgements}
 
 The notion of flat replacements and the fact that complex cobordism is a flat replacement of singular homology are due to Piotr Pstr\k{a}gowski. Piotr asked me if it would be possible to use techniques from (generalized) Tannaka duality to establish an existence theorem for flat replacements. I am very grateful to Piotr for several helpful conversations about this topic. Proposition~\ref{prop:non-zero_detection} is due to Kevin Coulembier. My original proof of this proposition only worked under restrictive conditions on the coefficient ring of the homology theory. I thank Niko Naumann for helpful discussions.

%% file: sites.tex
\section{Grothendieck tensor categories associated to homology theories}

\subsection{Coefficient categories}\label{section:coefficient_categories}
 Recall that an object $C$ of a category $\ca{C}$ is called \emph{finitely presentable} if the hom-functor $\ca{C}(C,-) \colon \ca{C} \rightarrow \Set$ preserves filtered colimits. A category is called \emph{locally finitely presentable} if it is cocomplete and it has a strong generator consisting of finitely presentable objects.

\begin{dfn}\label{dfn:Grothendieck_tensor}
 A \emph{Grothendieck tensor category} is a locally finitely presentable abelian symmetric monoidal closed category $\ca{C}$ such that the unit object $\U \in \ca{C}$ is finitely presentable and the duals form a strong generator of $\ca{C}$. A Grothendieck tensor category is called a \emph{coefficient category} if the unit $\U$ is in addition projective.
\end{dfn}

Examples of coefficient categories are additive presheaf categories endowed with the Day convolution monoidal structure, specifically the category of graded abelian groups or the category of Mackey functors of a finite group. In fact, all coefficent categories are of this form.  Recall that a symmetric monoidal category is called \emph{rigid} if every object has a dual.

\begin{prop}\label{prop:coeff_cats_presheaf}
 If $\ca{K}$ is a coefficient category, then there exists an essentially small $\Ab$-enriched rigid symmetric monoidal category $\ca{P}$ and an equivalence between $\ca{K}$ and the additive presheaf category $[\ca{P}^{\op},\Ab]$ with the Day convolution symmetric monoidal structure. 
\end{prop}

\begin{proof}
 We can let $\ca{P}$ be the full subcategory of duals in $\ca{K}$. The fact that $\ca{K}$ is a coefficient category implies that $\ca{P}$ consists of small projective objects, hence that the inclusion of $\ca{P}$ in $\ca{K}$ induces the desired equivalence by \cite[Theorem~5.26]{KELLY_BASIC}. Compatibility with the Day convolution symmetric monoidal structure follows from \cite[Theorem~5.1]{IM_KELLY}. 
\end{proof}

 We are interested in homology theories with values in coefficient categories (for example graded abelian groups or Mackey functors) which are suitably compatible with the monoidal struture. In this first section, we study only the latter (compatibility with the monoidal structure). For the remainder of this section, we fix an $\Ab$-enriched monoidal category $\ca{B}$, a rigid $\Ab$-enriched category $\ca{P}$, and an additive symmetric strong monoidal functor $K \colon \ca{P} \rightarrow \ca{B}$.  
 
 The examples we have in mind for $\ca{B}$ are the stable homotopy category, Voevodsky's triangulated category of mixed motives, or the category of genuine $G$-spectra for a compact Lie group $G$. A general possible choice for $K \colon \ca{P} \rightarrow \ca{B}$ is the inclusion of the (non-full) smallest $\Ab$-enriched symmetric monoidal subcategory consisting of invertible objects. In fact, we can further restrict to objects in a subgroup of the Picard group of $\ca{B}$ (for example to obtain $\mathrm{RO}(G)$-graded homology theories frequently studied in equivariant stable homotopy theory). Another example would be to take the Burnside category for $\ca{P}$, resulting in Mackey functor valued homology theories.
 
 The homology theories we are interested all arise from the following construction. We write $\ca{K}=[\ca{P}^{\op},\Ab]$ for the coefficient category associated to $\ca{P}$ and we denote the Day convolution tensor product by $\dt$. Recall that this tensor product is defined as the coend
 \[
 V \dt W (e) \defl \int^{p,q \in \ca{P}} Vp \otimes Wq \otimes \ca{P}(e,p\otimes q)
 \]
 of abelian groups. We will sometimes write the tensor product of abelian groups simply as concatenation and the tensor product in $\ca{C}$ as $\cdot$ to increase the legibility of large diagrams. Evaluation in $\U$ defines a symmetric lax monoidal functor
 \[
 U\defl \ev_{\U} \colon \ca{K} \rightarrow \Ab
 \]
 and given a $\ca{K}$-enriched category $\mathbb{A}$, we call $U_{\ast} \mathbb{A}$ the \emph{underlying $\Ab$-enriched category} of $\mathbb{A}$.
 
 Given $X, Y \in \ca{B}$, we write $\mathbb{B}(X,Y)$ for $\ca{B}\bigl(K(-)\otimes X,Y\bigr) \colon \ca{P}^{\op} \rightarrow \Ab$. From the definition of the Day convolution tensor products, we get unique morphisms $\circ$ making the diagram
 \[
 \xymatrix@C=-8pt{ \ca{B}(Kp \cdot Y, Z)  \ca{B}(Kq \cdot X,Y)  \ca{P}(e,p \cdot q) \ar[r] \ar[d]_{\id \otimes Kp \cdot - \otimes (- \cdot X) \circ K } & \mathbb{B}(Y,Z) \dt \mathbb{B}(X,Y)(e) \ar@{-->}[d]^{\circ} \\
 \ca{B}(Kp \cdot Y,Z) \ca{B}(Kp \cdot  Kq \cdot X, Kp \cdot  Y) \ca{B}(Ke \cdot X, Kp \cdot Kq \cdot X) \ar[r]_-{\circ} & \mathbb{B}(X,Z)(e) }
 \]
 commutative and morphisms $\id_X \colon \ca{P}(-,\U) \rightarrow \mathbb{B}(X,X)$ corresponding by Yoneda to the isomorphism $K\U \otimes X \cong X$.
 
 \begin{prop}\label{prop:enrichment_from_K}
  The objects $\mathbb{B}(X,Y) \in \ca{K}$ with composition and identities as above define a $\ca{K}$-enriched category $\mathbb{B}$ whose underlying $\Ab$-enriched category is (isomorphic to) $\ca{B}$.
 \end{prop}
 
 \begin{proof}
  The relevant associativity and unit laws follow from the corresponding axioms for the category $\ca{B}$ together with basic properties of coends. The isomorphism between $U_{\ast} \mathbb{B}$ and $\ca{B}$ is induced by the natural isomorphism $K\U \otimes X \cong X$.
 \end{proof}

 \begin{lemma}\label{lemma:copowers}
  The $\ca{K}$-category $\mathbb{B}$ has copowers $\ca{P}(-,p) \odot X$ for all $p \in \ca{P}$, $X \in \ca{B}$, given by $\ca{P}(-,p) \odot X=Kp \otimes X$.
 \end{lemma}
 
 \begin{proof}
 The identity in $\ca{B}(Kp \otimes X, Kp \otimes X)$ corresponds to $\ca{P}(-,p) \rightarrow \mathbb{B}(X,Kp \otimes X)$ by the $\Ab$-enriched Yoneda lemma. By the weak Yoneda lemma for $\ca{K}$, this induces a $\ca{K}$-natural transformation $\alpha_Y \colon \mathbb{B}(Kp \otimes X,Y) \Rightarrow [\ca{P}(-,p),\mathbb{B}(X,Y)]_{\mathrm{Day}}$ (that is, the $\alpha_Y$ are $\ca{K}$-natural in $Y$).
 
 From the definition of composition in $\mathbb{B}$ in terms of composition in $\ca{B}$ it follows that composing with $\alpha_Y$ induces the composite
\begin{align*}
 \ca{K} \bigl( \ca{P}(-,q), \mathbb{B}(Kp \otimes X,Y) \bigr)  & \cong \ca{B}(Kq \otimes Kp \otimes X,Y) \\ 
 &\cong   \ca{K}\bigl(\ca{P}(-,q\otimes p),\ca{B}(X,Y)\bigr) \\
 &\cong \ca{K}\bigl(\ca{P}(-,q),[\ca{P}(-,p),\mathbb{B}(X,Y)]_{\mathrm{Day}}\bigr)
\end{align*}
 of natural isomorphisms. Thus $\alpha$ is a $\ca{K}$-natural isomorphism exhibiting $Kp\otimes X$ as the desired copower.
 \end{proof}
 
\begin{prop}\label{prop:K-category_symmetric_monoidal}
 The tensor product of $\ca{B}$ lifts to a $\ca{K}$-functor 
 \[
 -\otimes - \colon \mathbb{B} \dt \mathbb{B} \rightarrow \mathbb{B}
 \]
 making $\mathbb{B}$ a symmetric monoidal $\ca{K}$-category.
\end{prop}

\begin{proof}
 Given $Z \in \ca{B}$, we define the $\ca{K}$-functors $-\otimes Z$ (respectively $Z \otimes -$) via the function
 \[
 \ca{B}(Kp \otimes X,Y ) \rightarrow \ca{B}\bigl(Kp \otimes (X \otimes Z),Y\otimes Z\bigr)
 \]
 which sends $\beta$ to $\beta \otimes Z \circ \alpha_{Kp,X,Z}$ (respectively the evident function obtained by using the symmetry in the case of $Z \otimes -$). The relevant diagram (see \cite[Diagram~(1.21)]{KELLY_BASIC}) commutes since for all $\beta \colon Kp \cdot X \rightarrow X^{\prime}$ and all $\gamma \colon Kq \cdot Y \rightarrow Y^{\prime}$ the diagram
 \[
 \xymatrix{Kp \cdot Kq \cdot X \cdot Y \ar[d]_{\cong} \ar[r]^-{\cong} & (Kp \cdot X) \cdot (Kq \cdot Y) \ar[rr]^-{\id \otimes \gamma} && Kp \cdot X \cdot Y^{\prime} \ar[d]^{\beta \otimes Y^{\prime}} \\  
 Kq \cdot Kp \cdot X \cdot Y \ar[r]_{\id \otimes \beta \otimes \id}  & Kq \cdot X^{\prime} \cdot Y \ar[r]_-{\cong} & X^{\prime } \cdot Kq \cdot Y \ar[r]_-{\id \otimes \gamma} & X^{\prime} \cdot Y^{\prime}}
 \]
in $\ca{B}$ is commutative. 

 It follows from the coherence theorem for symmetric monoidal categories that the associator, the symmetry, and the unit constraints in $\ca{B}$ are $\ca{K}$-natural for this tensor product, so $\mathbb{B}$ is a symmetric monoidal $\ca{K}$-category.
\end{proof}

 Our next goal is to understand the precise relationship between additive functors $\ca{B} \rightarrow \Ab$ and $\ca{K}$-functors $\mathbb{B} \rightarrow \mathbb{K}$, where $\mathbb{K}$ stands for the self-enrichment of $\ca{K}$ given by $[-,-]_{\mathrm{Day}}$. We are particularly interested in functors which are compatible with the symmetric monoidal structure.
 
 \begin{dfn}\label{dfn:K-functor_from_Ab-functor}
 Let $F \colon \ca{B} \rightarrow \Ab$ be an additive functor. We write $\bar{F} \colon \ca{B} \rightarrow \ca{K}$ for the additive functor which sends $X$ to $F\bigl(K(-)^{\vee} \otimes X \bigr)$. For $X,Y \in \ca{B}$, we write
 \[
 \mathbb{F}_{X,Y} \colon \mathbb{B}(X,Y) \rightarrow [\bar{F}X,\bar{F}Y]_{\mathrm{Day}}
 \]
 for the morphism in $\ca{K}$ corresponding by adjunction to the morphism out of the Day convolution tensor product given for each $e,p,q \in \ca{P}$ by the morphisms
 \[
 \ca{B}(Kp \otimes X,Y) F(Kq^{\vee} \otimes X) \ca{P}(e,p\otimes q) \rightarrow F(Ke^{\vee} \otimes  X)
 \]
 which send $\alpha \otimes x \otimes \varphi$ to $F(K \varphi^{\vee} \otimes \alpha \circ Kq^{\vee} \otimes \mathrm{coev} \otimes X)(x)$.
 \end{dfn}
 
 In order to establish basic properties of the morphisms defined above, we use the coherence theorem of symmetric monoidal categories. This is conveniently encoded in the formalism of string diagrams, see \cite{JOYAL_STREET_TENSOR}. The evaluation and coevaluation of a pair of dual objects is depicted using cups and caps, so that the triangle identities amount to the statement that ``wiggles'' in the diagram can be straightened. We use the convention that tensor products of objects are written from top to bottom and morphisms are written from left to right. For example, given $\alpha \colon X\otimes Y \rightarrow X^{\prime} \otimes Y^{\prime}$ and $\beta \colon Y^{\prime} \otimes Z \rightarrow Z^{\prime}$, we write
 \begin{center}
 \begin{tikzpicture}
 \node[circle, inner sep=0pt, minimum size=24pt] (A) [draw] at (1,1.5) {\small $\alpha$};
 \node[circle, inner sep=0pt, minimum size=24pt] (B) [draw] at (2,0.5) {\small $\beta$};
 \begin{knot}
 \strand[thick] (A) to  (3,1.5);
 \strand[thick] (B) to (3,0.5);
 \strand[thick] (A) to (B);
 \strand[thick] (0,2) to [out=0, in=135] (A);
 \strand[thick] (0,1) to [out=0, in=225] (A);
 \strand[thick] (0,0) to (1,0) to [out=0, in=225] (B);
 \end{knot}
 \node[above right] at (0,0) {$Z$};
  \node[above right] at (0,1) {$Y$};
   \node[above right] at (0,2) {$X$};
   \node[above left] at (3,1.5) {$X^{\prime}$};
   \node[above left] at (3,0.5) {$Z^{\prime}$};
 \end{tikzpicture}
\end{center}  
for the composite $ X^{\prime} \otimes \beta \circ \alpha \otimes Z$. Since we are working in a symmetric monoidal category, we can freely exchange the two types of crossings of strands.

 We also need the following basic fact about coends.
 
 \begin{lemma}\label{lemma:coends}
  Given objects $f,\ell,p,q \in \ca{P}$, composition in $\ca{P}$ induces isomorphisms
  \[
  \int^{e\in \ca{P}} \ca{P}(e,p\otimes q) \otimes \ca{P}(\ell, e\otimes f) \cong \ca{P}(\ell,p\otimes q \otimes f)
  \]
  and
  \[
  \int^{g\in \ca{P}} \ca{P}(g,q\otimes f) \otimes \ca{P}(\ell, p\otimes g) \cong \ca{P}(\ell,p\otimes q \otimes f)
  \]
  of abelian groups.
 \end{lemma}
 
 \begin{proof}
  Both isomorphisms are immediate consequences of the Yoneda isomorphisms in the form of \cite[Formula~(3.71)]{KELLY_BASIC}.
 \end{proof}
 
 \begin{lemma}\label{lemma:K-functor_from_Ab-functor}
  In the situation of Definition~\ref{dfn:K-functor_from_Ab-functor}, the morphisms $\mathbb{F}_{X,Y}$ define a $\ca{K}$-functor $\mathbb{F} \colon \mathbb{B} \rightarrow \mathbb{K}$ with underlying additive functor given by $\bar{F} \colon \ca{B} \rightarrow \ca{K}$. If $\alpha \colon F \Rightarrow G$ is a natural transformation, then the components
  \[
  \bar{\alpha}_X \defl \alpha_{K(-)^{\vee}\otimes X} \colon F\bigl(K(-)^{\vee} \otimes X \bigr) \rightarrow G\bigl(K(-)^{\vee} \otimes X \bigr)
  \]
  define a $\ca{K}$-natural transformation $\bar{\alpha} \colon \mathbb{F} \Rightarrow \mathbb{G}$.
 \end{lemma}
 
 \begin{proof}
  Commutativity of the square
  \[
   \vcenter{
  \xymatrix{\mathbb{B}(Y,Z) \dt \mathbb{B}(X,Y) \ar[r]^-{\circ} \ar[d]_{\mathbb{F} \otimes \mathbb{F}} & \mathbb{B}(X,Z) \ar[d]^{\mathbb{F}} \\ [\bar{F}Y,\bar{F}Z] \dt [\bar{F} X, \bar{F} Y] \ar[r]_-{\circ} & [\bar{F} X, \bar{F} Z] }
  }
  \]
  corresponds by adjunction to commutativity of the square
 \begin{equation}\label{eqn:K-functoriality}
\vcenter{
\xymatrix@C=40pt{ \mathbb{B}(Y,Z) \dt \mathbb{B}(X,Y) \dt \bar{F} X \ar[r]^-{\circ  \otimes \bar{F}X } \ar[d]_{\mathrm{act}} & \mathbb{B}(X,Z)\dt \bar{F}X \ar[d]^{\mathrm{act}} \\ \mathbb{B}(Y,Z) \dt \bar{F}Y \ar[r]_-{\mathrm{act}} & \bar{F}Z }
}
 \end{equation}
 where $\mathrm{act}$ denotes the action of Definition~\ref{dfn:K-functor_from_Ab-functor}.
 
 The top composite of \eqref{eqn:K-functoriality} restricts to the morphism
 \[
 \ca{B}(Kp \cdot Y, Z) \ca{B}(Kq \cdot X, Y) \bar{F}(Kf^{\vee} \cdot X) \ca{P}(g,q\otimes f) \ca{P}(\ell, p \otimes g) \rightarrow F(K\ell^{\vee} \cdot Z)
 \]
 which sends $\beta \otimes \alpha \otimes x \otimes \tau \otimes \sigma$ to the image of $x$ under $F$ applied to the morphism on the left below
 \begin{equation}\label{eqn:defining_action_in_B}
 \vcenter{\hbox{
\begin{tikzpicture}[baseline=-5ex]
\node [above right] (A) at (-0.5,1.5) {$K f^{\vee}$};
\node [above right] (B) at (-0.5,0) {$X$};
\node[circle, inner sep=0pt, minimum size=24pt] (C) [draw] at (1,1.5) {\small $K \tau^{\vee}$};
\node[circle, inner sep=0pt, minimum size=24pt] (D) [draw] at (1,0) {\small $\alpha$};
\node[circle, inner sep=0pt, minimum size=24pt] (E) [draw] at (3,1.5) {\small $K\sigma^{\vee}$};
\node[circle, inner sep=0pt, minimum size=24pt] (F) [draw] at (3,0) {\small $\beta$};
\node[above left] (G) at (4.5,1.5) {$K \ell^{\vee}$};
\node[above left] (H) at (4.5,0) {$Z$};
\begin{knot}[clip width=5, clip radius=8pt]
\strand[thick] (-0.5,1.5) to (C);
\strand[thick] (C) to (E);
\strand[thick] (E) to (4.5,1.5);
\strand[thick] (-0.5,0) to (D);
\strand[thick] (D) to (F);
\strand[thick] (F) to (4.5,0);
\strand[thick] (D) to [out=135, in=270] (0,0.75) to [out=90,in=225] (C);
\strand[thick] (F) to [out=135, in=270] (2,0.75) to [out=90,in=225] (E);
\end{knot}
\end{tikzpicture}
\qquad
\begin{tikzpicture}
 \node[circle, inner sep=0pt, minimum size=24pt] (A) [draw] at (1,0) {\small $ K \varphi$};
 \node[circle, inner sep=0pt, minimum size=24pt] (B) [draw] at (1,1.5) {\small $K \psi^{\vee}$};
 \node[circle, inner sep=0pt, minimum size=24pt] (C) [draw] at (2,-0.75) {\small $\alpha$};
 \node[circle, inner sep=0pt, minimum size=24pt] (D) [draw] at (3,0) {\small $\beta$};
 \begin{knot} [clip width=5, clip radius=8pt]
 \strand[thick] (-0.5,-0.75) to (C);
  \strand[thick] (A) to (C);
 \strand[thick] (C) to (D);
 \strand[thick] (A) to [ out=45, in=180] (2,0.6) to [out=0, in=135] (D);
 \strand[thick] (A) to [out=135, in=270] (0,0.75) to [out=90,in=225] (B);
 \strand[thick] (-0.5,1.5) to (B);
 \strand[thick] (B) to (4,1.5);
 \strand[thick] (D) to (4,0);
 \node [above left] at (4,0) {$Z$};
 \node [above left] at (4,1.5) {$K \ell^{\vee}$};
 \node [above right] at (-0.5,-0.75) {$X$};
 \node [above right] at (-0.5,1.5) {$Kf^{\vee}$};
 \node[above] at (2,0.6) {$Kp$};
 \end{knot}
\end{tikzpicture}
 }}
 \end{equation}
 in $\ca{B}$.
 
  The bottom composite of \eqref{eqn:K-functoriality} restricts to the morphism
  \[
 \ca{B}(Kp \cdot Y, Z) \ca{B}(Kq \cdot X, Y) \ca{P}(e, p\otimes q) \bar{F}(Kf^{\vee} \cdot X) \ca{P}(\ell, e \otimes f) \rightarrow F(K\ell^{\vee} \cdot Z)
 \]
 which sends $\beta \otimes \alpha \otimes \varphi \otimes x \otimes \psi$ to the image of $x$ under $F$ applied to the morphism given by the string diagram on the right in \eqref{eqn:defining_action_in_B}. Using the fact that $K$ is symmetric strong monoidal and Lemma~\ref{lemma:coends}, one checks that the two morphisms in \eqref{eqn:defining_action_in_B} induce the same morphism out of the coend, so \eqref{eqn:K-functoriality} commutes. Compatibility of $\mathbb{F}$ with identities follows more directly, so $\mathbb{F}$ is indeed a $\ca{K}$-functor. That the underlying additive functor of $\mathbb{F}$ is $\bar{F}$ follows directly from the definition.
 
 The $\ca{K}$-naturality square
 \[
 \xymatrix{\mathbb{B}(X,Y) \ar[r]^{\mathbb{F}_{X,Y}} \ar[d]_{\mathbb{G}_{X,Y}} & [\bar{F}X,\bar{F} Y] \ar[d]^{(\bar{\alpha}_Y)_{\ast}} \\ 
 [\bar{G}X, \bar{G}Y] \ar[r]_-{(\bar{\alpha}_X)^{\ast}} & [\bar{F}X, \bar{G}Y] }
 \]
 corresponds by adjunction to a square which commutes by (unenriched) naturality of $\alpha$.  
 \end{proof}
 
 \begin{dfn}\label{dfn:lax_monoidal_structure}
 Let $(F,\varphi^F,\varphi^F_0) \colon \ca{B} \rightarrow \Ab$ be a symmetric lax monoidal functor. Given objects $X,Y \in \ca{B}$,  we write $\bar{\varphi}^F_{X,Y} \colon \bar{F}X \dt \bar{F}Y \rightarrow \bar{F}(X\otimes Y)$ for the morphism in $\ca{K}$ induced by the homomorphism
 \[
 F(Kp^{\vee} \cdot X) F(Kq^{\vee} \cdot Y) \ca{P}(e,p\otimes q) \rightarrow F(Ke^{\vee} \cdot X \cdot Y)
 \]
 of abelian groups sending $x\otimes y \otimes \varphi$ to the image of $\varphi^F_{X,Y}(x\otimes y) \in F(Kp^{\vee} \cdot X \cdot Kq^{\vee} \cdot Y)$ under the homomorphism given by applying $F$ to the morphism
  \begin{center} 
 \begin{tikzpicture}
  \node[circle, inner sep=0pt, minimum size=24pt] (A) [draw] at (3,2.5) {\small $K\varphi^{\vee}$};
  \begin{knot}[clip width=8, clip radius=8pt]
  \strand[thick] (-0.5,3) to (0,3) to [out=0,in=180] (2,2) to [out=0,in=225] (A);
  \strand[thick] (-0.5,2) to (0,2) to [out=0, in=180] (2,1) to [out=0, in=180] (4.5,1);
  \strand[thick] (-0.5,1) to (0,1) to [out=0, in=180] (2,3) to [out=0, in=135] (A);
  \strand[thick] (A) to (4.5,2.5);
  \strand[thick] (-0.5,0) to (0,0) to (4.5,0);
  \flipcrossings{1,2}
  \node[above right] at (-0.5,0) {$Y$};
  \node[above right] at (-0.5,1) {$Kq^{\vee}$};
  \node[above right] at (-0.5,2) {$X$};
  \node[above right] at (-0.5,3) {$Kp^{\vee}$};
  \node[above left] at (4.5,2.5) {$K\ell^{\vee}$};
  \node[above left] at (4.5,1) {$X$};
  \node[above left] at (4.5,0) {$Y$};
  \end{knot}
 \end{tikzpicture}
\end{center}
 in $\ca{B}$. We write $\bar{\varphi}_0^F \colon \ca{P}(-,\U) \rightarrow \bar{F} \U$ for the morphism which corresponds to
 \[
 \xymatrix{ \mathbb{Z} \ar[r]^-{\varphi_0^F} & F\U \ar[r]^-{F(\cong)} & F(K\U^{\vee} \otimes \U) }
 \]
 by Yoneda.
 \end{dfn}
 
 We write $\slm_{\Ab}(\ca{B},\Ab)$ for the category of symmetric lax monoidal additive functors and $\slm_{\ca{K}}(\mathbb{B},\mathbb{K})$ for the category of symmetric lax monoidal $\ca{K}$-functors.
 
 \begin{prop}\label{prop:symmetric_lax_monoidal_K-functor}
  The assignment which sends an object  
  \[
  (F,\varphi^F,\varphi^F_0) \in \slm_{\Ab}(\ca{B},\Ab)
  \]
 to $(\mathbb{F},\bar{\varphi}^F,\bar{\varphi}^F_0)$ (see Definition~\ref{dfn:lax_monoidal_structure}) and which sends a symmetric monoidal natural transformation $\alpha \colon F \Rightarrow G$ to $\bar{\alpha} \colon \mathbb{F} \Rightarrow \mathbb{G}$ (see Definition~\ref{dfn:K-functor_from_Ab-functor}) gives a well-defined functor $\Phi \colon \slm_{\Ab}(\ca{B},\Ab) \rightarrow \slm_{\ca{K}}(\mathbb{B},\mathbb{K})$.
 \end{prop}
 
 \begin{proof}
  It is clear that the assignment is functorial if it is well-defined, so we need to check that $\bar{\varphi}^{F}$ is $\ca{K}$-natural in both variables, that $(\mathbb{F},\bar{\varphi}^F,\bar{\varphi}^F_0)$ is a symmetric lax monoidal $\ca{K}$-functor, and that $\bar{\alpha}$ is symmetric monoidal if $\alpha$ is symmetric monoidal. Each of these assertions can be checked in a manner similar to the proof of Lemma~\ref{lemma:K-functor_from_Ab-functor}.
  
  For example, $\ca{K}$-naturality of $\bar{\varphi}^F$ in the first variable is equivalent to commutativity of the diagrams
  \[
  \vcenter{
  \xymatrix@C=40pt{
   \mathbb{B}(X,Y) \ar[r]^-{-\dt \bar{F}Z \circ \mathbb{F}} \ar[d]_{\mathbb{F} \circ (-\otimes Z)} & [\bar{F} X \dt \bar{F}Z, \bar{F}Y \dt \bar{F}Z] \ar[d]^{(\bar{\varphi}^{F}_{Y,Z})_{\ast}} \\
   [\bar{F}(X\otimes Z) , \bar{F}(Y \otimes Z)] \ar[r]_-{(\bar{\varphi}^F_{X,Z})^{\ast}} & [\bar{F} X \dt \bar{F}Z,\bar{F}(Y \otimes Z)]
  }}
  \]
  for all $X,Y,Z \in \ca{K}$. By using Lemma~\ref{lemma:coends}, this follows from the fact that, given morphisms $\sigma \colon \ell \rightarrow p \otimes g$, $\tau \colon g \rightarrow q \otimes f$, $\varphi \colon e \rightarrow p \otimes q$, and $\psi \colon \ell \rightarrow e \otimes f$ in $\ca{P}$ such that $\varphi \otimes f \circ \psi=p\otimes \tau \circ \sigma$, the two morphisms
   \begin{center} 
 \begin{tikzpicture}
  \node[circle, inner sep=0pt, minimum size=24pt] (A) [draw] at (3,2.5) {\small $K\tau^{\vee}$};
  \node[circle, inner sep=0pt, minimum size=24pt] (B) [draw] at (4.5, 2.5) {\small $K \sigma^{\vee}$};
  \node[circle, inner sep=0pt, minimum size=24pt] (C) [draw] at (4.5,1) {\small $\alpha$};
  \begin{knot}[clip width=8, clip radius=8pt]
  \strand[thick] (-0.5,3) to (0,3) to [out=0,in=180] (2,2) to [out=0,in=225] (A);
  \strand[thick] (-0.5,2) to (0,2) to [out=0, in=180] (2,1) to [out=0, in=180] (C);
  \strand[thick] (-0.5,1) to (0,1) to [out=0, in=180] (2,3) to [out=0, in=135] (A);
  \strand[thick] (A) to (B);
   \strand[thick] (C) to [out=135, in=270] (3.5,1.75) to [out=90,in=225] (B);
  \strand[thick] (B) to (6,2.5);
  \strand[thick] (C) to (6,1);
  \flipcrossings{1,2}
  \node[above right] at (-0.5,1) {$Kf^{\vee}$};
  \node[above right] at (-0.5,2) {$X$};
  \node[above right] at (-0.5,3) {$Kq^{\vee}$};
  \node[above left] at (6,2.5) {$K\ell^{\vee}$};
  \node[above left] at (6,1) {$Y$};
  \end{knot}
 \end{tikzpicture}
\end{center}
and
\begin{center}
\begin{tikzpicture}
\node[circle, inner sep=0pt, minimum size=24pt] (A) [draw] at (1.5,2.5) {\small $K\varphi^{\vee}$};
\node[circle, inner sep=0pt, minimum size=24pt] (B) [draw] at (1.5,1) {\small $\alpha$};
\node[circle, inner sep=0pt, minimum size=24pt] (C) [draw] at (4.5,1.5) {\small $K\psi^{\vee}$};
\begin{knot}[clip width=8, clip radius=8pt]
\strand[thick] (0,2.5) to (A);
\strand[thick] (0,1) to (B);
\strand[thick] (0,0) to (1.5,0) to [out=0, in=180] (3.5,2) to [out=0, in=135] (C);
\strand[thick] (B) to [out=0, in=180] (3.5,0) to (6,0);
\strand[thick] (A) to [out=0,in=180] (3.5,1) to [out=0,in=225] (C);
\strand[thick] (C) to (6,1.5);
\strand[thick] (B) to [out=135, in=270] (0.5,1.75) to [out=90,in=225] (A);
\end{knot}
\node[above right] at (0,2.5) {$Kq^{\vee}$};
\node[above right] at (0,1) {$X$};
\node[above right] at (0,0) {$Kf^{\vee}$};
\node[above left] at (6,1.5) {$K \ell^{\vee}$};
\node[above left] at (6,0) {$Y$};
\end{tikzpicture}
\end{center}
 in $\ca{B}$ are equal for all $\alpha \colon Kp \otimes X \rightarrow Y$. Commutativity of the square expressing $\ca{K}$-naturality of $\bar{\varphi}^F$ in the second variable is checked by an analogous computation (it is slightly more involved since there are more instances of the symmetry isomorphism in this case).
 
 Each of the axioms asserting compatibility of $\bar{\varphi}^F$ and $\bar{\varphi}^F_0$ with the $\ca{K}$-monoidal structure follows from the corresponding axiom for $\varphi^{F}$ and $\varphi^{F}_{0}$. For the axiom involving the associator (and thus a triple tensor product), one again uses Lemma~\ref{lemma:coends} to show this.
 
 Finally, the two axioms for $\bar{\alpha}$ to be a symmetric monoidal $\ca{K}$-natural transformation reduce to the corresponding axioms for $\alpha$.
 \end{proof}
 
 The base change 2-functor
 \[
 U_{\ast} \colon \ca{K}\mbox{-}\CAT \rightarrow \Ab\mbox{-}\CAT
 \]
 induced by the symmetric lax monoidal functor $U=\mathrm{ev}_{\U} \colon \ca{K} \rightarrow \Ab$ is compatible with the symmetric monoidal structure, so it sends symmetric lax monoidal $\ca{K}$-functors to symmetric lax monoidal additive functors. The composite
 \[
 \xymatrix{ \slm_{\ca{K}} (\mathbb{B}, \mathbb{K}) \ar@{-->}[rd]_{\Psi} \ar[r]^-{(U_{\ast})_{\mathbb{B},\mathbb{K}}} & \slm_{\Ab}(\ca{B},\ca{K}) \ar[d]^{\slm_{\Ab}(\ca{B},U)} \\  & \slm_{\Ab}(\ca{B},\Ab)}
 \]
 defines a functor going in the opposite direction of the functor $\Phi$ of Proposition~\ref{prop:symmetric_lax_monoidal_K-functor}.
 
 \begin{thm}\label{thm:K-functor_Ab-functor_equivalence}
  The functor $\Psi$ above and the functor
  \[
  \Phi \colon \slm_{\Ab}(\ca{B},\Ab) \rightarrow \slm_{\ca{K}}(\mathbb{B},\mathbb{K})
  \]
  of Proposition~\ref{prop:symmetric_lax_monoidal_K-functor} are mutually inverse equivalences.
 \end{thm}
 
 \begin{proof}
 Note that $\Psi \Phi(F)$ sends $X$ to $F(K\U^{\vee} \otimes X)$. By applying the natural isomorphism $K\U^{\vee} \otimes X \cong X$, we get an isomorphism $\Psi \Phi (F) \cong F$ which is easily seen to be symmetric monoidal and natural in $F$. This shows in particular that $\Psi$ is essentially surjective on objects. Using a 2-by-3 argument, it thus suffices to check that $\Psi$ is full and faithful.
 
 To see this, it is convenient to use the biequivalence between $\ca{P}$-copowered $\ca{K}$-categories and categories with a $\ca{P}$-action (see \cite[Theorem~3.4]{GORDON_POWER}). Namely, the $\ca{K}$-category $\mathbb{B}$ has copowers $\ca{P}(-,p) \odot X$ given by $Kp \otimes X$, while $\mathbb{K}$ has copowers given by
 \[
 \ca{P}(-,p) \dt M \cong [\ca{P}(-,p^{\vee}),M]_{\mathrm{Day}}
 \]
 for all $M \in \ca{K}$. The natural isomorphisms
 \begin{align*}
  [\ca{P}(-,p^{\vee}),M](q) & \cong \ca{K}\bigl( \ca{P}(-,q), [\ca{P}(-,p^{\vee}),M] \bigr) \\
 & \cong \ca{K}\bigl( \ca{P}(-,q\otimes p^{\vee}),M\bigr) \\
  & \cong M(q\otimes p^{\vee})  
 \end{align*}
 show that the $\ca{P}$-action on $\ca{K}$ is given by $p \cdot M=M(-\otimes p^{\vee})$. We use the same notation to denote the $\ca{P}$-action on $\ca{B}$, that is, the notation $p \cdot X$ is a shorthand for $Kp \otimes X \in \ca{B}$.
 
 Since every object in $\ca{P}$ has a dual, these copowers are \emph{absolute} weighted colimits (see \cite{STREET_ABSOLUTE}), so for each $\ca{K}$-functor $\mathbb{G} \colon \mathbb{B} \rightarrow \mathbb{K}$, the colimit comparison morphism
 \[
 \bar{\mathbb{G}}_{p,X} \colon p \cdot \mathbb{G} X \rightarrow \mathbb{G}(p\cdot X)
 \]
 is an isomorphism for all $p \in \ca{P}$ and all $X \in \ca{B}$.
 
 Using this, it is now straightforward to show that $\Psi$ is faithful. Indeed, suppose $\mathbb{F}, \mathbb{G} \colon \mathbb{B} \rightarrow \mathbb{K}$ are $\ca{K}$-functors and $\alpha,\beta \colon \mathbb{F} \Rightarrow \mathbb{G}$ are $\ca{K}$-natural transformations such that
 \[
 (\alpha_X)_{\U}=(\beta_X)_{\U} \colon \mathbb{F}X( \U ) \rightarrow \mathbb{G}X(\U)
 \]
 for all $X \in \ca{B}$. Then we have in particular $ (\alpha_{p \cdot X})_{\U}=(\beta_{p \cdot X})_{\U} $ for all $p \in \ca{P}$, so the commutative diagram
 \[
 \xymatrix{ \mathbb{F}(p \cdot X)(\U) \ar[r]^-{\bar{\mathbb{F}}} \ar[d]_{(\alpha_{p \cdot X})_{\U}} & (p \cdot \mathbb{F}X)(\U) \ar[r]^{\cong} \ar[d]^{(p \cdot \alpha_X)_{\U}} & \mathbb{F}X(p^{\vee}) \ar[d]^{(\alpha_X)_{p^{\vee}}} \\
  \mathbb{G}(p \cdot X)(\U) \ar[r]_-{\bar{\mathbb{G}}} & (p \cdot \mathbb{G}X)(\U) \ar[r]_{\cong} & \mathbb{G}X(p^{\vee})
   }
 \]
 and the analogous diagram for $\beta$ show that $\alpha_X=\beta_X$ for all $X \in \ca{B}$, hence that $\Psi$ is faithful.
 
 It remains to check that $\Psi$ is full. Thus let $(\mathbb{F},\varphi^{\mathbb{F}},\varphi^{\mathbb{F}}_0)$ and $(\mathbb{G},\varphi^{\mathbb{G}},\varphi^{\mathbb{G}}_0)$ be symmetric lax monoidal $\ca{K}$-functors $\mathbb{B} \rightarrow \mathbb{K}$ and let $\beta \colon \Psi \mathbb{F} \Rightarrow \Psi \mathbb{G}$ be a symmetric monoidal natural transformation. We need to show that there exists a symmetric monoidal $\ca{K}$-natural transformation $\gamma \colon \mathbb{F} \Rightarrow  \mathbb{G}$ such that $\Psi \gamma=\beta$. For $X \in \ca{B}$, we define $\gamma_X$ by the collection $(\gamma_X)_p$ of homomorphisms making the diagram
 \[
 \xymatrix{ \mathbb{F}(p^{\vee} \cdot X )(\U) \ar[r]^-{\bar{\mathbb{F}}} \ar[d]_{\beta_{Kp^{\vee}\otimes X }} & \mathbb{F}X(\U \otimes p^{\vee\vee} ) \ar[r]^-{\cong} & \mathbb{F}X(p) \ar@{-->}[d]_{\exists ! }^{(\gamma_X)_p}  \\
 \mathbb{G}(p^{\vee} \cdot X)(\U) \ar[r]_-{\bar{\mathbb{G}}} & \mathbb{G}X(\U \otimes p^{\vee\vee}) \ar[r]_-{\cong} & \mathbb{G}X(p)
  }
 \]
 commutative. From this definition it is straightforward to check that $\gamma$ is compatible with the $\ca{P}$-actions, hence that it defines a $\ca{K}$-natural transformation by \cite[Theorem~3.4]{GORDON_POWER}. Moreover, it is clear that $\Psi \gamma =\beta$, so we only need to check that $\gamma$ is symmetric monoidal. Compatibility with $\varphi^F_0$ and $\varphi^G_0$ follows directly from the corresponding fact for $\beta$.
 
 We have reduced the problem to showing that the diagram
 \[
 \xymatrix{ \mathbb{F} X \dt \mathbb{F}Y \ar[d]_{\gamma_X \dt \gamma_Y} \ar[r]^-{\varphi^{\mathbb{F}}_{X,Y} } & \mathbb{F}(X \otimes Y)  \ar[d]^{\gamma_{X \otimes Y}} \\ 
\mathbb{G}X \dt \mathbb{G}Y \ar[r]_-{\varphi^{\mathbb{G}}_{X,Y}} & \mathbb{G}(X \otimes Y) 
  }
 \]
 is commutative for all $X,Y \in \ca{B}$. By precomposing with the universal morphism
 \[
 \mathbb{F}X(p) \otimes \mathbb{F}Y(q) \otimes \ca{P}(e,p\otimes q) \rightarrow \mathbb{F}X \dt \mathbb{F}Y(e)
 \]
 and a Yoneda argument in $e \in \ca{P}$, this reduces to commutativity of the diagram
 \begin{equation}\label{eqn:monoidal}
\vcenter{
 \xymatrix{ p \cdot \mathbb{F}X (\U) \otimes q \cdot \mathbb{F}Y(\U) \ar[d]_{p \cdot (\gamma_X)_{\U} \otimes q \cdot (\gamma_Y)_{\U}} \ar[r] & p \cdot \bigl( q\cdot \mathbb{F}(X\otimes Y)(\U) \bigr) \ar[d]^{p \cdot \bigl(q \cdot (\gamma_{X\otimes Y})_{\U} \bigr)} \\
p \cdot \mathbb{G}X (\U) \otimes q \cdot \mathbb{F}Y(\U) \ar[r] & p \cdot \bigl( q\cdot \mathbb{G}(X\otimes Y)(\U) \bigr) 
 }}
 \end{equation}
 for all $p,q \in \ca{P}$.
 
 Since $\varphi^{\mathbb{F}}$ is $\ca{K}$-natural in both variables, the diagram
 \[
 \xymatrix{ p \cdot ( q \cdot \mathbb{F} X \dt \mathbb{F}Y) \ar[r] \ar[d]_{p \cdot (q \cdot \varphi^{\mathbb{F}}_{X,Y} )} & p \cdot \bigl(\mathbb{F}X \dt \mathbb{F} (q \cdot Y)\bigr) \ar[r] \ar[d]^{p \cdot \varphi^{\mathbb{F}}_{X,q\cdot Y}} & \mathbb{F} (p\cdot X) \dt \mathbb{F} (q \cdot Y) \ar[d]^{\varphi^{\mathbb{F}}_{p \cdot X, q\cdot Y}} \\
    p \cdot ( q \cdot \mathbb{G} X \dt \mathbb{G}Y) \ar[r]  & p \cdot \bigl(\mathbb{G}X \dt \mathbb{G} (q \cdot Y)\bigr) \ar[r] & \mathbb{G} (p\cdot X) \dt \mathbb{G} (q \cdot Y)
  }
 \]
 (whose unlabelled arrows are colimit comparison morphisms) is commutative. Combining this with the analogous diagram for $\varphi^{\mathbb{G}}$ and the definition of the symmetric lax monoidal structure of $\Psi \mathbb{F}$ and $\Psi \mathbb{G}$ on can check that \eqref{eqn:monoidal} is indeed commutative.
 \end{proof}
 
 A second application of the theory of absolute weighted colimits gives the following proposition. Recall that a lax monoidal functor $F$ is called \emph{normal} if $\varphi_0 \colon \U \rightarrow F\U$ is an isomorphism.
 
 \begin{prop}\label{prop:normal_lax_implies_iso}
 Let $\mathbb{B}^{\prime}$ be a monoidal $\ca{K}$-category and $(\mathbb{F},\varphi,\varphi_0 )\colon \mathbb{B} \rightarrow \mathbb{B}^{\prime}$ a normal lax monoidal functor. Then the morphisms
 \[
\varphi_{Kp,X} \colon \mathbb{F} Kp \otimes \mathbb{F}X \rightarrow \mathbb{F}(Kp \otimes X)
 \]
 and
 \[
 \varphi_{X,Kp} \colon \mathbb{F} X \otimes \mathbb{F} Kp \rightarrow \mathbb{F}(X \otimes Kp)
 \]
 are isomorphisms for all $p \in \ca{P}$ and $X \in \mathbb{B}$.
 \end{prop}
 
 \begin{proof}
 Note that $Kp \cong Kp \otimes \U \cong \ca{P}(-,p) \odot \U$ by Lemma~\ref{lemma:copowers}. Since $\ca{P}(-,p)$ has a dual in $\ca{K}$, the colimit $\ca{P}(-,p) \odot \U$ is absolute (see \cite{STREET_ABSOLUTE}). Thus all the vertical morphisms in the diagram
 \[
 \xymatrix@C=60pt{
 \ca{P}(-,p) \odot ( \mathbb{F} \U \otimes \mathbb{F}X ) \ar[d]_{\overline{-\otimes \mathbb{F}X }} \ar[r]^-{\ca{P}(-,p) \odot \varphi_{\U,X}} & \ca{P}(-,p) \odot \mathbb{F}(\U \otimes X) \ar[d]^{\overline{\mathbb{F}}}  \\
 \bigl( \ca{P}(-,p) \odot \mathbb{F} \U \bigr) \otimes \mathbb{F}X   \ar[d]_{- \otimes \mathbb{F}X \circ \overline{\mathbb{F}} } &  \mathbb{F}\bigl(\ca{P}(-,p) \odot (\U \otimes X)\bigr) \ar[d]^{\mathbb{F} \circ \overline{-\otimes X}}  \\ 
 \mathbb{F}\bigl(\ca{P}(-,p)\odot \U \bigr) \otimes \mathbb{F}X  \ar[r]_-{\varphi_{\ca{P}(-,p)\odot \U,X }} & \mathbb{F}\Bigl(\bigl( \ca{P}(-,p) \odot \U \bigr) \otimes X \Bigr) }
 \]
 are isomorphisms.
 
 From the normality of $\mathbb{F}$ it follows that $\varphi_{\U,X}$ is an isomorphism as well, hence that $\varphi_{Kp,X}$ is an isomorphism for all $p \in \ca{P}$. The claim about $\varphi_{X,Kp}$ follows analogously.
 \end{proof}
 
 For any symmetric lax monoidal $\ca{K}$-functor $\mathbb{F} \colon \mathbb{B} \rightarrow \mathbb{K}$, we get a commutative algebra $B \defl \mathbb{F} \U \in \mathbb{K}$ and a canonical lift
 \[
 \mathbb{F}^{\mathrm{norm}} \colon \mathbb{B} \rightarrow \mathbb{K}_{B}
  \]
 to the symmetric monoidal $\ca{K}$-category of $B$-modules in $\mathbb{K}$. By construction, $\mathbb{F}^{norm}$ is \emph{normal} lax monoidal. We call $B$ the \emph{coefficient ring} of $\mathbb{F}$. The following corollary summarizes the main results of this section.
 
 \begin{cor}\label{cor:lax_implies_some_isos}
 Let $\ca{B}, \ca{P}$ be $\Ab$-enriched symmetric monoidal categories such that every $p \in \ca{P}$ has a dual and let $K \colon \ca{P} \rightarrow \ca{B}$ be a symmetric strong monoidal functor. Let $\ca{K} \defl [\ca{P}^{\op},\Ab]$ be the category of additive presheaves on $\ca{P}$ endowed with the Day convolution symmetric monoidal structure.
 
 Let $(F,\varphi,\varphi_0) \colon \ca{B} \rightarrow \Ab$ be a symmetric lax monoidal functor and let
 \[
 (\bar{F},\bar{\varphi},\bar{\varphi}_0) \colon \ca{B} \rightarrow \ca{K}
 \]
 be the induced symmetric lax monoidal functor sending $X$ to $F\bigl(K(-)^{\vee} \otimes X\bigr)$ (see Definitions~\ref{dfn:K-functor_from_Ab-functor} and \ref{dfn:lax_monoidal_structure}). Let $B =\bar{F} \U$ be the coefficient ring of $\bar{F}$ and write $\bar{F}^{\mathrm{norm}} \colon \ca{B} \rightarrow \ca{K}_{B}$ for the canonical lift to $B$-modules. Then the morphisms
 \[
\bar{\varphi}_{Kp,X} \colon \bar{F} Kp \ten{B} \bar{F}X \rightarrow \bar{F}(Kp \otimes X)
 \]
 and
 \[
\bar{\varphi}_{X,Kp} \colon \bar{F} X \ten{B} \bar{F} Kp \rightarrow \bar{F}(X \otimes Kp)
 \]
 are isomorphisms for all $p \in \ca{P}$ and all $X \in \ca{B}$.
 \end{cor}
 
 \begin{proof}
 By Proposition~\ref{prop:symmetric_lax_monoidal_K-functor}, $\bar{F}^{\mathrm{norm}}$ is the underlying additive functor of the normal lax monoidal $\ca{K}$-functor $\mathbb{F}^{\mathrm{norm}} \colon \mathbb{B} \rightarrow \mathbb{K}_{B}$. The claim thus follows directly from Proposition~\ref{prop:normal_lax_implies_iso}.
 \end{proof}
 
 \subsection{The additive site of a homology theory}\label{section:sites} 
 
 Throughout this section, we fix a $\otimes$-triangulated category $\ca{T}$, that is, a triangulated category with an (additive) symmetric monoidal structure such that $X \otimes -$ is triangulated for all $X \in \ca{T}$. We also assume that $\ca{P}$ is a rigid symmetric monoidal category and $K \colon \ca{P} \rightarrow \ca{T}$ is symmetric strong monoidal functor. We will often be interested in the case where $\ca{P}$ is a Picard category, in which case all its objects are invertible. For example, $\ca{T}$ could be the stable homotopy category and $\ca{P}$ the Picard category generated by a single \emph{odd} invertible object $\ell$ (meaning that the symmetry $\ell \otimes \ell \rightarrow \ell \otimes \ell$ is equal to $-\id$) and $K$ the functor sending $\ell$ to $S^1$. In this case, the coefficient category $\ca{K}=[\ca{P}^{\op},\Ab]$ is the category of $\mathbb{Z}$-graded abelian groups, with symmetry given by the Koszul sign convention.
  
 We will further assume that $\ca{T}$ is symmetric monoidal \emph{closed}, and that the triangulation is compatible with the symmetric monoidal closed structure. In other words, the axioms (TC1) and (TC2) of \cite{MAY} hold or, equivalently, the axioms of \cite[A.2]{HOVEY_PALMIERI_STRICKLAND} hold (see \cite[Remark~4.2]{MAY}). The important fact we need is that the internal hom preserves distinguished triangles in the second variable and it preserves distinguished triangles up to sign in the first variable. We call such a category a \emph{closed $\otimes$-triangulated} category. Furthermore, we make the following standing assumption about the compatibility of $K$ and the triangulated structure:
 \begin{equation}\label{eqn:s1_in_image}
  \text{There exists an object } \ell \in \ca{P} \; \text{with} \; K\ell \cong S^{1}
 \end{equation}
 (in other words, the object $S^{1}=\U[1]$ lies in the image of $K$). We write $\cell(\ca{P})$ for the smallest thick subcategory of $\ca{T}$ containing  the image $\{Kp \in \ca{T} \; \vert \; p \in \ca{P} \}$ of $K$. Thus $\cell(\ca{P})$ is the closure of $\{Kp \in \ca{T} \; \vert \; p \in \ca{P} \}$ under shifts, retracts, and cofibers (mapping cones).
 
 \begin{lemma}\label{lemma:cell_P_rigid}
  The category $\cell(\ca{P})$ is closed under finite direct sums, finite tensor products, and duals.
 \end{lemma}
 
 \begin{proof}
  Any thick subcategory is closed under finite direct sums. To see that $\cell(\ca{P})$ is closed under finite tensor products, note that $\U \in \cell(\ca{P})$ since the functor $K \colon \ca{P} \rightarrow \ca{T}$ is strong monoidal. For each $p \in \ca{P}$, the category
  \[
  \ca{B} \defl \{ X \in \ca{T} \;\vert\; X \otimes Kp \in \cell(\ca{P}) \}
  \]
  is thick and contains $Kq$ for all $q \in \ca{P}$. Thus $A \otimes Kp \in \cell(\ca{P})$ for all $A \in \cell(\ca{P})$. The same argument applied to $\ca{B}_A \defl \{X \in \ca{T} \;\vert \; A \otimes X \in \cell(\ca{P}) \}$ shows that $\cell(\ca{P})$ is closed under binary tensor products.
  
  To see that $\cell(\ca{P})$ is closed under duals, consider the full subcategory $\ca{C} \defl \{ X \in \cell(\ca{P}) \;\vert\; X \text{ has a dual and }  X^{\vee} \in \cell(\ca{P}) \}$ of $\cell(\ca{P})$. Since $\ca{P}$ is rigid, we have $Kp \in \ca{C}$ for all $p \in \ca{P}$. Clearly $\ca{C}$ is closed under retracts. It is closed under shifts since $X[1] \cong X \otimes S^{1}$ and $\cell(\ca{P})$ is closed under binary tensor products (see above). Now let $A, B \in \ca{C}$ and let
  \[
  \xymatrix{A \ar[r]^-{f} & B \ar[r] & C_f \ar[r] & A[1] }
  \]
  be a distinguished triangle. Since the monoidal closed structure is compatible with the triangulation, the characterization of duals in terms of the internal hom shows that $C_f$ has a dual. Moreover, there exists a distinguished triangle
  \[
  \xymatrix{B^{\vee} [-1] \ar[r] & A^{\vee}[-1] \ar[r] & C_f^{\vee} \ar[r] & B^{\vee} } \smash{\rlap{,}}
  \]
  so $C_f^{\vee} \in \cell(\ca{P})$ and therefore $C_f \in \ca{C}$. Thus $\ca{C}=\cell(\ca{P})$, which shows that $\cell(\ca{P})$ is closed under duals.
 \end{proof}
 
 \begin{dfn}\label{dfn:H-dual}
 Let $F \colon \ca{T} \rightarrow \Ab$ be a symmetric lax monoidal additive functor which is homological and let $\bar{F}$ be as in Definition~\ref{dfn:K-functor_from_Ab-functor}. Let $B=\bar{F} \U \in \ca{K}$ be the coefficient ring of $\bar{F} \colon \ca{T} \rightarrow \ca{K}$ and let
 \[
 H \defl \bar{F}^{\mathrm{norm}} \colon \ca{T} \rightarrow \ca{K}_B
 \]
 be the normal lift of $\bar{F}$ (see Corollary~\ref{cor:lax_implies_some_isos}). An object $A \in \cell(\ca{P})$ is called an \emph{$H$-dual} if $HA \in \ca{K}_B$ has a dual. A morphism $f \colon X \rightarrow Y$ in $\ca{T}$ is called an \emph{$H$-epimorphism} if $Hf$ is an epimorphism. We write $\ca{A}_H$ for the full subcategory of $\cell(\ca{P})$ consisting of $H$-duals.
\end{dfn}

\begin{lemma}
 The functor $H$ in Definition~\ref{dfn:H-dual} is homological.
\end{lemma}

\begin{proof}
 This follows from the fact that the functor $Kp \otimes - \colon \ca{T} \rightarrow \ca{T}$ is triangulated for all $p \in \ca{P}$.
\end{proof}

\begin{prop}\label{prop:basic_properties_of_H-duals}
 The category $\ca{A}_H$ of $H$-duals (see Definition~\ref{dfn:H-dual}) has the following properties:
\begin{enumerate}
\item[(i)] For each $p \in \ca{P}$, the object $Kp$ lies in $\ca{A}_H$.
\item[(ii)] For all $A \in \ca{A}_H$ and all $X \in \cell(\ca{P})$, the morphisms
\[
\varphi_{A,X} \colon HA \ten{B} HX \rightarrow H(A\otimes X)
\]
and
\[
\varphi_{X,A} \colon HX \ten{B} HA \rightarrow H(X\otimes A)
\]
are isomorphisms.
\item[(iii)] If $A \in \ca{A}_H$, then its dual $A^{\vee}$ lies in $\ca{A}_H$.
\item[(iv)] The category $\ca{A}_H$ is closed under finite tensor products. In particular, $\ca{A}_H$ is a rigid symmetric monoidal category.
\item[(v)] If $f \colon A \rightarrow B$ is a morphism in $\ca{A}_H$ such that $\mathrm{coker}(Hf) \in \ca{K}_B$ has a dual, then both the cofiber $C_f$ and the fiber $F_f$ of $f$ lie in $\ca{A}_H$. In this case, the morphism $\mathrm{coker}(Hf) \rightarrow HC_f$ is a split monomorphism.
\end{enumerate}
\end{prop}

\begin{proof}
Part~(i) follows directly from Corollary~\ref{cor:lax_implies_some_isos}.

 To see~(ii), let $\ca{B} \defl \{ Y \in \cell(\ca{P}) \;\vert\; \varphi_{A,Y} \text{ is an isomorphism for all } A \in \ca{A}_H \}$. By Corollary~\ref{cor:lax_implies_some_isos}, we have $Kp \in \ca{B}$ for all $p \in \ca{P}$, and $\ca{B}$ is clearly closed under retracts. Since $HA$ has a dual, the functor $HA \ten{B} H(-) \colon \ca{T} \rightarrow \ca{K}_B$ is homological. Combined with the five lemma, this implies that $\ca{B}$ is closed under cofibers. That $\ca{B}$ is closed under shifts follows from the compatibility
 \[
 \xymatrix{
 (HA \ten{B} HY) \ten{B} HS^{1} \ar[r]^-{\cong} \ar[d]_{\cong}^{Y \in \ca{B}} & HA \ten{B} (HY \ten{B} HS^{1}) \ar[d]^{\cong}_{\text{Cor.~\ref{cor:lax_implies_some_isos}}} \\
 H(A \otimes Y) \ten{B} HS^{1} \ar[d]_{\cong}^{\text{Cor.~\ref{cor:lax_implies_some_isos}}} & HA \ten{B} H(Y \otimes S^{1}) \ar[d]^{\varphi_{A,Y \otimes S^{1}}} \\
 H \bigl( (A\otimes Y) \otimes S^{1} \bigr) \ar[r]_-{\cong} & H\bigl(A \otimes (Y \otimes S^{1}) \bigr) }
 \]
 of $\varphi$ with associators and Assumption~\eqref{eqn:s1_in_image} that $S^1=K \ell$ for some $\ell \in \ca{P}$. Thus $\ca{B}=\cell(\ca{P})$, so $\varphi_{A,X}$ is an isomorphism for all $A \in \ca{A}_H$ and all $X \in \cell(\ca{P})$. The claim about $\varphi_{X,A}$ follows by symmetry.
 
 To see~(iii), note that $A^{\vee} \in \cell(\ca{P})$ by Lemma~\ref{lemma:cell_P_rigid}. Thus both $\varphi_{A,A^{\vee}}$ and $\varphi_{A^{\vee},A}$ are isomorphisms by~(ii). The usual proof that strong monoidal functors preserve duals now shows that
 \[
 \xymatrix{B \ar[r]^-{\varphi_0} & H\U \ar[r]^-{H(\mathrm{coev})} & H(A \otimes A^{\vee}) \ar[r]^-{\varphi_{A,A^{\vee}}^{-1}} & HA \ten{B} HA^{\vee} }
 \]
 and
 \[
 \xymatrix{ HA^{\vee} \ten{B} HA \ar[r]^-{\varphi_{A^{\vee},A}} & H(A^{\vee} \otimes A) \ar[r]^-{H(\mathrm{ev})} & H\U \ar[r]^-{\varphi_0^{-1}} & B }
 \]
 exhibit $HA^{\vee}$ as dual of $HA$ (it is not required that any further instances of $\varphi$ are invertible, the coherence axioms for lax monoidal functors suffice to show that the above morphisms satisfy the triangle identities). Thus $A^{\vee}$ is indeed an $H$-dual.
 
 To see~(iv), note that $\varphi_{A,A^{\prime}}$ is an isomorphism for all $A,A^{\prime} \in \ca{A}_H$ by~(ii), so $H(A\otimes A^{\prime}) \cong HA \ten{B} HA^{\prime}$ has a dual. We have $\U \in \ca{A}_H$ by~(i), and~(iii) shows that $\ca{A}_H$ is rigid. 
 
 Finally, to see~(v), let $A, B \in \ca{A}_H$ and let $f \colon A \rightarrow B$ be a morphism such that $\mathrm{coker}(Hf)$ has a dual. Recall that the full subcategory $\ca{K}_B^{\mathrm{dual}}$ of duals in $\ca{K}_B$ is precisely the category of finitely presentable projective objects in $\ca{K}_B$ since the unit object $B$ is finitely presentable projective.
 
 Consider the exact sequence
 \[
 \xymatrix@C=10pt{ HA \ar[rr]^-{Hf} && HB \ar[rr]^-{Hg}  \ar@{->>}[rd] && HC_f \ar[rr]^-{Hh} \ar@{->>}[rd] && H(A[1]) \ar[rr]^-{-H(f[1])} \ar@{->>}[rd] && H(B[1]) \ar@{->>}[rd] \\ &&& P_3 \ar@{ >->}[ru] && P_2  \ar@{ >->}[ru] && P_1  \ar@{ >->}[ru] && P_0 }
 \]
 with image factorization depicted in the second row. Since the restriction of $H$ to $\ca{A}_H$ is strong monoidal by~(ii) and $S^{1} \in \ca{A}_H$ by~(i) and~\eqref{eqn:s1_in_image}, it follows that $P_0 \cong P_3 \ten{B} HS^{1}$. Moreover, $P_3 \cong \mathrm{coker}(Hf)$ is finitely presentable projective, so $P_0$ is pojective. Thus $P_1$ is a direct summand of $H(B[1])$, so it is finitely presentable projective. The same reasoning shows that $P_2$ is finitely presentable projective. Thus $HC_f \cong P_3 \oplus P_2$ is finitely presentable projective, so it has a dual. This also shows that all the epimorphisms and all the monomorphisms in the above diagram are split. The fiber $F_f$ is isomorphic to $C_f [-1]$, so $H F_f$ has a dual as well. This shows that both $C_f$ and $F_f$ lie in $\ca{A}_H$.
 \end{proof}
 
 \begin{prop}\label{prop:Pstragowski_topology}
 In the situation of Definition~\ref{dfn:H-dual}, let $p \colon  A \rightarrow B$ be an $H$-epimorphism in $\ca{A}_H$ and let
 \[
 \xymatrix{F_p \ar[r] & A \ar[r]^-{p} & B \ar[r] & F_p[1] }
 \]
 be a distinguished triangle in $\ca{T}$. Then $F_p \in \ca{A}_H$ and the sequence
 \begin{equation}\label{eqn:exact_sequence_in_K_B}
 \xymatrix{0 \ar[r] & HF_p \ar[r] & HA \ar[r]^-{Hp} & HB \ar[r] & 0}
 \end{equation}
 in $\ca{K}_B$ is exact. In particular, if $f \colon B^{\prime} \rightarrow B$ is a morphism in $\ca{A}_H$, then there exists a morphism $f^{\prime} \colon A^{\prime} \rightarrow A$ in $\ca{A}_H$ and an $H$-epimorphism $p^{\prime} \colon A^{\prime} \rightarrow B^{\prime}$ such that the square
 \[
 \xymatrix{ A^{\prime} \ar[r]^-{p^{\prime}} \ar[d]_{f^{\prime}} & B^{\prime} \ar[d]^{f} \\ A \ar[r]_-{p} & B }
 \]
 is commutative. Thus the $H$-epimorphisms form a singleton Grothendieck coverage on $\ca{A}_H$.
 \end{prop}
 
 \begin{proof}
  The object $F_p$ lies in $\ca{A}_H$ by Part~(v) of Proposition~\ref{prop:basic_properties_of_H-duals}. Since $H$ is homological and both $Hp$ and $H(p[-1]) \cong Hp \ten{B} HS^{-1}$ are epimorphisms, the sequence~\eqref{eqn:exact_sequence_in_K_B} is exact. 
  
  Let $A^{\prime}$ be the fiber of the $H$-epimorphism $\left( \begin{smallmatrix} p && -f \end{smallmatrix} \right) \colon A \oplus B^{\prime} \rightarrow B$. In this case, the exact sequence~\eqref{eqn:exact_sequence_in_K_B} implies that $HA^{\prime}$ is the pullback of $Hp$ along $Hf$. Thus $Hp^{\prime}$ is an epimorphism since $\ca{K}_B$ is abelian, which shows the second claim.
 \end{proof}
 
 \begin{dfn}\label{dfn:Pstragowski-topology}
 The additive Grothendieck topology $\tau_H$ on $\ca{A}_H$ induced by the singleton coverage of $H$-epimorphisms is called the Pstr\k{a}gowski-topology. Its category of sheaves is denoted by $\sh(\ca{A}_H) \subseteq [\ca{A}_H^{\op},\Ab]$.
 \end{dfn}
 
 For each $H$-epimorphism $p \colon A \rightarrow B$ in $\ca{A}_H$, we fix a distinguished triangle
 \[
\xymatrix{F_p \ar[r] & A \ar[r]^-{p} & B \ar[r] & F_p[1]} 
 \]
 and we denote the resulting set of diagrams $\{\xymatrix@1{F_p \ar[r] & A \ar[r]^-{p} & B } \}$ in $\ca{A}_H$ by $\Sigma_H$.
 
 \begin{prop}\label{prop:sheaves_characterization}
 The category $\sh(\ca{A}_H)$ is the full subcategory $\Lex_{\Sigma_H}(\ca{A}_H)$ of $[\ca{A}_H^{\op},\Ab]$ consisting of additive functors $G$ such that
 \[
 \xymatrix{0 \ar[r] & GB \ar[r]^-{Gp} & GA \ar[r] & GF_p }
 \]
 is left exact in $\Ab$ for each diagram in $\Sigma_H$. In particular, $\sh(\ca{A}_H)$ is closed under filtered colimits in $[\ca{A}_H^{\op},\Ab]$, hence locally finitely presentable, and the exact reflector $L \colon [\ca{A}_H^{op},\Ab] \rightarrow \sh(\ca{A}_H)$ preserves finitely presentable objects. 
 \end{prop}
 
 \begin{proof}
 The first claim follows as in the proof of \cite[Proposition~2.1.5]{SCHAEPPI_MOTIVES} from the fact that $F_p$ is a weak kernel of $p$ in $\ca{T}$ (and thus also in $\ca{A}_H$). Since the right adjoint $\sh(\ca{A}_H) \rightarrow [\ca{A}_H^{\op},\Ab]$ preserves filtered colimits, its left adjoint $L$ preserves finitely presentable objects. Thus $\sh(\ca{A}_H)$ is locally finitely presentable, as claimed.
 \end{proof}
 
 \begin{thm}\label{thm:Grothendieck_tensor_from_homology_theory}
 Let $F \colon \ca{T} \rightarrow \Ab$ be a symmetric lax monoidal homological functor, $\bar{F} \colon \ca{T} \rightarrow \ca{K}$ the functor sending $X$ to $F\bigl(K(-)^{\vee} \otimes X \bigr)$, and $B \defl \bar{F} \U$ the coefficient ring of $\bar{F}$ (cf.\ Corollary~\ref{cor:lax_implies_some_isos}). Let
 \[
 H\defl \bar{F}^{\mathrm{norm}} \colon \ca{T} \rightarrow \ca{K}_B
 \]
 be the normal lift of $\bar{F}$ and let $\ca{A}_H \subseteq \cell(\ca{P})$ be the full subcategory of $H$-duals. We write $J \colon \ca{A}_H \rightarrow \ca{T}$ for the inclusion.
 
 \begin{enumerate}
 \item[(i)] The category $\sh(\ca{A}_H)$ of sheaves for the Pstr\k{a}gowski-topology (see Defininition~\ref{dfn:Pstragowski-topology}) is a Grothendieck tensor category and the reflector 
 \[
 L \colon [\ca{A}_H^{\op},\Ab] \rightarrow \sh(\ca{A}_H)
 \]
 is symmetric strong monoidal for the Day convolution.
 \item[(ii)] The composite
 \[
 \xymatrix@C=25pt{\ca{T} \ar[rr]^-{\mathrm{Hom}_{\ca{A}_H}(J,-)} && [\ca{A}_H^{\op},\Ab] \ar[r]^-{L} & \sh(\ca{A}_H) }
 \]
 is symmetric normal lax monoidal and homological.
 \item[(iii)] The functor $- \ten{\ca{A}_H} H \colon [\ca{A}_H^{\op},\Ab] \rightarrow \ca{K}_B$ restricts to a symmetric strong monoidal left adjoint
 \[
 - \ten{\ca{A}_H} H \colon \sh(\ca{A}_H) \rightarrow \ca{K}_B
 \]
 with the following property: if $G \in \sh(\ca{A}_H)$ is finitely generated (that is, there exists an epimorphism $G^{\prime} \rightarrow G$ with $G^{\prime}$ finitely presentable) and $G \ten{\ca{A}_H} H \cong 0$, then $G \cong 0$.
 \end{enumerate}
 \end{thm}
 
 \begin{proof}
  To see~(i), we need to check that the conditions of Day's reflection theorem are satisfied. As observed in \cite[\S 2.2]{SCHAEPPI_MOTIVES}, these follow from the fact that $\Sigma_H$ is (up to isomorphism) closed under $A \otimes -$ for all $A \in \ca{A}_H$ and the characterization of $\sh(\ca{A}_H)$ in terms of $\Sigma_H$ in Proposition~\ref{prop:sheaves_characterization}. That the unit object is finitely presentable follows from the second part of Proposition~\ref{prop:sheaves_characterization}. The category is generated by duals since $\ca{A}_H$ is rigid (see Proposition~\ref{prop:basic_properties_of_H-duals}).
  
  To see~(ii), note that $\mathrm{Hom}_{\ca{A}_H}(J,-)$ is homological and $L$ is exact, so the composite is homological. Since $L$ is strong monoidal, it only remains to check that $\widetilde{J} \defl \mathrm{Hom}_{\ca{A}_H}(J,-)$ is symmetric lax monoidal. It preserves the unit strictly and the morphisms
  \[
  \ca{T}(JA,X) \otimes \ca{T}(JB,Y) \otimes \ca{A}_H(C,A\otimes B) \rightarrow \ca{T}(JC,X \otimes Y)
  \]
  which send $f \otimes g \otimes \varphi$ to $(f\otimes g) \circ J\varphi$ induce the desired natural transformation $\varphi_{X,Y} \colon \widetilde{J}X \dt \widetilde{J} Y \rightarrow \widetilde{J} ( X\otimes Y)$.
  
  It remains to show~(iii). From the universal property of $\Lex_{\Sigma_H} (\ca{A}_H)$ (see \cite[Theorem~2.2.1]{SCHAEPPI_MOTIVES} ) and Proposition~\ref{prop:Pstragowski_topology} it follows that $-\ten{\ca{A}_H} H$ restricts to a symmetric strong monoidal left adjoint on $\sh(\ca{A}_H)$.
  
  In the interest of legibility, we simply write $\ca{A}$ for $\ca{A}_H$ in the remainder of the proof. We write $\eta$ for the unit of the left adjoint $L$. Let $G \colon \ca{A}^{\op} \rightarrow \Ab$ be a finitely generated additive sheaf such that $G \ten{\ca{A}} H \cong 0$. The objects $L \bigl(\ca{A}(-,A) \bigr)$, $A \in \ca{A}$, are closed under finite direct sums and they generate $\sh(\ca{A})$, so there exists an object $A \in \ca{A}$ and an epimorphism $p \colon L \bigl(\ca{A}(-,A) \bigr) \rightarrow G$ in $\sh(\ca{A})$. The composite
  \[
  \xymatrix{\ca{A}(-,A) \ar[r]^-{\eta_{\ca{A}(-,A)}} & L \bigl(\ca{A}(-,A) \bigr) \ar[r]^-{p} & G }
  \]
  in $[\ca{A}^{\op},\Ab]$ is sent to an epimorphism $L \bigl(\ca{A}(-,A) \bigr) \rightarrow LG \cong G$ by $L$. Consider the kernel $k \colon K \rightarrow \ca{A}(-,A)$ of $p \eta_{\ca{A}(-,A)}$ in $[\ca{A}^{\op},\Ab]$. We choose an epimorphism
  \[
q\colon  \oplus_{i \in I} \ca{A}(-,A_i) \rightarrow K
  \]
  in $[\ca{A}^{\op},\Ab]$. By construction, the exact functor $L$ sends the cokernel of $kq$ to $LG \cong G$, so it suffices to show that $L(kq)$ is an epimorphism in $\sh(\ca{A})$.
  
  The right adjoint $\mathrm{Hom}_{\ca{A}}(H,-)$ of $-\ten{\ca{A}} H$ factors trough $\sh(\ca{A})$ by Propositions~\ref{prop:Pstragowski_topology} and \ref{prop:sheaves_characterization}, so $\eta \ten{\ca{A}} H \colon - \ten{\ca{A}} H \Rightarrow L(-)\ten{\ca{A}} H$ is an isomorphism. From this it follows that $-\ten{\ca{A}} H$ sends the cokernel of $kq$ to $G \ten{\ca{A}} H \cong 0$, that is, $kq \ten{\ca{A}} H$ is an epimorphism in $\ca{K}_B$. The object $\ca{A}(-,A) \ten{\ca{A}} H \cong HA$ is finitely presentable (and projective), so there exists a finite subset $I_0 \subseteq I$ such that
  \[
  H\bigl( \xymatrix{\oplus_{i \in I_0} \ca{A}(-,A_i) \ar[r]^-{\mathrm{incl}} & \oplus_{i \in I} \ca{A}(-,A_i) \ar[r]^-q  & K \ar[r]^-k & \ca{A}(-,A) } \bigr)
  \]
  is an epimorphism. From the definition of the Pstr\k{a}gowski-topology it follows that $L\bigl(\oplus_{i \in I_0 } \ca{A}(-,A_i) \rightarrow \ca{A}(-,A) \bigr)$ is an epimorphism, hence that $L(kq)$ is an epimorphism. Thus $G \cong LG \cong \mathrm{coker}\bigl(L(kq)\bigr) \cong 0$, as claimed.
 \end{proof}
 
 In general, the functor $- \ten{\ca{A}_{H}} H \colon \sh(\ca{A}_H) \rightarrow \ca{B}$ will \emph{not} be left exact. Consequenctly, Part~(iii) of the above theorem does not imply that the functor is faithful. If $\ca{A} \subseteq \ca{T}$ is any subcategory, $H \colon \ca{T} \rightarrow \ca{K}_B$ a (homological) functor, we can consider the left Kan extension of $H \vert_{ \ca{A}}$ along the inclusion $J \colon \ca{A} \rightarrow \ca{T}$. By \cite[Formula~(4.17)]{KELLY_BASIC}, this is given by the composite of $\mathrm{Hom}_{\ca{A}}(J,-) \colon \ca{T} \rightarrow [\ca{A}^{\op},\Ab]$ with $-\ten{\ca{A}} H \colon [\ca{A}^{\op},\Ab] \rightarrow \ca{K}_B$. By definition, we have a natural transformation $\alpha \colon \mathrm{Hom}_{\ca{A}}(J,-) \ten{\ca{A}} H \Rightarrow H$ and its restriction to $\ca{A}$ is an isomorphism by \cite[Proposition~4.23]{KELLY_BASIC}. In particular, in the situation of Theorem~\ref{thm:Grothendieck_tensor_from_homology_theory}, we have a natural transformation
 \begin{equation}\label{eqn:natural_transformation_Kan_to_H}
\alpha \colon L\bigl( \mathrm{Hom}_{\ca{A}}(J,-) \bigr) \ten{\ca{A}_H} H \Rightarrow H \colon \ca{T} \rightarrow \ca{K}_B 
 \end{equation}
 whose restriction to $\ca{A}_H$ is an isomorphism (this relies on the fact that $-\ten{\ca{A}_H} H$ sends the unit of $L$ to an isomorphism). At this level of generality, we do not expect that this natural transformation is an isomorphism on all of $\ca{T}$.  More can be said if the homology theory $F$ is $\ca{T}$-flat.
 
 \begin{dfn}\label{dfn:homology_theory_of_Adams_type}
 Let $F \colon \ca{T} \rightarrow \Ab$ be a symmetric lax monoidal homological functor and let $H \defl \bar{F}^{\mathrm{norm}} \colon \ca{T} \rightarrow \ca{K}_B$. If the restriction of $H$ to the category $\ca{A}_H$ of $H$-duals in $\cell(\ca{P})$ induces a \emph{left exact} functor
 \[
 -\ten{\ca{A}_H} H \colon [\ca{A}_{H}^{\op},\Ab] \rightarrow \ca{K}_B \smash{\rlap{,}}
 \]
 then $F$ is said to be of $\ca{T}$-flat relative to $K \colon \ca{P} \rightarrow \ca{T}$.
  \end{dfn}
  
\begin{prop}\label{prop:Adams_characterization}
 Let $F \colon \ca{T} \rightarrow \Ab$ be a symmetric lax monoidal homological functor and let $H \defl \bar{F}^{\mathrm{norm}}$. Let $\ca{A}_H$ be the category of $H$-duals in $\cell(\ca{P})$. Then the following are equivalent:
 \begin{enumerate}
 \item[(i)] The homological functor $F$ is $\ca{T}$-flat relative to $K \colon \ca{P} \rightarrow \ca{T}$.
 \item[(ii)] The restriction of $F$ to $\ca{A}_H$ is \emph{flat}: the functor $- \ten{\ca{A}_H} F \colon [\ca{A}_H^{\op},\Ab] \rightarrow \Ab$ is exact.
 \item[(iii)] There exists a filtered diagram $\ca{I} \rightarrow \ca{A}_H$, $i \mapsto A_i$ and isomorphisms
 \[
 \colim_{i \in \ca{I}} \ca{A}_H(A_i^{\vee},A) \cong FA
 \]
 which are natural in $A \in \ca{A}_H$.
 \end{enumerate}
\end{prop}

\begin{proof}
 Clearly (i) implies (ii) since both the forgetful functor $\ca{K}_B \rightarrow \ca{K}$ and the functor $U=\mathrm{ev}_{\U} \colon \ca{K} \rightarrow \Ab$ preserve colimits, hence functor tensor products, and they are both exact. It is well-known that~(ii) holds if and only if the restriction of $F$ to $\ca{A}_H$ is a filtered colimit of representable functors. Together with the rigidity of $\ca{A}_H$ (see Proposition~\ref{prop:basic_properties_of_H-duals}~(iv)), this implies (iii). Finally, from~(iii) it follows that the functor which sends $A \in \ca{A}_H$ to $F(Kp^{\vee} \otimes A)$ can be written as filtered colimit of the functors $\ca{A}_H(A_i^{\vee}\otimes Kp,-)$, so $\mathrm{ev}_{p} \circ H$ is flat for each $p \in \ca{P}$. This implies the exactness of $- \ten{\ca{A}_H} H \colon [\ca{A}_H^{\op},\Ab] \rightarrow \ca{K}_B$, hence that $F$ is of $\ca{T}$-flat (relative to $K$).
\end{proof}

 In the case of $\ca{T}$-flat homological functors, we can say more about the canonical natural transformation~\eqref{eqn:natural_transformation_Kan_to_H}.

\begin{prop}\label{prop:Left_Kan_if_Adams}
 Suppose that the image of $K \colon \ca{P} \rightarrow \ca{T}$ consists of compact objects and that the symmetric lax monoidal functor $F \colon \ca{T} \rightarrow  \Ab$ preserves coproducts. In this case, if $F$ is $\ca{T}$-flat, then the natural transformation~\eqref{eqn:natural_transformation_Kan_to_H} is an isomorphism on the smallest localizing subcategory of $\ca{T}$ which contains the image of $K \colon \ca{P} \rightarrow \ca{T}$.
\end{prop}

\begin{proof}
 Note that both the domain and the codomain of the natural transformation~\eqref{eqn:natural_transformation_Kan_to_H} are homological if $F$ is of $\ca{T}$-flat. The assumption implies that $\cell(\ca{P})$ consists of compact objects, so both functors preserve coproducts. Thus the subcategory of $X \in \ca{T}$ such that $\alpha_X$ is an isomorphism is localizing and it contains $\ca{A}_H$, hence in particular the image of $K$.
\end{proof}

\subsection{Flat replacements}\label{section:Adams_replacement}

 In this section, we fix a closed $\otimes$-triangulated category $\ca{T}$, a small rigid symmetric monoidal $\Ab$-category $\ca{P}$ and a symmetric strong monoidal additive functor $K \colon \ca{P} \rightarrow \ca{K}$ and we let $\ca{K}\defl [\ca{P}^{\op},\Ab]$. In addition, we fix a symmetric lax monoidal additive functor $F \colon \ca{T} \rightarrow \Ab$, we let $B \in \ca{K}$ be the coefficient ring of $\bar{F}$, and we let $H \defl \bar{F}^{\mathrm{norm}} \colon \ca{T} \rightarrow \ca{K}_B$ be the normal lift of $\bar{F}$. We write $J \colon \ca{A}_H \rightarrow \ca{T}$ for the inclusion of the $H$-duals in $\cell(\ca{P})$ and we use the abbreviation $\widetilde{J} \defl \mathrm{Hom}_{\ca{A}_H}(J,-) \colon \ca{T} \rightarrow [\ca{A}_H^{\op},\Ab]$.
 
 If $F$ is $\ca{T}$-flat, then the category of sheaves on $\ca{A}_H$ has a description in terms of comodules.
 
 \begin{prop}\label{prop:Adams_type_implies_sheaves_are_comodules}
  In the situation of Theorem~\ref{thm:Grothendieck_tensor_from_homology_theory}, if $F$ is $\ca{T}$-flat relative to $K$, then the functor
  \[
  -\ten{\ca{A}_H} H \colon \sh(\ca{A}_H) \rightarrow \ca{K}_B
  \]
  is faithful and exact, hence comonadic. The comonad on $\ca{K}_B$ is given by $\Gamma \ten{B}-$ for a flat coalgebroid $(B,\Gamma)$ in $\ca{K}$.
 \end{prop}

\begin{proof}
 Since $-\ten{\ca{A}_H} H \colon [\ca{A}_H^{\op},\Ab] \rightarrow \ca{K}_B$ is exact, so is its restriction to the category of sheaves. From Theorem~\ref{thm:Grothendieck_tensor_from_homology_theory}~(iii) it follows that $-\ten{\ca{A}_{H}} H$ is also faithful on the category of sheaves, hence comonadic.
 
 It remains to show that the comonad is induced by a flat coalgebroid in $\ca{K}$. To simplify the notation, we let $\ca{A} \defl \ca{A}_H$. The comonad in question coincides with the comonad induced by the adjunction $-\ten{\ca{A}} H \dashv \mathrm{Hom}_{\ca{A}}(H,-) \colon [\ca{A}^{\op},\Ab] \rightarrow \ca{K}_B$. We claim that this is the underlying comonad of a $\ca{K}$-enriched comonad $\mathbb{K}_B \rightarrow \mathbb{K}_B$ (see \S \ref{section:coefficient_categories} for the notation). The functor $H \colon \ca{A} \rightarrow \ca{K}_B$ is the underlying functor of the $\ca{K}$-enriched functor $\mathbb{F}^{\mathrm{norm}} \colon \mathbb{A} \rightarrow \mathbb{K}_B$. Let $U\defl \ev_{\U} \colon \ca{K} \rightarrow \Ab$ and let $\Psi \colon U_{\ast} [\mathbb{A}^{\op},\mathbb{K}] \rightarrow [\ca{A}^{\op},\Ab]$ be the functor which sends $G$ to $U \circ U_\ast G$. 
 
 The proof of Theorem~\ref{thm:K-functor_Ab-functor_equivalence} shows that $\Psi$ is an equivalence of categories (by simply ignoring the lax monoidal structure on either side). Moreover, there is a natural isomorphism $\Psi \circ \mathrm{Hom}_{\mathbb{A}}(\mathbb{F}^{\mathrm{norm}},-) \cong \mathrm{Hom}_{\ca{A}}(H,-)$. From this it follows that the comonad induced by the additive adjunction $-\ten{\ca{A}} H \dashv \mathrm{Hom}_{\ca{A}}(H,-)$ is isomorphic to the underlying additive comonad induced by the $\ca{K}$-adjunction $-\ten{\mathbb{A}} \mathbb{F}^{\mathrm{norm}} \dashv \mathrm{Hom}_{\mathbb{A}}(\mathbb{F}^{\mathrm{norm}},-)$.
 
 Since any $\ca{K}$-cocontinuous functor $\mathbb{K}_B \rightarrow \mathbb{K}_B$ is given by tensoring with some bimodule by \cite[Theorem~4.51]{KELLY_BASIC}, it suffices to show that $\mathrm{Hom}_{\mathbb{A}}(\mathbb{F}^{\mathrm{norm}},-)$ is $\ca{K}$-cocontinuous. This amounts to showing that $\mathbb{K}_B(\mathbb{F}^{\mathrm{norm}}A,-)$ is $\ca{K}$-cocontinuous for each $A \in \ca{A}$. Since the object $\mathbb{F}^{\mathrm{norm}}A=HA$ has a dual by definition of $\ca{A}$, this follows from the fact that $\mathbb{K}_B(B,-)$ is $\ca{K}$-cocontinuous. Thus the comonad is given by $\Gamma \ten{B} -$ for a coalgebroid $(B,\Gamma)$ in $\ca{K}$. Clearly $\Gamma$ is flat since we assumed that $F$ is $\ca{T}$-flat. 
\end{proof}

 It would be interesting to know if the coalgebroid of Proposition~\ref{prop:Adams_type_implies_sheaves_are_comodules} is in fact a Hopf algebroid, and to compare it with the classical Hopf algebroid associated to a flat homology theory in the case of spectra. This would follow if the equivalence $\Psi$ is symmetric strong monoidal.

 Even if $F$ is \emph{not} $\ca{T}$-flat, we can still ask if $\sh(\ca{A}_H)$ is equivalent to the category of comodules of some commutative Hopf algebroid $(A,\Gamma)$ in $\ca{K}$ (where the commutative algebra $A$ is not necessarily isomorphic to $B$). In this case, we have a faithful and exact symmetric strong monoidal left adjoint $W \colon \sh(\ca{A}_H) \rightarrow \ca{K}_A$, so we get a new homological functor
 \begin{equation}\label{eqn:Adams_replacement_from_W}
 F(W) \defl \xymatrix@C=20pt{\ca{T} \ar[r]^-{\widetilde{J}} & [\ca{A}_H^{\op},\Ab] \ar[r]^-{L} & \sh(\ca{A}_H) \ar[r]^-{W} & \ca{K}_A \ar[r]^-{\mathrm{forget}} & \ca{K} \ar[r]^-{\;U=\mathrm{ev}_{\U}} & \Ab }
 \end{equation}
 which is symmetric lax monoidal. The first goal of this section is to show that $F(W)$ is $\ca{T}$-flat relative to $K \colon \ca{P} \rightarrow \ca{T}$ if $W$ is compatible with the $\ca{K}$-enrichment (see Theorem~\ref{thm:Adams_replacement} below). To express the compatibility with enrichment, it is convenient to use the following proposition. We write $\ssm_{\ca{K}}$ for the 2-category of symmetric monoidal $\ca{K}$-categories, symmetric strong monoidal functors, and symmetric monoidal $\ca{K}$-natural transformations, respectively $\ssm_{\Ab}$ for the $\Ab$-enriched version of this 2-category.
 
 \begin{prop}\label{prop:strong_monoidal_enriched_equivalent_to_slice}
  The full sub-2-category of $\ssm_{\ca{K}}$ consisting of symmetric monoidal $\ca{K}$-categories $\mathbb{B}$ with copowers $\ca{P}(-,p) \odot X$ for all $p \in \ca{P}$, $X \in \mathbb{B}$ is biequivalent to the slice 2-category $\ca{P} \slash \ssm_{\Ab}$. The biequivalence sends $\mathbb{B}$ to the $\Ab$-category $U_{\ast} \mathbb{B}$ with functor $\ca{P} \rightarrow U_{\ast} \mathbb{B}$ given by $p \mapsto \ca{P}(-,p) \odot \U$.
 \end{prop}
 
 \begin{proof}
 The inverse biequivalence sends $\ca{B}$ to the symmetric monoidal $\ca{K}$-category $\mathbb{B}$ constructed in Propositions~\ref{prop:enrichment_from_K} and \ref{prop:K-category_symmetric_monoidal}. Given a 1-cell $(\ca{B},F) \rightarrow (\ca{B}^{\prime},F^{\prime})$ in the slice 2-category, that is, a symmetric strong monoidal additive functor $G \colon \ca{B} \rightarrow \ca{B}^{\prime}$ and a symmetric monoidal isomorphism $\alpha \colon F^{\prime} \Rightarrow GF$, we construct a $\ca{K}$-functor $\mathbb{G}$ using the composites
 \[
 \xymatrix@C=50pt{\ca{B}\bigl(F(-)\otimes X,Y\bigr) \ar@{-->}[d]_{\mathbb{G}_{X,Y}} \ar[r]^-{F} & \ca{B}^{\prime}\Bigl(G\bigl(F(-) \otimes X\bigr),GY\Bigr) \ar[d]^{(\varphi_{F(-),X})^{\ast}} \\ \ca{B}^{\prime}\bigl(F^{\prime}(-) \otimes GX,GY\bigr) & \ca{B}^{\prime}\bigl(GF(-)\otimes GX,GY\bigr) \ar[l]^{\alpha_{(-)} \otimes GX)^{\ast}} }
 \]
 for $X,Y \in \ca{B}$. With this definition, a morphism $\xymatrix@1{F\U \otimes X  \ar[r]^-{\cong} & X \ar[r]^-{f} &Y}$ in $U_{\ast} \mathbb{B}$ induces the morphisms
 \[
 \bigl(F(-)\otimes f\bigr)^{\ast} \colon \ca{B}\bigl(F(-)\otimes Y,Z\bigr) \rightarrow \ca{B}\bigl(F(-)\otimes X,Z\bigr)
 \]
 and
 \[
 f_{\ast} \colon \ca{B}\bigl(F(-)\otimes Z,X\bigr) \rightarrow \ca{B}\bigl(F(-)\otimes Z,Y\bigr)
 \]
 by pre- and postcomposition. From this and the coherence axioms for symmetric monoidal functors it follows that the symmetric strong monoidal structure of $G$ lifts to one of $\mathbb{G}$ (the explicit description of pre- and postcomposition above is used to check $\ca{K}$-naturality).
 
 The axiom
 \[
\vcenter{ \xymatrix{ & \ca{P} \ar[ld]_{F^{\prime}} \ar[rd]^F \ar@{}[d]|(0.4){\cong} \\ \ca{B} \rrtwocell^{G}_{G^{\prime}}{\gamma} && \ca{B}^{\prime} }} \quad=\quad
\vcenter{
 \xymatrix{& \ca{P} \ar[ld]_{F^{\prime}} \ar[rd]^F \ar@{}[d]|(0.6){\cong} \\ \ca{B} \ar[rr]_{G^{\prime}} && \ca{B}^{\prime}}
 }
 \]
 expressing the fact that $\gamma$ is a 2-cell in the slice $\ca{P} \slash \ssm_{\Ab}$, together with the explicit description of pre- and postcomposition with components of $\gamma$ shows that $\gamma$ lifts to a $\ca{K}$-natural symmetric monoidal transformation $\mathbb{G} \Rightarrow \mathbb{G}^{\prime}$. This defines a 2-functor
 \[
 \ca{P} \slash \ssm_{\Ab} \rightarrow \ssm_{\ca{K}} 
 \]
 and by construction, its composite with $U_{\ast}$ is pseudonaturally isomorphic to the identity. From the description of $\ca{K}$-enrichments in terms of $\ca{P}$-actions (see \cite[Theorem~3.4]{GORDON_POWER}) it follows that the above 2-functor is a biequivalence. The claim thus follows from the 2-by-3 property of biequivalences.
 \end{proof}
 
 The strong monoidal functors $YK  \colon \ca{P} \rightarrow [\ca{A}_{H}^{\op},\Ab]$ and $LYK \colon \ca{P} \rightarrow \sh(\ca{A}_H)$ thus define $\ca{K}$-enrichments of their targets. By the above proposition, $L$ is the underlying additive functor of a symmetric strong monoidal $\ca{K}$-functor for these $\ca{K}$-enrichments. Since the functor $\widetilde{J} \colon \ca{T} \rightarrow [\ca{A}_H^{\op},\Ab]$ is merely lax monoidal, the above proposition is not applicable.
 
 \begin{lemma}\label{lemma:J_tilde_enrichment}
 The functor $\widetilde{J} \colon \ca{T} \rightarrow [\ca{A}_H^{\op},\Ab]$ is the underlying additive functor of a symmetric lax monoidal $\ca{K}$-functor, where the $\ca{K}$-enrichments are given by $K$ and $YK$ respectively.
 \end{lemma}
 
 \begin{proof}
 The $\ca{P}$-action on the target is given by $p \cdot G =G(-\otimes Kp^{\vee})$ by construction of the Day convolution tensor product. The isomorphisms
 \[
 p \cdot \mathrm{Hom}_{\ca{A}_H}(J,X)=\ca{T}\bigl(J(-\otimes Kp^{\vee}),X\bigr) \cong \ca{T}(J-,Kp \otimes X)
 \]
 give the desired $\ca{K}$-enrichment by \cite[Theorem~3.4]{GORDON_POWER}. Compatibility with the lax monoidal structure follows from the coherence axioms of $K$ and the symmetric monoidal category $\ca{T}$.
 \end{proof}
 
 \begin{lemma}\label{lemma:HW_description}
 Let $A$ be a commutative algebra in $\ca{K}$ and let $W \colon \sh(\ca{A}_H) \rightarrow \ca{K}_A$ be a symmetric strong monoidal functor which is compatible with the enrichment. Let $H(W)$ be the normal lift of the symmetric monoidal functor $\overline{F(W)} \colon \ca{T} \rightarrow \ca{K}$ induced by the composite~\eqref{eqn:Adams_replacement_from_W}. Then $H(W)$ is naturally isomorphic to the composite $WL\widetilde{J} \colon \ca{T} \rightarrow \ca{K}_A$.
 \end{lemma}
 
 \begin{proof}
 Compatibility of $W$ with the enrichment and Lemma~\ref{lemma:J_tilde_enrichment} imply that the composite
 \[
 \xymatrix{\ca{T} \ar[r]^-{\widetilde{J}} & [\ca{A}_H^{\op},\Ab] \ar[r]^-{L} & \sh(\ca{A}_H) \ar[r]^-{W} & \ca{K}_A \ar[r]^-{\mathrm{forget}} & \ca{K} }
 \]
 is the underlying functor of a lax monoidal $\ca{K}$-functor. From Theorem~\ref{thm:K-functor_Ab-functor_equivalence} it follows that the $\ca{K}$-functor in question is naturally isomorphic to $\Phi\bigl(F(W)\bigr)$, so the above composite is naturally isomorphic to $U_{\ast} \Phi\bigl(F(W) \bigr)=\overline{F(W)}$. Thus the normal lift $H(W)$ is naturally isomorphic to $WL\widetilde{J}$, as claimed.
 \end{proof}
 
 \begin{thm}\label{thm:Adams_replacement}
 Let $A \in \ca{K}$ be a commutative algebra, $W \colon \sh(\ca{A}_H) \rightarrow \ca{K}_A$ an exact, faithful, and symmetric strong monoidal left adjoint which is compatible with the $\ca{K}$-enrichment. Then the symmetric lax monoidal functor $F(W)$ given by the composite~\eqref{eqn:Adams_replacement_from_W} is a $\ca{T}$-flat homological functor (relative to $K \colon \ca{P} \rightarrow \ca{T}$). Moreover, the following hold (where $H(W)$ is defined in Lemma~\ref{lemma:HW_description}):
 \begin{enumerate}
 \item[(i)] The category $\ca{A}_H$ is a full subcategory of $\ca{A}_{H(W)}$ and a morphism in $\ca{A}_H$ is an $H$-epimorphism if and only if it is an $H(W)$-epimorphism.
 \item[(ii)] For every $H(W)$-dual $Z$ in $\cell(\ca{P})$, there exists an $H$-dual $X$ in $\cell(\ca{P})$ and an $H(W)$-epimorphism $X \rightarrow Z$, so the inclusion $\ca{A}_H \rightarrow \ca{A}_{H(W)}$ induces an equivalence of categories of additive sheaves.
 \item[(iii)] The site $\ca{A}_{H(W)}$ is independent of the choice of $A$ and $W$.
 \end{enumerate}
 \end{thm}
 
 \begin{proof}
 Let $G \colon \ca{A}_H \rightarrow \Ab$ be the restriction of the composite
 \[
 \xymatrix{ [\ca{A}_H^{\op}, \Ab] \ar[r]^-{L} & \sh(\ca{A}_H) \ar[r]^-{W} & \ca{K}_A \ar[r]^-{\mathrm{forget}} & \ca{K} \ar[r]^-{U} & \Ab }
 \]
 to $\ca{A}_H$ (along the Yoneda embedding). Since the above functor is cocontinuous, it is isomorphic to $\Lan_Y G$. Thus $F(W)\cong \Lan_Y G \widetilde{J} \cong \Lan_J G$ (for the second isomorphism, see \cite[Formulas~(4.17) and (4.13)]{KELLY_BASIC}). Moreover, the above functor is exact by assumption, so $G$ is flat, hence a filtered colimit of representables. Since left Kan extension along a fixed functor preserves colimits and representables, it follows that $F(W)$ is naturally isomorphic to a filtered colimit of functors $\ca{T}(A_i^{\vee},-)$ with $A_i \in \ca{A}_H$. To show that $F(W)$ is $\ca{T}$-flat, it thus suffices to show Claim~(i) (in fact, we only need the inclusion of $\ca{A}_H$ in $\ca{A}_{H(W)}$ for this, see Proposition~\ref{prop:Adams_characterization}).
 
 To see~(i), note that the faithful and exact functor $W$ detects (and preserves) duals. From the natural isomorphism $H(W) \cong WL \widetilde{J}$ (see Lemma~\ref{lemma:HW_description}) it follows that $A \in \cell{\ca{P}}$ is an $H(W)$-dual if and only if $L \widetilde{J}$ is a dual in $\sh(\ca{A}_H)$. Since $L$ preserves duals, this shows that $H$-duals are $H(W)$-duals. From the same isomorphism it follows that a morphism $f \colon A \rightarrow B$ in $\ca{A}_H$ is an $H(W)$-epimorphism if and only if $LY(f)$ is an epimorphism in $\sh(\ca{A}_H)$. That this is equivalent to $Hf$ being an epimorphism follows from Part~(iii) of Theorem~\ref{thm:Grothendieck_tensor_from_homology_theory}.
 
 To see~(ii), let $Z \in \cell(\ca{P})$ be an $H(W)$-dual and let
 \[
 q \colon \oplus_{i \in I} \ca{A}_H(-,X_i) \rightarrow \widetilde{J}Z
 \]
 be an epimorphism in $[\ca{A}_H^{\op},\Ab]$. As we have seen above, $L \widetilde{J} Z$ has a dual, so it is finitely presentable since $\sh(\ca{A}_H)$ is a Grothendieck tensor category (see Part~(i) of Theorem~\ref{thm:Grothendieck_tensor_from_homology_theory}). It follows that there exists a finite subset $I_0 \subseteq I$ such that $L$ applied to the composite
 \[
 \xymatrix{ \oplus_{i \in I_0} \ca{A}_H(-,X_i) \ar[r]^-{\mathrm{incl}} & \oplus_{i \in I} \ca{A}_H(-,X_i) \ar[r]^-{q}  & \widetilde{J}Z }
 \]
 is an epimorphism. By Yoneda, this composite is of the form $\widetilde{J}p$ for some morphism $p \colon \oplus_{i \in I_0} X_i \rightarrow Z$ in $\ca{T}$. Since $L \widetilde{J} p$ is an epimorphism and $H(W) \cong WL \widetilde{J}p$, this morphism $p$ is the desired $H(W)$-epimorphism.
 
 Finally, that the site $\ca{A}_{H(W)}$ is independent of the algebra $A$ and the functor $W$ follows from the natural isomorphism $H(W) \cong WL \widetilde{J}$ and the fact that $W$ preserves and detects both duals and epimorphisms.
 \end{proof}
 
 \begin{dfn}\label{dfn:Adams_replacement}
 In the situation of Theorem~\ref{thm:Adams_replacement}, we say that the homological functor $F \colon \ca{T} \rightarrow \Ab$ admits a \emph{flat replacement}. For any $A$ and $W$ as in Theorem~\ref{thm:Adams_replacement}, we call the functor $F(W)$ given by the composite~\eqref{eqn:Adams_replacement_from_W} a flat replacement of $F$.
 \end{dfn}

 An important example of an flat replacement due to Pstr\k{a}gowski is the following.
 
 \begin{thm}\label{thm:complex_cobordism}
 Complex cobordism is a flat replacement of singular homology with integral coefficients. In fact, the site associated to complex cobordism is equal to the site associated to singular homology with integral coefficients.
 \end{thm}
 
 \begin{proof}
 In this case, the category $\cell(\ca{P})$ is the category of finite spectra. By applying a sufficiently high suspension, we can work with finite $\mathrm{CW}$-complexes $X$. That $\mathrm{MU}_{\ast} X$ is projective if and only if $\mathrm{H}(X; \mathbb{Z})$ is projective is proved in \cite[Corollary 3.10]{CONNER_SMITH}.
 
 From \cite[Corollary~1.5]{MILNOR_MOORE} it follows that $\mathbb{Z} \ten{\mathrm{MU}_{\ast}} \mathrm{MU}_{\ast} f$ is an epimorphism if and only if $\mathrm{MU}_{\ast} f$ is an epimorphism. From \cite[Lemma~3.1]{CONNER_SMITH} we know that there is a natural isomorphism $\mathbb{Z} \ten{\mathrm{MU}_{\ast}} \mathrm{MU}_{\ast} X \cong \mathrm{H}_{\ast}(X;\mathbb{Z})$. Thus the notion of $\mathrm{MU}_{\ast}$-epimorphism coincides with the notion of $\mathrm{H}_{\ast}(-;\mathbb{Z})$-epimorphism. 
 \end{proof}
 
 The notion of flat replacements is particularly useful in contexts where Brown representability holds for coproduct preserving homological functors. Let $\ca{T}_{\ca{P}}$ be the smallest localizing subcategory of $\ca{T}$ which contains $Kp$ for all $p \in \ca{P}$. Moreover, we assume from now on that the objects $Kp$ are \emph{compact}, that is, that the functors $\ca{T}_{\ca{P}}(Kp,-)$ preserve coproducts for all $p \in \ca{P}$. It follows that all $A \in \cell(\ca{P})$ are compact as well.
 
 \begin{lemma}\label{lemma:homological_left_kan}
 Suppose $\ca{A}$ is a triangulated subcategory of $\ca{T}_{\ca{P}}$ which consists of compact objects and which contains $Kp$ for all $p \in \ca{P}$ and let $J \colon \ca{A} \rightarrow \ca{T}_{\ca{P}}$ be the inclusion. If $G \colon \ca{T}_{\ca{P}} \rightarrow \Ab$ is a coproduct preserving homological functor, then the identity of $GJ$ exhibits $G$ as left Kan extension $\Lan_J GJ$ of $GJ$ along $J$. 
 \end{lemma}
 
 \begin{proof}
 By definition of Kan extensions, the identity induces a natural transformation $\alpha \colon \Lan_J GJ \Rightarrow G$. The functor $GJ \colon \ca{A} \rightarrow \Ab$ is homological, hence flat (the category of elements of $GJ$ is cofiltered). Thus $- \ten{\ca{A}} G$ is exact and $\Lan_J GJ \cong \mathrm{Hom}_{\ca{A}}(J,-) \ten{\ca{A}} G$ is homological and coproduct preserving. Thus the subcategory of objects $X \in \ca{T}_{\ca{P}}$ for which $\alpha_X$ is an isomorphism is localizing. Since $J$ is full and faithful, this subcategory contains $\ca{A}$, so it must be equal to $\ca{T}_{\ca{P}}$
 \end{proof}
 
 Recall that an additive functor with values in $\Ab$ is called \emph{small} if it is isomorphic to the left Kan extension of a functor along the inclusion of a small subcategory. Equivalently, if it is a small colimit of representable functors. Note that $\cell(\ca{P})$ is a small triangulated subcategory of $\ca{T}_{\ca{P}}$ containing the image of $K$. It consists of compact objects since each $Kp$, $p \in \ca{P}$ is compact. The above lemma thus shows that all coproduct preserving homological functors $\ca{T}_{\ca{P}} \rightarrow \Ab$ are small. We thus have a functor
 \[
 F_B \colon \ca{T}_{\ca{P}} \rightarrow [\ca{T}_{\ca{P}}, \Ab]_{\mathrm{small}}, \quad X \mapsto \ca{T}_{\ca{P}}(\U, X \otimes -)
 \]
 which lands in the full subcategory of coproduct preserving homological functors. A morphism $p$ in $\ca{T}_{\ca{P}}$ is called a \emph{phantom} if $F_B p$ is zero. These form an additive ideal, so we get a new additive category $\ca{T}_{\ca{P}}\slash{\sim}$ by taking the hom-wise quotient. The category $\ca{T}_{\ca{P}}$ is a \emph{Brown category} if the functor
 \[
 F_{B}^{\prime} \colon \ca{T}_{\ca{P}} \slash{ \sim} \rightarrow [\ca{T}_{\ca{P}},\Ab]_{\mathrm{small}}
 \]
 induced by $B$ is full and faithful and its essential image consists of \emph{all} coproduct preserving homological functors. Note that this imposes a strong smallness condition on our category: this homological form of Brown representability works if there is a full subcategory $\ca{G}$ generating $\ca{T}_{\ca{P}}$ which is countable (meaning that there is only a countable set of morphisms in $\ca{G}$). Nevertheless, this holds in many examples we care about, namely the stable homotopy category itself \cite{ADAMS_BROWN_REPRESENTABILITY}, categories of genuine $G$-spectra for compact Lie groups (\cite[Corollary~9.4.4]{HOVEY_PALMIERI_STRICKLAND}) and the stable $\mathbb{A}^1$-homotopy category of suitable schemes (\cite{NAUMANN_SPITZWECK}). For general results, see \cite{NEEMAN} \cite[Theorem~4.1.5]{HOVEY_PALMIERI_STRICKLAND}.
 
 We will need to use some compatibility with the monoidal structure. First note that the Day convolution induces a tensor product on $[\ca{T}_{\ca{P}},\Ab]_{\mathrm{small}}$. Indeed, by the discussion at the beginning of \cite[\S 7]{DAY_LACK}, this is the case if the functors defining the promonoidal structure on the domain are small; in our case, we are dealing with a monoidal category, so the functors in question are given by $\ca{T}_{\ca{P}}(-,X \otimes Y)$ which are representable, hence small. The Day convolution tensor product can equivalently be described as the left Kan extension
 \[
 \xymatrix{\ca{T}_{\ca{P}} \otimes \ca{T}_{\ca{P}} \dtwocell\omit{^<-6.5>} \ar[d]_{F \otimes G} \ar[r]^-{\otimes}  & \ca{T}_{\ca{P}} \ar@{-->}[d]^{F \otimes_{\mathrm{Day}} G}  \\ \Ab \otimes \Ab \ar[r]_-{\ten{\mathbb{Z}}} & \Ab} 
 \]
 of additive functors. From this description it follows that the morphisms
 \[
 \ca{T}_{\ca{P}}( \U, X \otimes A ) \ten{\mathbb{Z}} \ca{T}_{\ca{P}}(\U, Y \otimes B) \rightarrow \ca{T}_{\ca{P}}(\U,X\otimes Y \otimes A \otimes B)
 \]
 given by the tensor product and symmetry in $\ca{T}$ endow the functor $F_B$ above with a normal symmetric lax monoidal structure. 
 
 From the definition of phantom morphisms, it follows that $X \otimes p$ is a phantom whenever $p$ is, so the symmetric monoidal structure of $\ca{T}_{\ca{P}}$ descends to the quotient category $\ca{T}_{\ca{P}} \slash{\sim}$. Using the notion of minimal weak colimits, one can show that the functor $F_{B}^{\prime}$ is in fact an equivalence of symmetric monoidal categories. Since this is not stated explicitly in \cite{HOVEY_PALMIERI_STRICKLAND}, we give a proof below. Recall that a cocone on a filtered diagram is called a \emph{minimal weak colimit} if it is a weak colimit and it is sent to a colimit cocone by the functor $F_B$ above.

 \begin{prop}\label{prop:Brown_functor_monoidal}
 If $\ca{T}_{\ca{P}}$ is a Brown category, then the functor
 \[
 F_{B}^{\prime} \colon \ca{T}_{\ca{P}} \slash{ \sim} \rightarrow [\ca{T}_{\ca{P}},\Ab]_{\mathrm{small} }
 \]
 is symmetric strong monoidal.
 \end{prop}
 
 \begin{proof}
 Note that there are no non-zero phantom morphisms $X \rightarrow Y$ if the domain $X$ lies in $\cell(\ca{P})$, so we have a natural inclusion $\cell(\ca{P}) \rightarrow \ca{T}_{\ca{P}} \slash{\sim}$.
 
 Combining the definition of Brown categories with Lemma~\ref{lemma:homological_left_kan}, it follows that the composite of $F_{B}^{\prime}$ with the functor $[\ca{T}_{\ca{P}},\Ab]_{\mathrm{small}} \rightarrow [\cell(\ca{P}),\Ab]$ given by restriction along $\cell(\ca{P}) \rightarrow \ca{T}_{\ca{P}}$ gives an equivalence between $\ca{T}_{\ca{P}}\slash{\sim}$ and the category of \emph{all} homological functors $\cell(\ca{P}) \rightarrow \Ab$. The latter is well-known to be equivalent to the category of flat functors, so there exists an equivalence
 \[
 \ca{T}_{\ca{P}} \slash{\sim} \simeq \Ind\bigl(\cell(\ca{P})^{\op} \bigr)
 \]
 which restricts to the Yoneda embedding along the inclusion $\cell(\ca{P}) \rightarrow \ca{T}_{\ca{P}} \slash{\sim}$. Moreover, by \cite[Proposition~2.2.4.(e)]{HOVEY_PALMIERI_STRICKLAND}, minimal weak colimits are preserved by $-\otimes Y$ for all $Y \in \ca{T}_{\ca{P}}$. Since every filtered diagram in $\cell(\ca{P})$ has a minimal weak colimit in a Brown category (\cite[Theorem~4.2.3]{HOVEY_PALMIERI_STRICKLAND}), the above equivalence reduces the problem to checking that the restriction of $F_{B}^{\prime}$ to $\cell(\ca{P})$ is strong monoidal.
 
 This restriction is naturally isomorphic to $A \mapsto \ca{T}_{\ca{P}}(A^{\vee},-)$. Thus the Kan extension defining the Day convolution tensor product of $B^{\prime} A_0$ and $B^{\prime} A_1$ is the left Kan extension of
 \[
 \ca{T}_{\ca{P}} \otimes \ca{T}_{\ca{P}}\bigl((A_0^{\vee},A_1^{\vee}),-\bigr) \colon \ca{T}_{\ca{P}} \otimes \ca{T}_{\ca{P}} \rightarrow \Ab
 \]
 along $\otimes \colon \ca{T}_{\ca{P}} \otimes \ca{T}_{\ca{P}} \rightarrow \ca{T}_{\ca{P}}$. Since left Kan extensions preserve representables, the Day convolution is naturally isomorphic to $\ca{T}_{\ca{P}}\bigl((A_0 \otimes A_1)^{\vee},-\bigr)$. Using this, one can check that $\varphi_{A_0,A_1}$ is indeed an isomorphism for $A_0, A_1 \in \cell(\ca{P})$.
 \end{proof}
 
 A \emph{commutative quasi-ring} in $\ca{T}_{\ca{P}}$ is a commutative algebra in $\ca{T}_{\ca{P}}\slash{\sim}$, that is, an object $E \in \ca{T}_{\ca{P}}$, together with a multiplication $\mu \colon E \otimes E \rightarrow E$ and a unit $\eta \colon \U \rightarrow E$ which is associative and unital up to phantom morphism.
 
 \begin{cor}\label{cor:monoidal_Brown_representability}
 Suppose that $\ca{T}_{\ca{P}}$ is a Brown category. Then the functor
 \[
 \CAlg(\ca{T}_{\ca{P}} \slash{\sim} ) \rightarrow \slm_{\Ab}(\ca{T}_{\ca{P}},\Ab)
 \]
 which sends a commutative quasi-ring $(E,\mu,\eta)$ to the functor $\ca{T}_{\ca{P}}(\U,E\otimes -)$ with lax monoidal structure $\varphi_{X,Y}$ given by the morphism
 \[
 \ca{T}_{\ca{P}}(\U,E\otimes X) \otimes  \ca{T}_{\ca{P}}(\U,E\otimes Y) \rightarrow  \ca{T}_{\ca{P}}(\U,E\otimes X \otimes Y)
 \]
 which sends $\alpha \otimes \beta$ to the composite
 \[
 \xymatrix@C=30pt{ \U \ar[r]^-{\alpha \otimes \beta} & E \otimes X \otimes E \otimes Y \ar[r]^-{\cong} &  E\otimes E \otimes X \otimes Y \ar[r]^-{\mu \otimes X \otimes Y} & E \otimes X \otimes Y}
 \]
and $\varphi_0$ given by $\eta$ is full and faithful. Its image consists of the symmetric lax monoidal additive functors which are homological and which preserve coproducts.
 \end{cor}
 
 \begin{proof}
 Since the full and faithful functor $F_{B}^{\prime}$ is strong monoidal by Proposition~\ref{prop:Brown_functor_monoidal}, it lifts to a full and faithful functor of commutative algebras. From the definition of the Day convolution tensor product as left Kan extension, it follows that the commutative algebras for the Day convolution are precisely the symmetric lax monoidal additive functors. The explicit description of the lax monoidal structure of $\ca{T}_{\ca{P}}(\U,E\otimes -)$ is immediate from the lax monoidal structure of the functor $F_B \colon \ca{T}_{\ca{P}} \rightarrow [\ca{T}_{\ca{P}},\Ab]_{\mathrm{small}}$. 
 \end{proof}
 
 \begin{dfn}
 We say that a commutative quasi-ring $(E,\mu,\eta)$ in $\ca{T}_{\ca{P}}$ is $\ca{T}$-flat relative to $K \colon \ca{P} \rightarrow \ca{T}$ if the symmetric lax monoidal homological functor
 \[
 \ca{T}_{\ca{P}}(\U,E\otimes -) \colon \ca{T}_{\ca{P}} \rightarrow \Ab
 \]
 is $\ca{T}$-flat relative to $K$ in the sense of Definition~\ref{dfn:homology_theory_of_Adams_type}.
 \end{dfn}
 
 The following proposition shows that a homotopy commutative ring spectrum (a commutative ring object in the stable homotopy category $\mathrm{SH}$) is $\mathrm{SH}$-flat relative to the functor $K$ induced from the Picard group of the stable homotopy category if and only if it is topologically flat in the sense of \cite[Definition~1.4.5]{HOVEY}.
 
 \begin{prop}\label{prop:Adams_quasi_ring_characterization}
 If $\ca{T}_{\ca{P}}$ is a Brown category, then a quasi-ring $E$ in $\ca{T}_{\ca{P}}$ is $\ca{T}$-flat relative to $K$ if and only if $E$ is a minimal weak colimit of a filtered diagram of objects $E_i \in \cell(\ca{P})$ with the property that $ \ca{T}_{\ca{P}}(K-,E\otimes E_i) \in \ca{K}_A $
 is finitely generated projective, where $A=\ca{T}(K-,E)$ denotes the coefficient ring of the symmetric lax monoidal functor $\ca{T}_{\ca{P}} \rightarrow \ca{K}$ which sends $X$ to $\ca{T}_{\ca{P}}(K-,E\otimes X)$.
 \end{prop}
 
 \begin{proof}
 First assume that $E$ is $\ca{T}$-flat relative to $K$. With $F = \ca{T}_{\ca{P}}(\U,E\otimes -)$, $H = \bar{F}^{\mathrm{norm}}$, the category $\ca{A}_H$ is the full subcategory of $\cell(\ca{P})$ of objects $Y$ with the property that $\ca{T}_{\ca{P}}(K-,E\otimes Y) \in \ca{K}_A$ is finitely generated and projective. Moreover, the restriction of $F$ to $\ca{A}_H$ is isomorphic to $\colim \ca{A}_H(E_i^{\vee},-)$ for some filtered diagram $i \mapsto E_i \in \ca{A}_H$ by Proposition~\ref{prop:Adams_characterization}~(iii).
  
As in the proof of Lemma~\ref{lemma:homological_left_kan}, the flatness of the restriction of $F$ to $\ca{A}_H$ implies that $F$ is naturally isomorphic to the left Kan extension of this restriction along the  inclusion $J \colon \ca{A}_H \rightarrow \ca{T}_{\ca{P}} $. Since left Kan extensions preserve representables and colimits, $\ca{T}_{\ca{P}}(\U,E\otimes -)$ is naturally isomorphic to $\colim_{i} \ca{T}_{\ca{P}}(\U,E_i \otimes -)$. The functor $F_B \colon \ca{T}_{P} \rightarrow [\ca{T}_{\ca{P}},\Ab]_{\mathrm{small}}$ is full and there are no non-zero phantom morphisms with domain $E_i$, so the colimit cocone is the image of a cocone $E_i \rightarrow E$ in $\ca{T}_{\ca{P}}$. The comparison morphism from any minimal weak colimit of the $E_i$ is thus sent to an isomorphism by $F_B$, so it is an isomorphism since $F_B$ is conservative by \cite[Theorem~4.1.8]{HOVEY_PALMIERI_STRICKLAND} (where $F_B$ is denoted by $V_{\bullet}$). Thus $E$ is indeed a minimal weak colimit of the objects $E_i$.
 
 The converse follows directly from the definition of minimal weak colimits and Proposition~\ref{prop:Adams_characterization}~(iii).
\end{proof}

 \begin{cor}\label{cor:Adams_replacement_for_Brown_categories}
 Assume that $\ca{T}_{\ca{P}}$ is a Brown category. If $F \colon \ca{T} \rightarrow \Ab$ has a flat replacement $F(W) \colon \ca{T} \rightarrow \Ab$, then there exists a $\ca{T}$-flat commutative quasi-ring $E \in \ca{T}_{\ca{P}}$ (relative to $K$) such that the restriction of $F(W)$ to $\ca{T}_{\ca{P}}$ is naturally isomorphic to $\ca{T}_{\ca{P}}(\U,E \otimes-)$.
 \end{cor}
 
 \begin{proof}
 From Theorem~\ref{thm:Adams_replacement} we know that $F(W)$ is $\ca{T}$-flat relative to $K$, so the same is true for its restriction to $\ca{T}_{\ca{P}}$ (by Proposition~\ref{prop:Adams_characterization}). 
 
 The assumption that each $Kp$ is compact implies that $\cell(\ca{P})$ consists of compact objects, from which it follows that the restriction of $\widetilde{J}$ to $\ca{T}_{\ca{P}}$ preserves coproducts. Since all the other functors appearing in the composite~\eqref{eqn:Adams_replacement_from_W} preserve coproducts, it follows that the restriction of $F(W)$ to $\ca{T}_{\ca{P}}$ preserves coproducts. By Corollary~\ref{cor:monoidal_Brown_representability}, this restriction is of the form $\ca{T}_{\ca{P}}(\U,E\otimes -)$ for a commutative quasi-ring $E$.
 \end{proof}

%% file: recognition.tex
\section{Locally free objects}

\subsection{Preliminaries}\label{section:grading}
 In Section~\ref{section:Adams_replacement}, the following question arose naturally: does there exist a faithful, exact, symmetric strong monoidal left adjoint $W \colon \sh(\ca{A}_H) \rightarrow \ca{K}_A$ which is compatible with the $\ca{K}$-enrichment? In this section, we develop some techniques to address this question for particular categories $\ca{P}$ and $\ca{K}=[\ca{P}^{\op},\Ab]$.
 
 Throughout this section, we fix a pair $(G,\varepsilon)$ of an abelian group $G$ and a homomorphism $\varepsilon \colon G \rightarrow \{1,-1 \}$. We call the elements $i \in G$ with $\varepsilon(i)=1$ \emph{even} and the other elements \emph{odd}. The category $\ca{P}$ is the free $\Ab$-enriched category on the discrete category $G$. It is strict monoidal via addition in $G$. The symmetry isomorphism is given by the Koszul sign rules, that is, the symmetry isomorphism $i+j \rightarrow j+i$ is given by $-1$ if both $i$ and $j$ are odd, and by the identity otherwise.

 The category $\ca{K}=[\ca{P}^{\op},\Ab]$ is the category of $G$-graded abelian groups. We denote it by $(G,\varepsilon)\mbox{-}\Ab$. A Grothendieck tensor category $\ca{C}$ equipped with a symmetric strong monoidal functor $K \colon \ca{P} \rightarrow \ca{C}$ is called a \emph{$(G,\varepsilon)$-graded Grothendieck tensor category}. We use the notation $L_i \defl Ki$ since these objects are reminiscent of line bundles. Given a second $(G,\varepsilon)$-graded tensor category $\ca{D}$, a \emph{graded tensor functor} $\ca{C} \rightarrow \ca{D}$ is a symmetric strong monoidal left adjoint $F \colon \ca{C} \rightarrow \ca{D}$ with a symmetric monoidal isomorphism $F(L_i) \cong L_i$. Under the equivalence of Proposition~\ref{prop:strong_monoidal_enriched_equivalent_to_slice}, these categories and functors correspond to categories and functors enriched in $(G,\varepsilon)\mbox{-}\Ab$.
 
 In order to show that flat replacements exist, we need to show that there exists a commutative algebra $A \in (G,\varepsilon)\mbox{-}\Ab$ and a symmetric strong monoidal left adjoint
 \[
 W \colon \ca{C} \rightarrow (G,\varepsilon)\mbox{-}\Ab_{A}
 \]
 which is faithful, exact, and compatible with the enrichment. This is precisely the existence problem for fiber functors in (enriched) Tannaka duality.

\subsection{Recollections about Adams algebras}

 The key idea in Deligne's existence theorem for fiber functors in \cite{DELIGNE} and its generalizations in \cite{DELIGNE_TENSORIELLES, SCHAEPPI_GEOMETRIC, SCHAEPPI_COLIMITS} is the construction of a faithfully flat algebra $B$ in $\ca{C}$ such that the unit of $\ca{C}_B$ is a projective generator (which immediatley implies the desired equivalence $\ca{C}_B$ with the category of modules over the ring of endomorphisms of the unit object $B$ of $\ca{C}_B$). Thus it is crucial to have good criteria for an algebra to be faithfully flat. The class of \emph{Adams algebras} introduced in \cite[\S 4]{SCHAEPPI_COLIMITS} has several convenient properties which facilitate this. Let $\ca{C}^{\mathrm{dual}}$ be the full category of $\ca{C}$ consisting of objects with a dual. 
 
 \begin{dfn}\label{dfn:Adams_algebra}
  A commutative algebra $(A,\mu,\eta)$ in $\ca{C}$ with unit $\eta \colon \U \rightarrow A$ is called an \emph{Adams algebra} if the following conditions are satisfied.
 
\begin{enumerate}
\item[(i)]
 There exist a filtered diagram $\ca{I} \rightarrow \ca{C}^{\mathrm{dual}}$, $j \mapsto A_j$ whose colimit in $\ca{C}$ is the underlying object $A$ of the algebra $(A,\mu,\eta)$;
 \item[(ii)]
 In addition, there exists a cone $\eta_j \colon \U \rightarrow A_j $ whose colimit is $\eta \colon \U \rightarrow A$;
 \item[(iii)]
 For each object $j \in \ca{I}$, the dual morphism $\eta_j^{\vee} \colon A_j^{\vee} \rightarrow \U^{\vee} \cong \U$ of the cone in~(ii) is an epimorphism.
\end{enumerate} 

\end{dfn} 
  
  The key result about Adams algebras that we need is the following.
 
 \begin{prop}\label{prop:Adams_iff_faithfully_flat}
 Let $\ca{C}$ be a symmetric monoidal closed locally finitely presentable abelian category with finitely presentable unit object. Let $(A,\mu,\eta)$ be a commutative algebra which satisfies condition (i) above. Then $A$ is an Adams algebra if and only if the base change functor
 \[
 (-)_A \colon \ca{C} \rightarrow \ca{C}_A, \quad C \mapsto C_A \defl A \otimes C
 \]
 is faithful and left exact. 
 \end{prop}

\begin{proof}
 From \cite[Proposition~4.1.2]{SCHAEPPI_COLIMITS} it follows that a faithfully flat algebra is Adams if and only if its underlying object is a filtered colimit of objects with a dual. Conversely, if $\ca{C}$ is a symmetric monoidal closed abelian category with exact filtered colimits, then every Adams algebra is faithfully flat by \cite[Proposition~4.1.6]{SCHAEPPI_COLIMITS}.  
\end{proof} 
 
 In particular, any symmetric strong monoidal functor between such categories which detects epimorphisms between objects with duals also detects whether or not an algebra satisfying (i) is faithfully flat. Thus it is useful to know some stability properties for algebras satisfying (i). Note that this condition does not involve the algebra structure in any way, so we only need to study filtered colimits in locally finitely presentable categories. Note that duals are finitely presentable if the unit is finitely presentable, so the following proposition is applicable.
 
 Recall that the \emph{closure} of a full subcategory $\ca{A}$ of $\ca{C}$ under filtered colimits is the smallest subcategory which contains $\ca{A}$ and is closed under filtered colimits. In general, the construction of such a category is very complicated since the closure has to be iterated, possibly even transfinitely. One starts by setting $\ca{C}_0= \ca{A}$ and then one takes for $\ca{C}_{i+1}$ all the colimits (in $\ca{C}$) of filtered diagrams in $\ca{C}_{i}$ (successor step), and one takes the union of the previously constructed subcategories for limit ordinals. However, in many categories this procedure terminates after just one step. 
 
 \begin{prop}\label{prop:ind_objects}
 Let $\ca{C}$ be a locally finitely presentable category and let $\ca{A}$ be a full subcategory of the category $\ca{C}_{\fp}$ of finitely presentable objects. Fix an object $C \in \ca{C}$. Then the following are equivalent:
 \begin{enumerate}
 \item[(a)] The object $C$ lies in the closure of $\ca{A}$ under filtered colimits;
 \item[(b)] There exists a filtered category $\ca{I}$ and a diagram $D \colon \ca{I} \rightarrow \ca{A}$ whose colimit in $\ca{C}$ is isomorphic to $C$.
 \end{enumerate}
 In other words, the closure of $\ca{A}$ under filtered colimits in $\ca{C}$ terminates after one step. Moreover, in this case, the canonical functor $\Ind(\ca{A}) \rightarrow \ca{C}$ induced by the inclusion of $\ca{A}$ is full and faithful and its essential image is the closure of $\ca{A}$ under filtered colimits.
 \end{prop}
 
 \begin{proof}
 We first prove the second claim. Recall that the free cocompletion $\Ind(\ca{A})$ can be constructed as the closure of the representable presheaves in $[\ca{A}^{\op},\Set]$ under filtered colimits. Let $K \colon \ca{A} \rightarrow \ca{C}$ denote the inclusion. From the universal property of free cocompletions we get the canonical functor 
 \[
 L_K \colon \Ind(\ca{A}) \rightarrow \ca{C}
 \] 
 which preserves filtered colimits and whose restriction along the Yoneda embedding is the inclusion $K$. From the assumption we thus know that the functor $L_K$ sends representable objects to finitely presentable objects in $\ca{C}$. Thus the subcategory of $\Ind(\ca{A})$ of objects $F$ for which the function
 \[
 (L_K)_{\ca{A}(-,A),F} \colon \Ind(\ca{A})\bigl(\ca{A}(-,A) ,F\bigr) \rightarrow \ca{C}\bigl(L_K \ca{A}(-,A),L_K F\bigr )
 \]
 on hom-sets is a bijection is closed under filtered colimits and contains the representable presheaves, so it must be all of $\Ind(\ca{A})$. Since contravariant representable functors send colimits to limits, it follows that $L_K \colon \Ind(\ca{A}) \rightarrow \ca{C}$ is full and faithful. Thus the essential image is in particular closed under filtered colimits. It contains $\ca{A}$ by construction, so this image is precisely the closure of $\ca{A}$ under filtered colimits.

 The non-trivial implication $(a) \Rightarrow (b)$ is thus reduced to the analogous claim that each object in $\Ind(\ca{A})$ is a filtered colimit of representable presheaves. This is a well known result: a presheaf $F$ lies in $\Ind(\ca{A})$ if and only if the category of elements $\el(F)$ is cofiltered, and any presheaf is a colimit of the canonical diagram
 \[
 \el(F)^{\op} \rightarrow [\ca{A}^{\op},\Set]
 \]
 which sends $(A,x \in FA)$ to $\ca{A}(-,A)$ and a morphism $f \colon (A,x) \rightarrow (A^{\prime},x^{\prime})$ (that is, a morphism $f \colon A^{\prime} \rightarrow A$ in $\ca{A}$ with $Ff(x)=x^{\prime}$) to $\ca{A}(-,f)$. 
 \end{proof}

 \begin{cor}
 Let $\ca{C}$ be a symmetric monoidal locally finitely presentable abelian category with finitely presentable unit object. A commutative algebra $(A,\mu,\eta)$ in $\ca{C}$ satisfies Condition~(i) of Definition~\ref{dfn:Adams_algebra} if and only if it lies in the full subcategory $\Ind(\ca{C}^{\mathrm{dual}})$.
 \end{cor}

\begin{proof}
 This follows from Proposition~\ref{prop:ind_objects} applied to the case $\ca{A}=\ca{C}^{\mathrm{dual}}$.
\end{proof} 
 
 The lfp abelian tensor we care about often have a generating set consisting of objects with duals. One can deduce a lot of information about a graded tensor functor between two such categories from its behaviour on objects with duals. For example, a left adjoint functor $F \colon \ca{C} \rightarrow \ca{D}$ which detects epimorphisms in $\ca{C}$ whose domain and codomain have a dual automatically detects epimorphisms between finitely presentable objects (since duals are finitely presentable, this follows from \cite[Lemma~3.2.6]{SCHAEPPI_DESCENT}. Such a functor need not be faithful nor exact, and is therefore easier to construct. An example of a functor of this kind worth keeping in mind is the base change $R \slash \mathfrak{m} \ten{R} -$ to the residue field for a commutative local ring $(R,\mathfrak{m})$ (by the Nakayama lemma). Note that in this case, there are many non-zero objects $C \in \ca{C}$ with $FC \cong 0$. Since these functors play an important role in the next section we introduce a short name for them.
 
 \begin{dfn}
 Let $\ca{C}$ be a (graded) Grothendieck tensor category. A (graded) symmetric strong monoidal functor $F \colon \ca{C} \rightarrow \ca{D}$ is called a \emph{Nakayama functor} if it detects epimorphisms in $\ca{C}$ whose domain and codomain have a dual (equivalently, whose domain and codomain are finitely presentable).
\end{dfn} 
 
 We can thus summarize the main results of this section as follows.
 
 \begin{thm}\label{thm:Nakayama_detection}
  If $F \colon \ca{C} \rightarrow \ca{D}$ is a (graded) Nakayama functor between (graded) Grothendieck tensor categories, then a commutative algebra $(A,\mu,\eta)$ in $\ca{C}$ whose underlying object lies in the closure of $\ca{C}^{\mathrm{dual}}$ under filtered colimits in $\ca{C}$ is faithfully flat if and only if the commutative algebra $FA$ in $\ca{D}$ is faithfully flat. This is the case if and only if both $A$ and $FA$ are Adams algebras.
 \end{thm}
 
 \begin{proof}
 This follows directly from Propostion~\ref{prop:Adams_iff_faithfully_flat} and Proposition~\ref{prop:ind_objects} and the assumption that $(\eta_i)^{\vee} \colon A_i^{\vee} \rightarrow \U$ is an epimorphism whenever $F(\eta_i^{\vee})$ is an epimorphism.
 \end{proof}

\subsection{Recognizing locally free objects}

 A \emph{type} is an element $(i_1 , \ldots , i_n) \in G^{n}$. We call a type \emph{even} if $\varepsilon(i_k)=1$ for all $k=1,\ldots,n$ and we call it \emph{odd} if $\varepsilon(i_k)=-1$ for all $k=1,\ldots,n$. We call even and odd types \emph{unmixed} and all other types \emph{mixed}. Throughout this section, we fix a $(G,\varepsilon)$-graded Grothendieck tensor category $\ca{C}$.
 
 \begin{dfn}
 Let $(i_1 ,\ldots i_n) \in G^n$. An object $M \in \ca{C}$ is called \emph{locally free of type $(i_1, \ldots i_k)$} if there exists an Adams algebra $A \in \ca{C}$ and an isomorphism $M_{A} \cong \oplus_{k=1}^n L_{i_k}$.
 \end{dfn}
 
 A key idea in Deligne's proof of the existence of fiber functors is to study the universal algebra which splits off a copy of the unit object as a direct summand. This can be adapted to the case of splitting off invertible objects. To do this, we will need the free commutative algebra $\Sym(M)$ on an object. Note that, in the graded case, this is not quite a polynomial algebra, for example, in the case of the category of graded modules, it will often be tensor product of exterior and polynomial algebras. It exhibits particularly strange behaviour if $2$ is not invertible in the endomorphism ring of the unit, for then the higher exterior powers of the unit itself do not vanish. Thus ``graded'' geometry is a bit peculiar in this case, since a line bundle need not be flat. Our conditions will have to exclude this kind of behaviour.
 
 Since $\Sym(M)$ is the free commutative algebra on an object, it is given by the graded algebra $\oplus_{i \in \mathbb{N}} \Sym^i(M)$, where $\Sym^i(M)$ denotes the quotient of the natural action of the symmetric group on the $n$-fold tensor product $M^{\otimes n}$. Note that, being left adjoint, it sends direct sums to coproducts of commutative algebras, that is, to their tensor product. Thus the crucial formula
 \[
 \Sym^n(M\oplus N) \cong \oplus_{i+j=n} \Sym^i(M) \otimes \Sym^j(N)
 \]
 follows directly from the universal property as in the ordinary (ungraded) case.
 
 Delignes construction can be generalized as follows. Fix an element $i \in G$. For an object $M$ with a dual, we can form the pushout
 \[
 \xymatrix{\Sym(\U) \ar[d] \ar[r] & \Sym\bigl(L_i^{-1} \otimes M \oplus (L_{i}^{-1} \otimes M)^{\vee}\bigr)  \ar[d]\\ \U \ar[r] & A(M,i) }
 \]
 in the category of commutative algebras in $\ca{C}$, where the top morphism is induced (via universal property) by the composite
 \[
 \xymatrix@C=40pt{\U \ar[r]^-{\coev} & L_i^{-1} \otimes M^{\vee}\otimes (L_i^{-1} \otimes M)^{\vee} \ar[r]^-{\mathrm{inclusion}} & \Sym\bigl(L_i^{-1} \otimes M \oplus (L_i^{-1} \otimes M)^{\vee}\bigr) }
 \]
 of the coevaluation and the natural inclusion $\Sym^{1} (L_i^{-1} \otimes M)^{\vee} \otimes \Sym^{1}  (L_i^{-1} \otimes M)^{\vee}$ in the symmetric algebra. The vertical arrow is the universal morphism which restricts to the identity on $\U$. By construction, the algebra $A(M,i)$ represents the functor which splits off $L_i$ as a direct summand: indeed, to give a homomorphism of commutative algebras $A(M,i) \rightarrow B$ in $\ca{C}$ amounts to a choice of morphisms of $B$-modules $p \colon M_B \rightarrow L_i$ and $s \colon L_i \rightarrow M_B$ such that $ps=\id$. Since it is constructed as a pushout of free algebras, the algebra $A(M,i)$ is preserved by any graded symmetric strong monoidal functor $F \colon \ca{C} \rightarrow \ca{D}$ of graded lfp tensor categories: there is a natural isomorphism $F\bigl(A(M,i)\bigr)\cong A(FM,i)$ of commutative algebras in $\ca{D}$.
 
 We can iterate this construction as follows.
 
 \begin{dfn}
 Let $(i_1 ,\ldots i_n) \in G^n$ and let $M$ be an object of $\ca{C}^{\mathrm{dual}}$. Then we inductively define a sequence of commutative algebras 
 \[
 \U \rightarrow A(M;i_1) \rightarrow A(M;i_1 ,i_2) \rightarrow \ldots \rightarrow A(M;i_1 ,\ldots,i_n)
 \]
 together with an epimorphism $p_n \colon M_{A(M;i_1,\ldots,i_n)} \rightarrow \oplus_{k=1}^n L_{i_k}$ in the category of $A(M;i_1,\ldots, i_n)$-modules and a section $s_n$ of $p_n$ as follows. We set $A(M;i_1)$ to be the algebra $A(M,i_1)$ defined by the pushout diagram above, and we let $p_1=p$, $s_1=s$. If $n>1$, we let $N_{n-1} \in \ca{C}_{A(M;i_1,\ldots i_{n-1})}$ be the kernel of $p_n$ in the category of $A(M;i_1, \ldots, i_{n-1})$-modules. We let $A(M; i_1, \ldots, i_n)$ be the universal algebra (in the cateogry of $A(M;i_1,\ldots, i_{n-1})$-modules) which splits off a copy of $L_{i_n}$ of the module $N_{n-1}$, that is, we set $A(M; i_1 ,\ldots, i_n) \defl A(N_{n-1},i_n)$. We use the base change of the split epimorphism $p_{n-1} \colon M_{A(M;i_1,\ldots, i_n)} \rightarrow \oplus_{k=1}^n L_{i_k}$ and the split epimorphism $p \colon (N_{n-1})_{A(N_{n-1},i_n)} \rightarrow L_n$ and the respective sections $s_{n-1}$ and $s$ to construct $p_n$ and $s_n$.
 \end{dfn}
 
 \begin{rmk}\label{rmk:split_unit}
 Suppose $B$ is a commutative algebra in $\ca{C}$ and $M$ is an object with a dual. If there exists a split epimorphism $q \colon M_B \rightarrow \oplus_{k=1}^n L_{i_k}$ in the category of $B$-modules, then we can inductively construct an algebra homomorphism $A(M;i_1 ,\ldots, i_n) \rightarrow B$. In particular, if $M$ itself has a direct summand isomorphic to $\oplus_{k=1}^n L_{i_k}$, then the unit of the algebra $A(M; i_1,\ldots, i_n)$ admits a retraction $A(M;i_1,\ldots, i_n) \rightarrow \U$.   
 \end{rmk}
 
 We now turn to the key result of this section.
 
 \begin{thm}\label{thm:iterated_splitting}
 Let $F \colon \ca{C} \rightarrow \ca{D}$ be a graded Nakayama functor between $(G,\varepsilon)$-graded Grothendieck tensor categories. Fix a type $(i_1, \ldots, i_n) \in G^n$. Let $M$ be an object with a dual and consider the commutative algebra $A \defl A(M;i_1, \ldots, i_n)$ in $\ca{C}$. Suppose that the following condition holds: for each $k \in \{1,\ldots,n\}$ and each $j \in \mathbb{N}$, the objects $\Sym^j(L_{i_k}^{-1} \otimes M )$ and $\Sym^j\bigl((L_{i_k}^{-1} \otimes M )^{\vee}\bigr)$ have a dual.
 
 Then $A$ is faithfully flat if and only if $FA$ is faithfully flat. In this case, it is an Adams algebra. In partiuclar, if the image $F(M)$ splits as a direct sum $FM \cong X \oplus L_{i_1} \oplus \ldots \oplus L_{i_n}$ in $\ca{D}$, then $A$ is an Adams algebra.
\end{thm}
 
 \begin{proof}
 We first show that under the stated assumption, the underlying object of the commutative algebra $A$ lies in the closure $\Ind(\ca{C}^{\mathrm{dual}})$ of $\ca{C}^{\mathrm{dual}}$ under filtered colimits in $\ca{C}$. Note that $\ca{C}^{\mathrm{dual}}$ is closed under finite tensor products and finite direct sums, so from Proposition~\ref{prop:ind_objects} and the fact that products of filtered categories are filtered we immediately see that $\Ind(\ca{C})$ is closed under finite tensor products and finite direct sums. Moreover, the category with one object with a single non-trivial idempotent on it is filtered, so  $\Ind(\ca{C})$ is closed under splitting of idempotents. In other words, it is closed under retracts. Since it is by definition closed under filtered colimits, it is also closed under the formation of arbitrary direct sums. Using these closure properties, we  can show by induction that the underlying object of the algebra $A$ lies in $\Ind(\ca{C}^{\mathrm{dual}})$.
 
 For ease of notation, we set $A_0 \defl \U$ and $A_n \defl A(M;i_1, \ldots, i_n)$. The base case $n=0$ is therefore vacuously true, so we assume $n>0$ and that the underlying object of $A_{n-1}$ lies in $\Ind(\ca{C}^{\mathrm{dual}})$. Let $N$ be the kernel of the split epimorphism $p_{n-1} \colon M_{A_{n-1}} \rightarrow \oplus_{k=1}^n L_{i_k}$. By definition, the algebra $A_n$ is given by the pushout diagram
 \[
 \xymatrix{ \Sym(A_n) \ar[d] \ar[r] & \Sym \bigl(L_{i_n}^{-1} \otimes N \oplus (L_{i_n}^{-1} \otimes N)^{\vee} \bigr) \ar[d] \\ A_{n-1} \ar[r] & A_n}
 \]
 in $\ca{C}_{A_{n-1}}$, where the vertical arrow corresponds to the identity on $N$ and the horizontal arrow corresponds to the coevaluation. The symmetric algebra $B \defl  \Sym \bigl(L_{i_n}^{-1} \otimes N \oplus (L_{i_n}^{-1} \otimes N)^{\vee} \bigr) $ is canonically $\mathbb{N}$-graded, and the coevaluation factors through the inclusion  $B_2 \rightarrow \oplus B_n$. Deligne has given a detailed description of this pushout in the case where the category $\ca{C}$ is the category of ind-objects of a \emph{rigid} abelian category. As already observed in \cite[\S 5]{SCHAEPPI_GEOMETRIC}, this construction works exactly the same in the considerably more general setting of Grothendieck tensor categories. Indeed, from \cite[Lemmas~5.13 and 5.14]{SCHAEPPI_GEOMETRIC} it follows that the underlying $A_{n-1}$-module of $A_n$ is given by the direct sum of the colimits of the two chains
 \[
 \xymatrix{\U \ar[r] & B_2 \ar[r] & B_4 \ar[r] & \ldots \ar[r] & B_{2k} \ar[r] & \ldots  }
 \]
 and
 \[
 \xymatrix{B_1 \ar[r] & B_3 \ar[r] & B_5 \ar[r] & \ldots \ar[r] & B_{2k+1} \ar[r] & \ldots  }
 \]
 respectively, where each connecting morphism $B_{j} \rightarrow B_{j+2}$ is given by the multiplication with the coevaluation (see \cite[Lemma~5.14]{SCHAEPPI_GEOMETRIC} for details). Since colimits of $\ca{C}_{A_{n-1}}$ are computed as in $\ca{C}$, this reduces the problem to checking that the object
 \[
 B_n= \Sym^n \bigl(L_{i_n}^{-1} \otimes N \oplus (L_{i_n}^{-1} \otimes N)^{\vee} \bigr) \cong \oplus_{i+j=n} \Sym^i(L_{i_n}^{-1} \otimes N) \otimes \Sym^j\bigl( (L_{i_n}^{-1} \otimes N)^{\vee}\bigr)
 \]
 lies in $\Ind(\ca{C}^{\mathrm{dual}})$. From the above mentioned closure properties, we immediately reduce the claim to checking that the two objects
 \[
 \Sym^j(L_{i_n}^{-1} \otimes N) \quad \text{and} \quad \Sym^j\bigl( (L_{i_n}^{-1} \otimes N)^{\vee}\bigr)
 \]
 lie in $\Ind(\ca{C}^{\mathrm{dual}})$ for all $j \in \mathbb{N}$. Note that $L_{i_n}^{-1} \otimes N$ is by definition a retract of $L_{i_n}^{-1} \otimes M_{A_{n-1}}$, so $(L_{i_n}^{-1} \otimes N)^{\vee}$ is a retract of $(L_{i_n}^{-1} \otimes M_{A_{n-1}})^{\vee}$. Since all functors preserve retracts, we have further reduced the problem to the claim that the two objects
 \[
  \Sym^j(L_{i_n}^{-1} \otimes M_{A_{n-1}}) \quad \text{and} \quad \Sym^j\bigl( (L_{i_n}^{-1} \otimes M_{A_{n-1}})^{\vee}\bigr)
 \]
 lie in $\Ind(\ca{C}^{\mathrm{dual}})$. The base change functor $(-)_{A_{n-1}} \colon \ca{C} \rightarrow \ca{C}_{A_{n-1}}$ preserves duals, gradings, and it commutes with the formation of symmetric powers since it preserves tensor products and colimits. Our assumption therefore implies that both objects are images under $(-)_{A_{n-1}}$ of objects with a dual in $\ca{C}$. We have reduced the claim to checking that, for any object $K \in \ca{C}^{\mathrm{dual}}$, the object $K \otimes A_{n-1}$ lies in $\ca{C}$. This finally follows from the inductive assumption that $A_{n-1}$ lies in $\Ind(\ca{C}^{\mathrm{dual}})$.
 
 From Theorem~\ref{thm:Nakayama_detection} it follows that the algebra $A =A(M;i_1, \ldots, i_n)$ is faithfully flat if and only if $F(A)$ is, and that it is an Adams algebra in this case. To see the final claim, note that $F(A)$ lies in $\Ind(\ca{D}^{\mathrm{dual}})$ since $F$ preserves duals and filtered colimits. It is thus Adams if and only if it is faithfully flat by Proposition~\ref{prop:Adams_iff_faithfully_flat}. It is flat since filtered colimits in $\ca{D}$ are exact, and by Remark~\ref{rmk:split_unit}, the unit of $F(A)$ is split, so $F(A)$ is clearly a faithfully flat algebra.
  \end{proof}
  
  This theorem gets us closer to a recognition theorem for locally free objects. Indeed, under the assumptions of the theorem, if $F(M) \cong \oplus_{k=1}^n L_{i_k}$, then there exists an Adams algebra $A \in \ca{C}$, an $A$-module $X$, and an isomorphism $M_A \cong X \oplus L_{i_1} \oplus \ldots \oplus L_{i_n}$ of $A$-modules. The assumptions, however, do not directly imply that $X$ has to be the zero object: while $X$ is finitely presentable as an $A$-module, its underlying object in $\ca{C}$ will in general be far from finitely presentable. Thus the fact that $FX \cong 0$ does not imply that $X\cong 0$.
  
 Here we use another idea of Deligne: it is sometimes possible to perform a finite construction using just the object $M$  in the base category $\ca{C}$, for example an exterior power, or a certain Schur functor, such that the base change of the construction detects if the object $X$ above is non-zero. The functor $F$ is then of course also required to detect whether or not the construction vanishes. The following proposition shows that the symmetrizer and the antisymmetrizer are candidates for such a construction. 
 
 \begin{prop}\label{prop:unmixed_locally_free_detection}
  Let $X$ be an object of $\ca{C}^{\mathrm{dual}}$ and $(i_1, \ldots, i_n) \in G^n$.  Let $M \defl L_{i_1} \oplus \ldots \oplus L_{i_n} \oplus X$. Then the identity of $L_{i_1} \otimes \ldots \otimes L_{i_n} \otimes X$ factors through both the symmetrizer $S \defl  \sum_{\sigma \in \Sigma_{n+1}} \sigma$ and the antisymmetrizer $A \defl \sum_{\sigma \in \Sigma_{n+1}} \sgn(\sigma) \sigma $ on $M^{\otimes (n+1)}$.
 \end{prop}
 
 \begin{proof}
 For $1 \leq k \leq n$, let $X_k \defl L_{i_k}$ and let $X_{n+1} \defl X$. We write $\alpha_k \colon X_k \rightarrow M$ for the inclusions of the direct sum and $\beta_k$ for the corresponding projections onto $X_k$. The symmetric group $\Sigma_{n+1}$ acts on the $(n+1)$-fold tensor product $M^{\otimes (n+1)}$. From the coherence axioms of symmetric monoidal categories, we know that the diagram
  \[
  \xymatrix{ M^{\otimes (n+1)} \ar[r]^-{\sigma} & M^{\otimes (n+1)}  \\ X_1 \otimes \ldots \otimes X_{n+1}  \ar[r]_-{\sigma} \ar[u]^{\alpha_1 \otimes \ldots \otimes \alpha_{n+1}}  & X_{\sigma(1)} \otimes \ldots \otimes X_{\sigma(n)} \ar[u]_{\alpha_{\sigma(1)} \otimes \ldots \otimes \alpha_{\sigma(n+1)}} }
  \]
 is commutative for each $\sigma \in \Sigma_{n+1}$. Thus the composite 
 \[
\xymatrix{ L_{i_1} \otimes \ldots \otimes L_{i_n} \otimes X \ar[d]^{\alpha_1 \otimes  \ldots \otimes  \alpha_{n+1}} \ar@{-->}[r] & L_{i_1} \otimes \ldots \otimes L_{i_n} \otimes X \\
M^{\otimes(n+1)}  \ar[r]_-{\sigma} & M^{\otimes (n+1)} \ar[u]_{\beta_1 \otimes \ldots \otimes \beta_{n+1}} & }  
 \]
 is equal to the identity if $\sigma=\id$ and zero if $\sigma \neq \id$. It follows that the identity morphism of $L_{i_1} \otimes \ldots \otimes L_{i_n} \otimes X$ factors through both the antisymmetrizer $A \defl \sum_{\sigma \in \Sigma_{n+1}} \sgn(\sigma) \sigma $ and the symmetrizer $S \defl  \sum_{\sigma \in \Sigma_{n+1}} \sigma$ on $M^{\otimes (n+1)}$.
 \end{proof}
 
 Thus if either the antisymmetrizer or the symmetrizer vanish, we have $X =0$. This will be the case for locally free objects which are purely even respectively purely odd. Dealing with locally free objects of mixed type is more complicated. In characteristic zero, Deligne has used Schur functors to solve the corresponding problem in the case of rigid categories.
 
 In \cite{COULEMBIER}, Coulembier claimed that this can be extended to positive characteristic. Unfortunately, there is a mistake in the proof of \cite[Theorem~5.2.2]{COULEMBIER}: the Schur functor in question need not vanish on the relevant super vector space in positive characteristic. Coulembier has provided a fix by working with Young symmetrizers instead of Schur functors (see Proposition~\ref{prop:non-zero_detection} below).
 
 We consider the rectangular Young tableau with $n$ rows and $m$ columns, labelled increasingly first from left to right an then from top to bottom. For example, the Young tableau
 \[
     \begin{ytableau}
    1 & 2 & 3 \\ 4 & 5 & 6
    \end{ytableau}
   \]
 is the special case $n=2$ and $m=3$. Let $X$ be an object of the symmetric monoidal category $\ca{C}$. The Young symmetrizer $S_X^{m,n}$ is the endomorphism of $X^{\otimes mn}$ which is given by the composite which symmetrizes with respect to permutations which fix rows and antisymmetrizes with respect to permutations which preserve columns. More explicitly, we write $\Sigma_{mn} \cap C$ and $\Sigma_{mn} \cap R$ for the permutations of $\{1,\ldots, mn\}$ which fix the columns respectively the rows of the tableau. For example, $\Sigma_{6} \cap R$ coincides with the standard inclusion of $\Sigma_3 \times \Sigma_3$ in $\Sigma_6$ in the case of the above diagram. With this notation, the morphism $S_X^{m,n}$ is given by the composite
 \begin{equation}\label{eqn:young_symmetrizer}
 S_X^{m,n} \defl \xymatrix@C=75pt{X^{ \otimes mn} \ar[r]^-{\sum_{\sigma \in \Sigma_{mn} \cap R} \sigma  } & X^{\otimes mn} \ar[r]^-{\sum_{\tau \in \Sigma_{mn} \cap C} \sgn(\tau) \tau}  & X^{\otimes mn}}
 \end{equation}
 in $\ca{C}$.
 
 The following result is due to Kevin Coulembier.
 
 \begin{prop}\label{prop:non-zero_detection}
 Let $\ca{C}$ be a $(G,\varepsilon)$-graded symmetric monoidal category with finite direct sums. Let $p,q \in \mathbb{N}$. Let $Y \in \ca{C}^{\mathrm{dual}}$ and let $(i_1, \ldots i_p) \in G_{\mathrm{even}}^{p}$ and $(i_{p+1}, \ldots, i_{p+q}) \in G_{\mathrm{odd}}^q$. If the Young symmetrizer
 \[
S_{ L_{i_1} \oplus \ldots \oplus L_{i_{p+q}} \oplus Y }^{(q+1,p+1)} \in \End_{\ca{C}} \bigl( ( L_{i_1} \oplus \ldots \oplus L_{i_{p+q}} \oplus Y )^{\otimes(p+1)(q+1)} \bigr)
 \]
 is zero, then $Y \cong 0$.
 \end{prop}
 
 \begin{proof}
 Let $X=L_{i_1} \oplus \ldots \oplus L_{i_{p+q}} \oplus Y$. We define the object $Y^{\prime}$ as a certain tensor product of $L_{i_k}$'s and $Y$. To specify this, we fill the first row of the rectangular Young diagram with the labels $L_{i_1}, \ldots, L_{i_{q+1}}$, the second row with $L_{i_2}, \ldots, L_{i_{q+2}}$, and so on up to row $p$. We fill the last row with the labels $L_{p+1}, \ldots, L_{p+q},Y$. Let $Y^{\prime}$ denote the tensor product of $(p+1)(q+1)$ factors, with factor at position $i$ given by the label of this Young diagram at position $i$.
 
 For example, for $p=2$, $q=3$, the labelled Young diagram is
 \[
   \begin{ytableau}
    L_{i_1} & L_{i_2} & L_{i_3} & L_{i_4} \\ L_{i_2} & L_{i_3} & L_{i_4} & L_{i_5} \\ L_{i_3} & L_{i_4} & L_{i_5} & Y
   \end{ytableau}
 \]
 and $Y^{\prime}=L_{i_1} \otimes L_{i_2} \otimes L_{i_3} \otimes L_{i_4} \otimes L_{i_2} \otimes L_{i_3} \otimes L_{i_4} \otimes L_{i_5} \otimes L_{i_3} \otimes L_{i_4} \otimes L_{i_5} \otimes Y$.
 
  We have an inclusion $Y^{\prime} \rightarrow X^{\otimes (p+1)(q+1)}$ and a projection $ X^{\otimes (p+1)(q+1)} \rightarrow Y^{\prime}$ given by the tensor product of the inclusions of $L_{i_k}$ and $Y$ in $X$ (respectively by the tensor product of the corresponding projections). We claim that the composite
 \[
 \xymatrix@C=40pt{ Y^{\prime} \ar[r]^-{\mathrm{incl}} & X^{\otimes (p+1)(q+1)} \ar[r]^-{S_X^{(q+1,p+1)}} & X^{\otimes (p+1)(q+1)} \ar[r]^-{\mathrm{proj}} & Y^{\prime} }
 \]
 is equal to the identity.
 
 For $\sigma \in \Sigma_{(p+1)(q+1)}$, the composite
 \[
  \xymatrix@C=40pt{ Y^{\prime} \ar[r]^-{\mathrm{incl}} & X^{\otimes (p+1)(q+1)} \ar[r]^-{\sigma} & X^{\otimes (p+1)(q+1)} \ar[r]^-{\mathrm{proj}} & Y^{\prime} }
 \]
 is non-zero only if $\sigma$ permutes all the boxes with label $L_{i_k}$ among themselves. In the example $p=2$, $q=3$, this means that $\sigma$ keeps $1$ and $12$ fixed and permutes the elements of the sets $\{2,5\}$, $\{3,6,9\}$, $\{4,7,10\}$, and $\{8,11\}$ among themselves. If $\sigma$ is one of the summands of $S^{(q+1,p+1)}_X$, then there exists a column permutation $\tau$ so that $\tau \circ \sigma$ is equal to a row permutation. We claim that this can only happen if $\sigma=\id$.
 
 We show by induction that such a permutation $\sigma$ keeps all the rows fixed, with base case the zeroth row being vacuously fixed. Thus we assume that the first $n$ rows are fixed and we label the elements of the $(n+1)$-st row  by $x_1,\ldots, x_{q+1}$. Since $\sigma$ can only send $x_1$ to one of the rows $1, \ldots, n$, the induction assumption implies that $x_1$ is fixed. Now assume that $\sigma$ fixes $x_1,\ldots, x_k$ for some $k \geq 1$.
 
 Again using the fact that all the elements in rows $1, \ldots, n$ are fixed, we find that $\sigma$ either fixes $x_{k+1}$ or it sends it to an element in one of the columns $1, \ldots, k$. Assume that $x_{k+1}$ is \emph{not} fixed.
 
 The composite $\tau \circ \sigma$ is a row permutation, so $\tau \circ \sigma$ must send $x_{k+1}$ to one of $x_1, \ldots, x_k$. On the other hand, since $\tau$ is a column permutation and $\sigma$ fixes $x_1, \ldots, x_k$, we must have $\tau(x_i)=x_i$ for $i=1,\ldots, k$ (else the composite $\tau \circ \sigma$ would not be a row permutation). Thus $\tau \circ \sigma$ sends two distinct elements to one of the $x_1, \ldots, x_k$, a contradiction. So $x_{k+1}$ is fixed by $\sigma$ and induction on $k$ shows that the $(n+1)$-st row is fixed by $\sigma$. Induction on $n$ shows that $\sigma=\id$, as claimed.
 
 The above argument shows that the identity on $Y^{\prime}$ factors through $S_X^{(q+1,p+1)}$. Thus if $S_X^{(q+1,p+1)}=0$, then we must have $Y^{\prime} \cong 0$. But $Y^{\prime}$ is $Y$ tensored with an invertible object, so $Y^{\prime} \cong 0$ implies $Y \cong 0$.
 \end{proof}

 \begin{rmk}
 Combining Proposition~\ref{prop:non-zero_detection} with Lemma~\ref{lemma:Young_symmetrizer_zero} shows that \cite[Theorem~5.2.2]{COULEMBIER} can be fixed by replacing the condition ``the Schur functor $\Gamma_{\lambda}(X)$ is zero'' with the condition ``the Young symmetrizer $S_X^{([X]_{\bar{\U}}+1, [X]_{\U}+1 )}$ is zero.''
 \end{rmk}

%% file: duals.tex
\section{Examples of homology theories with flat replacements}

\subsection{Detecting duals}\label{section:detecting_duals}

 In this section, we fix a $\otimes$-triangulated category $\ca{T}$ which is monoidal closed and whose internal hom preserves distinguished triangles. We let $H \colon \ca{T} \rightarrow \ca{K}_B$ be as in Definition~\ref{dfn:H-dual} and we write $\ca{A} \defl \ca{A}_H$ for the additive site of $H$-duals. By Part~(ii) of Proposition~\ref{prop:basic_properties_of_H-duals}, the restriction of $H$ to $\ca{A}$ is symmetric strong monoidal. Thus the induced functor $- \ten{\ca{A}} H \colon [\ca{A}^{\op},\Ab] \rightarrow \ca{K}_B$ is also symmetric strong monoidal. We write $L \colon [\ca{A}^{\op},\Ab] \rightarrow \sh(\ca{A})$ for the associated sheaf functor. The goal of this section is to prove the following theorem.
 
 \begin{thm}\label{thm:detecting_duals}
 If $Z \in [\ca{A}^{\op},\Ab]$ is finitely presentable and $Z \ten{\ca{A}} H$ has a dual, then $LZ \in \sh(\ca{A})$ has a dual.
 \end{thm}
 
 We prove this with a sequence of lemmas. Let $Z$ be a finitely presentable object of $[\ca{A}^{\op},\Ab]$ such that $Z \ten{\ca{A}} H$ has a dual. Pick a presentation
 \[
 \xymatrix{YA \ar[r]^-{Yf} & YB \ar@{->>}[r] & Z}
 \]
 with $A,B \in \ca{A}$. Then $Z \ten{\ca{A}} H$ is isomorphic to $\mathrm{coker}(Hf)$, so  $\mathrm{coker}(Hf)$ is finitely presentable and projective. By Part~(v) of Proposition~\ref{prop:basic_properties_of_H-duals}, the cofiber $C$ of $f$ lies in $\ca{A}$. It follows in particular that $Z$ is isomorphic to the image of $YB \rightarrow YC$.
 
 \begin{lemma}\label{lemma:dual_finitely_presentable_and_preserved}
 The internal hom $[Z,\U]$ is finitely presentable and it is preserved by $-\ten{\ca{A}} H \colon [\ca{A}^{\op},\Ab] \rightarrow \ca{K}_B$.
 \end{lemma}
 
 \begin{proof}
 Since the internal hom of $\ca{T}$ is compatible with the triangulated structure, the triangle
 \[
 \xymatrix{ YC^{\vee} \ar[r] & YB^{\vee} \ar[r]^-{f^{\vee}} & YA^{\vee} \ar[r] & YC^{\vee}[1]}
 \]
 is distinguished. The image of $YC^{\vee} \rightarrow YB^{\vee}$ is isomorphic to $[Z,\U]$ since the functor $[-,\U] \colon [\ca{A}^{\op},\Ab] \rightarrow [\ca{A}^{\op},\Ab]$ sends cokernels to kernels. This image is also isomorphic to the cokernel of $YA^{\vee}[-1] \rightarrow YC^{\vee}$ since the triangle
 \[
\xymatrix{YA^{\vee}[-1] \ar[r] & YC^{\vee} \ar[r] & YB^{\vee} \ar[r]^-{f^{\vee}} & YA^{\vee} } 
 \]
 is distinguished, so $[Z,\U]$ is finitely presentable. Since the homological functor $H$ sends these triangles to exact sequences, it follows that $[Z,\U]$ is sent to the kernel of $Hf^{\vee} \colon HB^{\vee} \rightarrow HA^{\vee}$, so it is preserved by $- \ten{\ca{A}} H$.
 \end{proof}
 
 In order to understand the behaviour of the internal hom $[Z,Z] \in [\ca{A}^{\op},\Ab]$, we need to use presentations
 \[
 \xymatrix{ YA_1 \ar[r]^-{Yf_1} \ar[d] & YB_1 \ar@{->>}[r] \ar[d] & Z_1 \ar[d] \\ YA_2 \ar[r]^-{Yf_2} & YB_2 \ar@{->>}[r] & Z_2}
 \]
 of morphisms $Z_1 \rightarrow Z_2$ such that $Z_i \ten{\ca{A}} H$ has a dual. By Part~(v) of Proposition~\ref{prop:basic_properties_of_H-duals}, the cofiber $C_i$ of $f_i$ lies in $\ca{A}$. Since the Yoneda embedding sends distinguished triangles in $\ca{A}$ to exact sequences, the morphism $Z_1 \rightarrow Z_2$ is isomorphic to the horizontal image factorization of the square
 \[
 \xymatrix{YB_1 \ar[r] \ar[d] & YC_1 \ar[d] \\ YB_2 \ar[r] & YC_2}
 \]
 in $[\ca{A}^{\op},\Ab]$. The following lemma gives an explicit description of the kernel of $Z_1 \rightarrow Z_2$.
 
 \begin{lemma}\label{lemma:kernel_description}
 Suppose that the fiber $F$ of the diagonal morphism $B_1 \rightarrow C_2$ lies in $\ca{A}$. Then the horizontal image factorizations
 \[
\vcenter{
 \xymatrix{YF \ar@{->>}[r] \ar[d] & Z_0 \ar@{ >->}[r] \ar@{ >->}[d] & YC_1 \ar@{=}[d] \\ YB_1 \ar@{->>}[r] \ar[d] & Z_1 \ar@{ >->}[r] \ar[d] & YC_1 \ar[d] \\ YB_2 \ar@{->>}[r] & Z_2 \ar@{ >->}[r] & YC_2}
 }
\quad \text{and} \quad
\vcenter{
 \xymatrix{HF \ar@{->>}[r] \ar[d] & J_0 \ar@{ >->}[r] \ar@{ >->}[d] & HC_1 \ar@{=}[d] \\ HB_1 \ar@{->>}[r] \ar[d] & J_1 \ar@{ >->}[r] \ar[d] & HC_1 \ar[d] \\ HB_2 \ar@{->>}[r] & J_2 \ar@{ >->}[r] & HC_2}
 }
 \]
 yield exact sequences
 \[
 \xymatrix{0 \ar[r] & Z_0 \ar@{ >->}[r] & Z_1 \ar[r] & Z_2 }
 \quad \text{and} \quad
  \xymatrix{0 \ar[r] & J_0 \ar@{ >->}[r] & J_1 \ar[r] & J_2 }
 \]
 respectively.
 \end{lemma}
 
 \begin{proof}
 The claim about the $Z_i$ is proved in \cite[Proposition~3.2.5]{VERDIER}. The only non-trivial assertion is exactness at $Z_1$ and at $J_1$. To see exactness at $Z_1$, let $D \in \ca{A}$ and consider an element $x \in Z_1(D)$. This is a morphism $x \colon D \rightarrow C_1$ which factors through $B_1 \rightarrow C_1$. If $x$ is sent to zero, the composite
 \[
 \xymatrix{D \ar[r] & B_1 \ar[r] & C_1 \ar[r] & C_2}
 \]
 is zero, so the morphism $D \rightarrow B_1$ factors through $F \rightarrow B_1$. This shows that $x$ lies in the image of the inclusion $Z_0 \rightarrow Z_1$.
 
 To see exactness at $J_1$, it is convenient to apply the Freyd--Mitchell embedding theorem to $\ca{K}_B$ so that we can use elements to do the diagram chase. If $x \in J_1$ is sent to zero in $J_2$, then any preimage $y \in HB_1$ of $x$ is sent to zero by the morphism $HB_1 \rightarrow HC_2$. Thus there exists $z \in HF$ such that $z \mapsto y$. The image of $z$ in $J_0$ is sent to $x$ by the inclusion $J_0 \rightarrow J_1$.
 \end{proof}
 
 \begin{lemma}\label{lemma:kernel_preserved}
 In the situation of Lemma~\ref{lemma:kernel_description}, suppose that in addition the fiber $F^{\prime}$ of $F \rightarrow C_1$ is an $H$-dual. Then $-\ten{\ca{A}} H \colon [\ca{A}^{\op},\Ab] \rightarrow \ca{K}_B$ preserves the kernel $Z_0$ of $Z_1 \rightarrow Z_2$: there are natural isomorphisms $Z_i \ten{\ca{A}} H \cong J_i$.
 \end{lemma}
 
 \begin{proof}
 The rows in the diagram
 \[
 \xymatrix{YF^{\prime} \ar[r] \ar[d] & YF \ar@{->>}[r] \ar[d] & Z_0 \ar[r] \ar[d] & 0 \\
 YA_1 \ar[r] \ar[d] & YB_1 \ar@{->>}[r] \ar[d] & Z_1 \ar[r] \ar[d] & 0  \\
 YA_2 \ar[r]  & YB_2 \ar@{->>}[r]  & Z_2 \ar[r]  & 0 }
 \]
 are exact, so the $Z_i$ are sent to the respective cokernels in $\ca{K}_B$. Since $H$ is homological, the rows of the diagram
 \[
  \xymatrix{HF^{\prime} \ar[r] \ar[d] & HF \ar@{->>}[r] \ar[d] & J_0 \ar[r] \ar[d] & 0 \\
 HA_1 \ar[r] \ar[d] & HB_1 \ar@{->>}[r] \ar[d] & J_1 \ar[r] \ar[d] & 0  \\
 HA_2 \ar[r]  & HB_2 \ar@{->>}[r]  & J_2 \ar[r]  & 0 }
 \]
 are exact. The claim follows from the natural isomorphism $HA \cong YA \ten{\ca{A}} H$.  
 \end{proof}

\begin{lemma}\label{lemma:internal_hom_preserved}
 The internal hom $[Z,Z]$ is finitely presentable and preserved by $- \ten{\ca{A}} H \colon [\ca{A}^{\op},\Ab] \rightarrow \ca{K}_B$.
\end{lemma}

\begin{proof}
 From the presentation
 \[
 \xymatrix{YA \ar[r]^-{Yf} & YB \ar@{->>}[r] & Z}
 \]
 we get a presentation
 \[
 \xymatrix{ Y(B^{\vee} \otimes A) \ar[r] \ar[d] & Y(B^{\vee} \otimes B) \ar@{->>}[r]  \ar[d] & YB^{\vee} \otimes Z \ar@{ >->}[r] \ar[d] & Y(B^{\vee} \otimes C) \ar[d] \\ 
 Y(A^{\vee} \otimes A) \ar[r] & Y(A^{\vee} \otimes B) \ar@{->>}[r] & YA^{\vee} \otimes Z \ar@{ >->}[r] & Y(A^{\vee} \otimes C) }
 \]
 of $YB^{\vee} \otimes Z \rightarrow YA^{\vee} \otimes Z$, which is isomorphic to $[YB,Z] \rightarrow [YA,Z]$.
 
 We claim that the diagonal morphism $H(B^{\vee} \otimes B) \rightarrow H(A^{\vee} \otimes C)$ factors as a split epimorphism followed by a split monomorphism. Indeed, $- \ten{\ca{A}} H$ sends the above image factorization to $[HB,Z \ten{A} H] \rightarrow [HA,Z\ten{A} H]$. This factors as a split epimorphism followed by a split monomorphism since  $Hf \colon HA \rightarrow HB$ factors accordingly (as a consequence of the fact that $\mathrm{coker} Hf \cong Z \ten{\ca{A}} H$ is projective). The claim now follows from the fact that $Z \rightarrow YC$, hence $YA^{\vee} \otimes Z \rightarrow Y(A^{\vee} \otimes C)$, is sent to a split monomorphism by $- \ten{\ca{A}} H$ by Part~(v) of Proposition~\ref{prop:basic_properties_of_H-duals}. Thus the fiber $F$ of $B^{\vee} \otimes B \rightarrow A^{\vee} \otimes C$ is an $H$-dual.
 
 We next claim that the morphism $HF \rightarrow H(B^{\vee} \otimes C)$ also factors as a split epimorphism followed by a split monomorphism. By Lemma~\ref{lemma:kernel_description}, the image factorization is the composite $J_0 \rightarrow [HB, Z \ten{\ca{A}} H]$ followed by the split monomorphism $[HB,Z \ten{\ca{A}} H] \rightarrow [HB, HC]$. The kernel $J_0$ of $[Hf,\id]$ is isomorphic to $[Z\ten{\ca{A}} H,Z\ten{\ca{A}} H]$, so the kernel is a split monomorphism since $HB \rightarrow Z\ten{\ca{A}}H$ is a split epimorphism. Thus the fiber $F^{\prime}$ of $F \rightarrow B^{\vee} \otimes C$ is also an $H$-dual.
 
 The resulting sequence
 \[
 \xymatrix{YF^{\prime} \ar[r] & YF \ar@{->>}[r] & [Z,Z]}
 \]
 gives the desired finite presentation of $[Z,Z]$. Lemma~\ref{lemma:kernel_preserved} shows that the exact sequence
 \[
 \xymatrix{0 \ar[r] & [Z,Z] \ar@{ >->}[r] & [YB,Z] \ar[r] & [YA,Z] }
 \]
 is preserved by $-\ten{\ca{A}} H$. It follows that $- \ten{\ca{A}} H$ preserves the internal hom $[Z,Z]$ since any symmetric strong monoidal functor preserves the internal hom $[X,Z]$ if $X$ has a dual.
\end{proof}
 
 \begin{lemma}\label{lemma:presentation_of_canonical_morphism}
 There exists a presentation
 \[
 \xymatrix{
YA_1 \ar[r] \ar[d] & YB_1 \ar@{->>}[r] \ar[d]^{Yg} & [Z,\U] \otimes Z \ar[d] \\
YA_2 \ar[r] & YB_2 \ar@{->>}[r] & [Z,Z]  
 }
 \]
of the canonical morphism $[Z,\U] \otimes Z \rightarrow [Z,Z]$ such that $LYg$ is an epimorphism. Moreover, if $C_2$ denotes the cofiber of $A_2 \rightarrow B_2$, then the morphism $HB_2 \rightarrow HC_2$ factors as a split epimorphism followed by a split monomorphism.
 \end{lemma}
 
 \begin{proof}
 Let
 \[
 \xymatrix{YA_2 \ar[r] & YB_2 \ar@{->>}[r] & [Z,Z]}
 \]
 be the presentation of Lemma~\ref{lemma:internal_hom_preserved}. Since $[Z,Z]\ten{\ca{A}} H \cong [Z\ten{\ca{A}} H, Z \ten{\ca{A}} H]$ by the same lemma, the cokernel of $HA_2 \rightarrow HB_2$ is projective. This implies that the cokernel of $HA_2 [1] \rightarrow HB_2[1]$ is also projective. It follows from the long exact sequence for the homological functor $H$ that $[Z,Z] \ten{\ca{A}} H \rightarrow HC_2$ is a split monomorphism.
 
 Consider the pullback
 \[
 \xymatrix{E \ar@{->>}[r] \ar[d] & [Z,\U]\otimes Z \ar[d] \\ YB_2 \ar@{->>}[r] & [Z,Z]}
 \]
 in $[\ca{A}^{\op},\Ab]$. By Lemma~\ref{lemma:dual_finitely_presentable_and_preserved}, the object $[Z,\U]\otimes Z$ is finitely presentable. The morphism $[Z,\U] \otimes Z \rightarrow [Z,Z]$ is sent to an isomorphism by $- \ten{\ca{A}} H$, so it is sent to an epimorphism by $L$ since the restriction of $-\ten{\ca{A}} H$ to $\sh(\ca{A})$ is a Nakayama functor by Part~(iii) of Theorem~\ref{thm:Grothendieck_tensor_from_homology_theory}. Since $L$ preserves pullbacks, $L(E \rightarrow YB_2)$ is an epimorphism. Thus we can find an object $B_1 \in \ca{A}$ and a morphism $YB_1 \rightarrow E$ such that the resulting square
 \[
 \xymatrix{YB_1 \ar@{->>}[r] \ar[d]^{Yg} & [Z,\U]\otimes Z \ar[d] \\ YB_2 \ar@{->>}[r] & [Z,Z]}
 \]
 satisfies the requirements that $LYg$ is an epimorphism and the top horizontal morphism is an epimorphism in $[\ca{A}^{\op},\Ab]$.
 
 Let $K_1$ be the kernel of $YB_1 \rightarrow [Z,\U]\otimes Z$, let $K_2$ be the kernel of $YB_2 \rightarrow [Z,Z]$, and let
 \[
 \xymatrix{E^{\prime} \ar@{->>}[r] \ar[d] & K_1 \ar[d] \\ YA_2 \ar@{->>}[r] & K_2}
 \]
 be a pullback square. Since $K_1$ is finitely generated, we can find an object $A_1 \in \ca{A}$ and a morphism $YA_1 \rightarrow A^{\prime}$ such that that $YA_1 \rightarrow E^{\prime} \rightarrow K_1$ is an epimorphism.
 \end{proof}
 
 \begin{proof}[Proof of Theorem~\ref{thm:detecting_duals}]
 It follows from Lemmas~\ref{lemma:dual_finitely_presentable_and_preserved} and \ref{lemma:internal_hom_preserved} that the canonical morphism $[Z,\U] \otimes Z \rightarrow [Z,Z]$ is sent to an isomorphism by $-\ten{\ca{A}} H$. By Part~(iii) of Theorem~\ref{thm:Grothendieck_tensor_from_homology_theory}, the functor $-\ten{A} H \colon \sh(\ca{A}) \rightarrow \ca{K}_B$ is a Nakayama functor. Thus $L([Z,\U] \otimes Z \rightarrow [Z,Z])$ is an epimorphism. In order to show that it is an isomorphism, it suffices to check that its kernel is preserved by $-\ten{A} H$. To do this, we will show that the presentation of Lemma~\ref{lemma:presentation_of_canonical_morphism} satisfies the conditions of Lemma~\ref{lemma:kernel_preserved}.
 
 Let $F$ be the fiber of $B_1 \rightarrow C_2$. Since $Hg$ is a split epimorphism, the morphism $HB_1 \rightarrow HC_2$ factors as a split epimorphism followed by a split monomorphism. It follows that $HF$ has a dual.
 
 Since $-\ten{\ca{A}} H$ sends $[Z,\U] \otimes Z \rightarrow [Z,Z]$ to an isomorphism, the composite
 \[
 \xymatrix{HF \ar[r] & HB_1 \ar@{->>}[r] & [Z,\U] \otimes Z \ten{\ca{A}} H }
 \]
 is zero. It follows that the composite $HF \rightarrow HB_1 \rightarrow HC_1$ is zero, so the fiber $F^{\prime}$ of $F \rightarrow C_1$ is an $H$-dual.
 
 Lemma~\ref{lemma:kernel_preserved} shows that the kernel $K$ of $[Z,\U] \otimes Z \rightarrow [Z,Z]$ is preserved by $-\ten{\ca{A}}H$. Thus $LK$ is sent to zero by $-\ten{\ca{A}} H \colon \sh(\ca{A}) \rightarrow \ca{K}_B$. Since this is a Nakayama functor, it follows that $LK \cong 0$. The functor $L$ is exact and symmetric strong monoidal, so it preserves the internal hom between finitely presentable objects. If follows that the canonical morphism
 \[
 [LZ,\U] \otimes LZ \rightarrow [LZ,LZ]
 \]
 is an isomorphism, so $LZ$ has a dual.
 \end{proof}

 \subsection{Existence of flat replacements}\label{section:existence_theorem}

 In this section, we assume that the requirements of \S \ref{section:grading} and \S \ref{section:detecting_duals} are both satisfied. Thus $\ca{P}$ is the free $\Ab$-category on some abelian group $G$, with strict monoidal structure given by addition in $G$ and symmetry given by the Koszul sign rules according to some homomorphism $\varepsilon \colon G \rightarrow \{1,-1\}$. Furthermore, we assume that the tensor triangulated category $\ca{T}$ is monoidal closed and that the internal hom objects are compatible with the triangulated structure.
 
 We fix a symmetric lax monoidal homological functor $F \colon \ca{T} \rightarrow \Ab$ and we let $H \defl \bar{F}^{\mathrm{norm}} \colon \ca{T} \rightarrow \ca{K}_B$, where $B=\bar{F} \U$ is the coefficient ring of $F$ (recall that $\bar{F}$ sends $X$ to $F\bigl(K(-)^{\vee} \otimes X\bigr)$ by Definition~\ref{dfn:K-functor_from_Ab-functor}). We write $\ca{A} \defl \ca{A}_H$ for the additive site of $H$-duals (see Definition~\ref{dfn:H-dual}) and $\ca{C} \defl \sh(\ca{A})$ for the category of sheaves on $\ca{A}$. We denote the associated sheaf functor by $L \colon [\ca{A}^{\op},\Ab] \rightarrow \ca{C}$.
 
 The composite
 \[
 \xymatrix{ \ca{P} \ar[r]^-{K} & \ca{A} \ar[r]^-{Y} & [\ca{A}^{\op},\Ab] \ar[r]^-{L} & \ca{C} } 
 \]
 endows $\ca{C}$ with an enrichment over $\ca{K}=(G,\varepsilon)\mbox{-}\Ab$ (see Proposition~\ref{prop:strong_monoidal_enriched_equivalent_to_slice}).
 
 \begin{lemma}\label{lemma:H_graded_nakayama}
 The functor $-\ten{\ca{A}} H \colon \ca{C} \rightarrow \ca{K}_B$ is a graded Nakayama functor.
 \end{lemma}
 
 \begin{proof}
 The functor $H$ is the underlying additive functor of the $\ca{K}$-functor $\mathbb{F}^{\mathrm{norm}}$ (see Proposition~\ref{prop:symmetric_lax_monoidal_K-functor}), so $H$ is compatible with the $\ca{K}$-enrichment. The isomorphism $- \ten{\ca{A}}  H\circ LY \cong H$ is composed of $\eta \ten{\ca{A}} H \colon - \ten{\ca{A}} H \rightarrow (L-)\ten{\ca{A}} H$ and $Y \ten{\ca{A}} H \cong H$, which are monoidal since $\eta$ is monoidal respectively by \cite[Theorem~5.1]{IM_KELLY}. Thus $-\ten{\ca{A}} H \colon \ca{C} \rightarrow \ca{K}_B$ is a graded tensor functor.
 
 The functor $-\ten{\ca{A}} H$ is a Nakayama functor by Part~(iii) of Theorem~\ref{thm:Grothendieck_tensor_from_homology_theory}
 \end{proof}

\begin{lemma}\label{lemma:detecting_zero_morphism}
 If $f \colon A \rightarrow A^{\prime}$ is a morphism in $\ca{A}$ with $Hf=0$, then $LYf=0$ in $\ca{C}$.
\end{lemma} 

\begin{proof}
 If $Hf=0$, then $\mathrm{coker} Hf \cong HA^{\prime}$ has a dual. Thus the fiber $k \colon F_f \rightarrow A$ of $f$ lies in $\ca{A}$ by Part~(v) of Proposition~\ref{prop:basic_properties_of_H-duals}. Exactness of the sequence
 \[
 \xymatrix{HF_f \ar[r]^-{Hk} & HA \ar[r]^-{0} & HA^{\prime}}
 \]
 implies that $k$ is an $H$-epimorphism. Thus $LYk$ is an epimorphism and the claim follows from the fact that $fk=0$.
\end{proof}

\begin{thm}\label{thm:detecting_even_type}
 The functor $H \colon \ca{A} \rightarrow \ca{K}_B$ detects locally free objects of even type in $\ca{C}$: if $X \in \ca{A}$ is an object such that $HX$ is locally free of even type $(i_1, \ldots, i_n)$, then there exists an Adams algebra $A$ in $\ca{C}$ such that $(LYX)_A \cong \oplus_{k=1}^{n} L_{i_k}$.
\end{thm}  
 
 \begin{proof}
 The objects $(L_{i_k}^{-1} \otimes LYX) \ten{\ca{A}} H$ and $(L_{i_k}^{-1} \otimes LYX)^{\vee} \ten{\ca{A}} H$ are locally free of even type, so their symmetric powers are also locally free of even type. It follows in particular that they have duals. Since $\mathrm{Sym}^{j}(L_{i_k}^{-1} \otimes YX)$ and $\mathrm{Sym}^{j}\bigl((L_{i_k}^{-1} \otimes YX)^{\vee} \bigr)$ are finitely presentable objects of $[\ca{A}^{\op},\Ab]$, it follows from Theorem~\ref{thm:detecting_duals} that $\mathrm{Sym}^{j}(L_{i_k}^{-1} \otimes LYX)$ and $\mathrm{Sym}^{j}\bigl((L_{i_k}^{-1} \otimes LYX)^{\vee}\bigr)$ have duals.
 
 This shows that the condition of Theorem~\ref{thm:iterated_splitting} are satisfied, so there exists an Adams algebra $A$ in $\ca{C}$ such that $(LYX)_A \cong X^{\prime} \oplus L_{i_1} \oplus \ldots \oplus L_{i_n}$ for some $A$-module $X^{\prime}$. By Proposition~\ref{prop:unmixed_locally_free_detection}, the identity of $X^{\prime} \otimes L_{i_1} \otimes \ldots \otimes L_{i_n}$ factors through the antisymmetrizer on $(LYX)_A^{\otimes (n+1)}$. To show that $X^{\prime} \cong 0$ it thus suffices to check that the antisymmetrizer on $LYX^{\otimes (n+1)}$ is zero (since the antisymmetrizer is preserved by base change). By Lemma~\ref{lemma:detecting_zero_morphism}, it suffices to check that the antisymmetrizer on $HX^{\otimes (n+1)} \in \ca{K}_B$ is zero.
 
 By assumption, there exists a faithfully flat $B$-algebra $B^{\prime}$ such that $HX_{B^{\prime}} \cong \oplus_{k=1}^{n} L_{i_k}$. Since each $i_k$ is even, we can find a faithfully flat algebra $B^{\prime \prime} \in \ca{K}_B$ such that $HX_{B^{\prime \prime}} \cong (B^{\prime \prime})^{\oplus n}$. This object lies in the image of the unique symmetric strong monoidal functor $\Ab \rightarrow \ca{K}_{B^{\prime\prime}}$, so the antisymmetrizer on $HX_{B^{\prime\prime}}$ vanishes since the exterior power $\Lambda^{n+1} (\mathbb{Z}^{\oplus n})$ is zero.
 \end{proof}
 
 In the proof of the above theorem, we have used the fact that symmetric powers of locally free objects of even type are again locally free of even type. In general, the same need not be true for locally free objects of mixed type. It is however true if $2$ is a unit in $B$.
 
 \begin{lemma}\label{lemma:locally_split_closed_under_symmetric_powers}
 If $2$ is a unit in $B$, then locally free objects of arbitrary type in $\ca{K}_B$ are closed under symmetric powers.
 \end{lemma}
 
 \begin{proof}
 The direct sum formula for symmetric powers reduces the problem to the case of a single object $L_i$. If $i$ is even, then $\mathrm{Sym}^{j} L_i \cong L_i^{\otimes j}$. If $i$ is odd, then $\mathrm{Sym}^{0} L_i =\U$, $\mathrm{Sym}^{1} L_i =L_i$, and $\mathrm{Sym}^{j} L_i =0$ for $j>1$. Indeed, writing $p \colon L_i^{\otimes j} \rightarrow \mathrm{Sym}^{j} L_i$ for the projection, we have $p=-p$ if $j>1$, so $2p=0$. Since $2$ is a unit, the claim follows.
\end{proof}  
 
 \begin{lemma}\label{lemma:Young_symmetrizer_zero}
 If $X \in \ca{K}_B$ is locally free of type $(i_1, \ldots, i_p, i_{p+1}, \ldots, i_{p+q})$ with $(i_1,\ldots, i_p)$ even and $(i_{p+1},\ldots , i_{p+q})$ odd, then the Young symmetrizer $S_X^{(q+1,p+1)}$ (see Equation~\ref{eqn:young_symmetrizer}) vanishes.
 \end{lemma}
 
\begin{proof}
This reduces to the case $X=\oplus_{k=1}^{q+p} L_{i_k}$ in $(G,\varepsilon)\mbox{-}\Ab$. We can further reduce this to the case $(G,\varepsilon)=(\mathbb{Z}\slash 2,\alpha)$ where $\alpha(1 +2\mathbb{Z})=-1$ (cf.\ Lemma~\ref{lemma:locally_Z_2_graded} below). Since the underlying abelian group is free, it suffices to check the claim after base change to $\mathbb{C}$. In this case, $S_X^{(q+1,p+1)}$ is, up to scalar multiple, the projection onto the Schur functor defined by the rectangular Young diagram with $p+1$ rows and $q+1$ columns. This Schur functor vanishes by \cite[Corollaire~1.9]{DELIGNE_TENSORIELLES}, so the projection is the zero morphism, as claimed.
\end{proof}

 \begin{thm}\label{thm:detecting_locally_split_objects}
 Suppose that $2$ is a unit in $B$. Then $H \colon \ca{A} \rightarrow \ca{K}_B$ detects locally free objects of arbitrary type in $\ca{C}$: if $X \in \ca{A}$ is an object such that $HX$ is locally free of type $(i_1, \ldots, i_p, i_{p+1}, \ldots, i_{p+q})$ with $(i_1, \ldots, i_p)$ even and $(i_{p+1}, \ldots, i_{p+q})$ odd, then there exists an Adams algebra $A$ in $\ca{C}$ such that $(LYX)_A \cong \oplus_{k=1}^{p+q} L_{i_k}$.
 \end{thm}
 
 \begin{proof}
 Since $2$ is a unit, Lemma~\ref{lemma:locally_split_closed_under_symmetric_powers} shows that the symmetric powers of the objects $(L_{i_k}^{-1} \otimes LYX) \ten{\ca{A}} H$ and $(L_{i_k}^{-1} \otimes LYX)^{\vee} \ten{\ca{A}} H$ are locally free, so they have duals. Both $\mathrm{Sym}^{j}(L_{i_k}^{-1} \otimes YX)$ and $\mathrm{Sym}^{j}\bigl((L_{i_k}^{-1} \otimes YX)^{\vee}\bigr)$ are finitely presentable objects of $[\ca{A}^{\op},\Ab]$, so Theorem~\ref{thm:detecting_duals} implies that $\mathrm{Sym}^{j}(L_{i_k}^{-1} \otimes LYX)$ and $\mathrm{Sym}^{j}\bigl((L_{i_k}^{-1} \otimes LYX)^{\vee}\bigr)$ have duals. By Theorem~\ref{thm:iterated_splitting}, there exists an Adams algebra $A$ in $\ca{C}$, an $A$-module $X^{\prime}$ and an isomorphism $(LYX)_A \cong  L_{i_1} \oplus \ldots \oplus L_{i_{q+p}} \oplus X^{\prime}$.
 
We will use Proposition~\ref{prop:non-zero_detection} to show that $X^{\prime} \cong 0$. To do this, we need to check that the Young symmetrizer $S_{(LYX)_A}^{(q+1,p+1)}$ is zero. Since Young symmetrizers are preserved by symmetric strong monoidal functors, it suffices to check that $LY(S_X^{(q+1,p+1)})$ is zero. By Lemma~\ref{lemma:detecting_zero_morphism}, this is the case if $H(S_X^{(q+1,p+1)}) \cong S_{HX}^{(q+1,p+1)}$ is zero. Since $HX$ is locally free of type $(i_1, \ldots, i_{p+q})$, the Young symmetrizer $S_{HX}^{(q+1,p+1)}$ vanishes by Lemma~\ref{lemma:Young_symmetrizer_zero}.
 \end{proof}
 
 Recall that an epimorphism $p \colon M \rightarrow \U$ in $\ca{C}$ is called \emph{locally split} if there exists a faithfully flat algebra $A$ such that the base change $p_A \colon M_A \rightarrow A$ is a split epimorphism (see \cite[Definition~5.2.1]{SCHAEPPI_COLIMITS}). If $\mathrm{Sym}^{j}(M^{\vee})$ has a dual for all $j \in \mathbb{N}$, then $p$ is locally split if and only if the algebra $A_p$ given by the colimit of the chain
 \[
 \xymatrix{ \U \ar[r] & M^{\vee} \ar[r] & \mathrm{Sym}^{2}(M^{\vee}) \ar[r] & \ldots }
 \]
 with connecting morphisms as in \cite[Lemma~5.3.3]{SCHAEPPI_COLIMITS} is an Adams algebra (see \cite[Proposition~5.3.4]{SCHAEPPI_COLIMITS}).
 
 \begin{prop}\label{prop:epimorphism_locally_split}
  Let $X \in \ca{A}$ be an object such that $HX \in \ca{K}_B$ is locally free of type $(i_1,\ldots, i_n)$. If either $(i_1,\ldots,i_n)$ is even or $2$ is a unit in $B$, then every epimorphism $LYX \rightarrow \U$ in $\ca{C}$ is locally split.
 \end{prop}
 
 \begin{proof}
  The assumption implies that $\mathrm{Sym}^{j}(LYX^{\vee}) \ten{\ca{A}} H$ has a dual for all $j \in \mathbb{N}$ (see Lemma~\ref{lemma:locally_split_closed_under_symmetric_powers}). From Theorem~\ref{thm:detecting_duals} it follows that $\mathrm{Sym}^{j}(LYX^{\vee})$ has a dual for all $j \in \mathbb{N}$. Since $-\ten{\ca{A}} H$ is a Nakayama functor (see Lemma~\ref{lemma:H_graded_nakayama}), it reflects epimorphisms between objects with duals. Moreover, the epimorphism
  \[
  LYX \ten{A} H \rightarrow \U \ten{A} H \cong B
  \]
  is split since $B$ is projective. Thus $p$ is a locally split epimorphism by \cite[Proposition~5.4.3(iii)]{SCHAEPPI_COLIMITS}.
 \end{proof}
 
 \begin{thm}\label{thm:existence_of_Adams_replacement}
 Suppose that $2$ is a unit in $B$. If every finitely generated projective object of $\ca{K}_B$ is locally free, then there exists a commutative algebra $B^{\prime}$ in $\ca{K}$ and a faithful and exact symmetric strong monoidal functor
 \[
 W \colon \sh(\ca{A}) \rightarrow \ca{K}_{B^{\prime}}
 \]
 which is compatible with the $\ca{K}$-enrichment. In particular, the homological functor $F$ has a flat replacement $F(W)$ (see Theorem~\ref{thm:Adams_replacement}). 
 \end{thm}
 
 \begin{proof}
 Since every epimorphism $p \colon LYX \rightarrow \U$ is locally split by Proposition~\ref{prop:epimorphism_locally_split}, there exists by \cite[Proposition~5.2.5]{SCHAEPPI_COLIMITS} a faithfully flat algebra $A$ in $\ca{C}$ such that the unit $A \in \ca{C}_{A}$ is projective.
 
 By Theorem~\ref{thm:detecting_locally_split_objects}, each object $LYX$, $X \in \ca{A}$ is locally free of some type $(i_1, \ldots, i_n)$. Thus there exists an Adams algebra $A_X$ in $\ca{C}$ such that $(LYX)_{A_X} \cong \oplus_{k=1}^{n} L_{i_k}$. The filtered colimit of all the finite tensor products of these defines a commutative algebra $A^{\prime}$ in $\ca{C}$ such that $(LYX)_{A^{\prime}}$ is isomorphic to a finite direct sum of objects $L_i$, $i \in G$ for all $X$ simultaneously. The algebra $A^{\prime}$ is faithfully flat by \cite[Lemma~5.7]{SCHAEPPI_GEOMETRIC}.
 
 The composite
 \[
 \xymatrix@C=35pt{ \ca{P} \ar[r]^-{K}  & \ca{A} \ar[r]^-{Y} & [\ca{A}^{\op},\Ab]  \ar[r]^-{L} & \ca{C} \ar[r]^-{(-)_{A \otimes A^{\prime}}} & \ca{C}_{A \otimes A^{\prime}} }
 \]
 defines a $\ca{K}$-enrichment on $\ca{C}_{A\otimes A^{\prime}}$ such that the base change functor $(-)_{A\otimes A^{\prime}}$ is compatible with the enrichments on $\ca{C}$ and $\ca{C}_{A\otimes A^{\prime}}$. The unit of $\ca{C}_{A \otimes A^{\prime}}$ is projective and it generates $\ca{C}_{A\otimes A^{\prime}}$ as a $\ca{K}$-category since the object $L_i \in \ca{C}_{A\otimes A^{\prime}}$ is the copower of the unit by $L_i \in \ca{K}$. Moreover, the $\ca{K}$-functor represented by the unit preserves all conical colimits (since $A\otimes A^{\prime}$ and all its shifts by $L_i$ are finitely presentable and projective) and it preserves all copowers by $L_i$ (since these are absolute $\ca{K}$-colimits). Thus the object $A\otimes A^{\prime}$ is a small-projective generator in the sense of \cite[\S 5.5]{KELLY_BASIC}. Let $B^{\prime} \in \ca{K}$ denote the endomorphism algebra of the unit object of $\ca{C}_{A\otimes A^{\prime}}$. The canonical $\ca{K}$-functor
 \[
 \ca{K}_{B^{\prime}} \rightarrow \ca{C}_{A \otimes A^{\prime}}
 \]
 is symmetric strong monoidal by \cite[Theorem~5.1]{IM_KELLY} and an equivalence by \cite[Theorem~5.26]{KELLY_BASIC}. Composing the base change functor with the inverse of this equivalence yields the desired functor $W \colon \ca{C} \rightarrow \ca{K}_{B^{\prime}}$.
 
 That the functor $F(W)$ given by the composite~\ref{eqn:Adams_replacement_from_W} is a flat replacement of $F$ is immediate from Definition~\ref{dfn:Adams_replacement}.
 \end{proof}
 
\begin{rmk}\label{rmk:Hopf_algebroid}
 The proof of Theorem~\ref{thm:existence_of_Adams_replacement} shows that the right adjoint
 \[
 \ca{K}_{B^{\prime}} \rightarrow \ca{C}
 \]
 of $W$ is $\ca{K}$-cocontinuous (it is up to equivalence the functor which sends an $A \otimes A^{\prime}$-module to its underlying object). Moreover, the adjunction satisfies the projection formula by \cite[Proposition~3.8]{SCHAEPPI_GEOMETRIC}. Equivalently, the adjunction is strong coclosed. Thus the induced comonad on $\ca{K}_{B^{\prime}}$ is Hopf monoidal by \cite[Proposition~4.4]{CHIKHLADZE_LACK_STREET}, so it is given by tensoring with a flat commutative Hopf algebroid $( B^{\prime}, \Gamma)$ in $\ca{K}$. The functor $W$ is comonadic, so we obtain an equivalence
 \[
 \ca{C} \simeq \Comod(B^{\prime}, \Gamma)
 \]
 of symmetric monoidal categories.
\end{rmk}
 
 \begin{example}
 The commutative ring spectra $\mathrm{H} \mathbb{Z}[\frac{1}{2}]$ and $\mathrm{H} \mathbb{Z}_{(p)}$ for $p$ an odd prime have flat replacements. The coefficient rings are given by $\mathbb{Z}[\frac{1}{2}]$ and $\mathbb{Z}_{(p)}$ respectively, so $2$ is a unit. All finitely generated projective (graded) modules over these rings are free, so a flat replacement exists by Theorem~\ref{thm:existence_of_Adams_replacement}.
 
 Note that in these cases, we could also localize the arguments of Conner--Smith \cite[\S 3]{CONNER_SMITH} to show that specific flat replacements are given by $\mathrm{MU}[\frac{1}{2}]$ respectively $\mathrm{MU}_{(p)}$.
 \end{example}

\begin{example}
 Let $p$ be an odd prime, $n \in \mathbb{N}$. The $p$-local truncated Brown--Peterson spectrum $\mathrm{BP} \langle n \rangle$ has coefficient ring $\mathrm{BP}\langle n \rangle_{\ast}=\mathbb{Z}_{(p)}[v_1, \ldots, v_n]$, so they have flat replacements by Theorem~\ref{thm:existence_of_Adams_replacement}. Since $\mathrm{H}\mathbb{Z}_{(p)}$ is a $\mathrm{BP}\langle n \rangle$-algebra, every $\mathrm{BP}\langle n \rangle_{\ast}$-dual is in particular an $(\mathrm{H}\mathbb{Z}_{(p)})_{\ast}$-dual. The results of Conner--Smith \cite[\S 3]{CONNER_SMITH} therefore show that $\mathrm{MU}_{(p)}$ is a flat replacement of $\mathrm{BP} \langle n \rangle$.
\end{example}

\begin{example}\label{example:derived}
 Let $R=\mathbb{Q}[x,y] \slash (xy)$ and let $\ca{T}$ be the unbounded derived category $\ca{D}(R)$ of $R$-modules. Consider the $\mathbb{Z}$-grading with $\varepsilon(n)=(-1)^n$ and
 \[
 K \colon (\mathbb{Z},\varepsilon) \rightarrow \ca{D}(R)
 \]
 the functor which sends $n$ to the chain complex $R$ concentrated in degree $n$.
 
 Let $A=R \slash y \cong \mathbb{Q}[x]$. The $R$-algebra $A$ is not $\ca{D}(R)$-flat. Indeed, using the fact that $\mathbb{Q}[x,y]$ is a unique factorization domain, one checks that the sequence
 \[
 \xymatrix{ \ldots \ar[r]^x & R \ar[r]^y & R \ar[r]^x & R \ar[r]^y & R \ar[r] & R\slash y \ar[r] & 0 }
 \]
 is exact. Thus $H_{\ast}(A \mathop{\otimes^{L}} A)=\mathrm{Tor}_{\ast}^R(A,A)$ is given by the homology of the complex
 \[
 \xymatrix{ \ldots \ar[r]^0 & A \ar[r]^x & A \ar[r]^0 & A \ar[r]^x & A \ar[r]^0 & A \ar[r] & 0 } \smash{\rlap{,}}
 \]
 so it is equal to $A \slash x$ in degree $1$. The module $A \slash x \cong \mathbb{Q}$ is not flat over $H_{\ast} A=A$. If $A$ were $\ca{D}(R)$-flat, then $H_{\ast}(A \mathop{\otimes^{L}} A)$ would be a filtered colimit of finitely generated projective $A$-modules, hence flat.
 
 Since $2$ is a unit in $A$, $A$ is connected, and all finitely generated projective $A$-modules are free, there exists a $\ca{D}(R)$-flat replacement of $A$ by Theorem~\ref{thm:existence_of_Adams_replacement}.
\end{example}

 \begin{prop}\label{prop:no_flat_algebra_gives_flat_replacement}
 No flat $R$-algebra $R^{\prime}$ gives a flat replacement of the homology theory $A_{\ast}=H_{\ast} (A \mathop{\otimes^{L}}-)$ of Example~\ref{example:derived}.
 \end{prop}

\begin{proof}
 Let $R^{\prime}$ be a flat $R$-algebra and suppose that $R^{\prime}_{\ast}=H_{\ast}(R^{\prime} \mathop{\otimes^L}-)$ is a flat replacement of $A=R\slash y$. The complex
 \[
 \xymatrix{0 \ar[r] & R \ar[r]^y & R \ar[r] & 0}
 \]
 has finitely generated projective $A_{\ast}$-homology, so it is an $A_{\ast}$-dual. By Theorem~\ref{thm:Adams_replacement}~(i), every $A_{\ast}$-dual is an $R^{\prime}_{\ast}$-dual, so the homology of
 \[
 \xymatrix{0 \ar[r] & R^{\prime} \ar[r]^y & R^{\prime} \ar[r] & 0}
 \]
 is projective over $R^{\prime}$. In particular, $R^{\prime \prime} \defl R^{\prime} \slash y$ is a projective $R^{\prime}$-module. Since $R^{\prime}$ is flat over $R$, so is $R^{\prime \prime}$. From the long exact sequence of $R$-modules in Example~\ref{example:derived} it follows that the sequence
 \[
 \xymatrix{ \ldots \ar[r] & R^{\prime \prime} \ar[r]^{0} & R^{\prime \prime} \ar[r]^x & R^{\prime \prime} \ar[r]^{0} & R^{\prime \prime} \ar[r] & \ldots}
 \]
 is exact, so $x$ is a unit in $R^{\prime \prime}$. Thus there exist $z,z^{\prime} \in R^{\prime}$ such that $xz=1+yz^{\prime}$.
 
 The composite
 \[
 \xymatrix{ R^{\prime} \ar[r]^-{
 \left(
 \begin{smallmatrix}
  z\\ -z^{\prime}
  \end{smallmatrix}
  \right)
  } 
  & R^{\prime} \oplus R^{\prime} \ar[r]^-{
 (  \begin{smallmatrix}
  x && y
  \end{smallmatrix} )
  } & R^{\prime} }
 \]
 is equal to the identity, so $p \defl (  \begin{smallmatrix}  x && y  \end{smallmatrix} ) \colon R \oplus R \rightarrow R $, considered as a morphism of chain complexes concentrated in degree $0$, is an $R^{\prime}_{\ast}$-epimorphism. But $A \otimes p=( \begin{smallmatrix} x && 0 \end{smallmatrix} ) $ is not an epimorphism, so $p$ is not an $A_{\ast}$-epimorphism. By Theorem~\ref{thm:Adams_replacement}~(i), a morphism between $A_{\ast}$-duals is an $A_{\ast}$-epimorphism if and only if it is an $R^{\prime}_{\ast}$-epimorphism, a contradiction. Thus $R^{\prime}$ cannot be a flat replacement of $A$.
\end{proof}

 It would be interesting to find an explicit flat replacement of $A=R\slash y$. Is the commutative differential graded algebra
 \[
 A^{\prime} \defl \xymatrix{R \ar[r]^{y} & R}
 \]
 (concentrated in degree $0$ and $1$) a flat replacement of $A$?

 The connection between the original coefficient ring $B$ and the new coefficient ring $B^{\prime}$ is not very strong. It is however true that $B^{\prime}$ does not have ``more'' torsion than $B$, in the sense made precise in Proposition~\ref{prop:torsion} below.
 
  \begin{lemma}\label{lemma:detecting_monomorphisms}
 Let $f \colon A \rightarrow A^{\prime}$ be a morphism in $\ca{A}$ such that $Hf$ is a monomorphism. Then $LYf$ is a monomorphism in $\ca{C}$.
 \end{lemma}
 
 \begin{proof}
 Let $j \colon K \rightarrow YA$ be the kernel of $Yf$ in $[\ca{A}^{\op},\Ab]$ and choose an epimorphism
 \[
(g_i)_{i \in I} \colon \bigoplus_{i \in I} YA_i \rightarrow K
 \]
 in $[\ca{A}^{\op},\Ab]$. Using the exactness of $L$, we find that it suffices to check that $Lg_i=0$ for all $i \in I$. Since $Lj$ is monic, this amounts to showing that $L(jg_i)=0$ for all $i \in I$. This reduces the problem to checking the following implication: if $g \colon Z \rightarrow A$ is a morphism in $\ca{A}$ such that $fg=0$, then $LY(g)=0$.
 
 Thus let $g \colon Z \rightarrow A$ be a morphism in $\ca{A}$ with $fg=0$. Then $HfHg=0$, so $Hg=0$ by the assumption that $Hf$ is a monomorphism. Thus $\mathrm{coker}(Hg) \cong HA$ has a dual. By Part~(v) of Proposition~\ref{prop:basic_properties_of_H-duals} it follows that the fiber $k \colon F_g \rightarrow Z$ of $g$ lies in $\ca{A}$. Moreover, the exact sequence
 \[
 \xymatrix{HF_g \ar[r]^{Hk} & HZ \ar[r]^{0} & HA}
 \]
 shows that $k$ is an $H$-epimorphism. Thus $LYk$ is an epimorphism by the definition of the additive Grothendieck topology on $\ca{A}$. We have $gk=0$, so $LYg \circ LYk=0$. This shows that $LYg=0$, which concludes the proof.
 \end{proof}

\begin{prop}\label{prop:torsion}
 Let $p$ be a prime. In the situation of Theorem~\ref{thm:existence_of_Adams_replacement}, if $B$ does not have any $p$-torsion, then $B^{\prime}$ does not have any $p$-torsion.
\end{prop} 

\begin{proof}
 The assumption implies that multiplication with $p$ is a monomorphism in $\ca{K}_B$. By Lemma~\ref{lemma:detecting_monomorphisms} it follows that multiplication with $p$ is a monomorphism from the unit object of $\ca{C}$ to itself. Since the functor $W$ of Theorem~\ref{thm:existence_of_Adams_replacement} is exact, it follows that multiplication with $p$ is a monomorphism from $B^{\prime}$ to $B^{\prime}$ in $\ca{K}_{B^{\prime}}$, so $B^{\prime}$ does not have any $p$-torsion.
\end{proof}

 \subsection{Locally free objects in module categories}
 
 In this section we keep the assumption that $\ca{K}=(G,\varepsilon)\mbox{-}\Ab$ and we further assume that $\varepsilon$ is surjective, that is, there exist $i \in G$ with $\varepsilon(i)=-1$. Let $B$ be a commutative ring in $\ca{K}$. A $B$-module $M$ is called \emph{free} if it is isomorphic to $\oplus_{i=1}^n L_{i_k}$ for some $i_1, \ldots, i_n \in G$. The $B$-module $M$ is called \emph{locally free} if there exists a faithfully flat $B$-algebra $B^{\prime}$ such that the base change $M_{B^{\prime}} = B^{\prime} \ten{B} M$ is free. An \emph{ideal} of $B$ will always mean a homogeneous ideal, that is, a subobject of $B$ in the category $\ca{K}_B$ of $B$-modules in $\ca{K}$. A graded ring $B$ is called \emph{local} if it has a unique maximal ideal $\mathfrak{m}$.
 
 \begin{lemma}[Nakayama Lemma]\label{lemma:graded_Nakayama}
 If $(B,\mathfrak{m})$ is local and $M$ is a finitely generated $B$-module with $M \slash \mathfrak{m} M \cong 0$, then $M \cong 0$.
 \end{lemma}
 
 \begin{proof}
 Suppose $M$ is generated by $x_1,\ldots, x_n$. Since $\mathfrak{m} M=M$, there exist elements $a_1, \ldots, a_n$ such that $x_n= \sum_{i=1}^n a_i x_i$. Thus $(1-a_n) x_n= \sum_{i=1}^{n-1} a_i x_i$. The ideal generated by $(1-a_n)$ is not contained in $\mathfrak{m}$, so it is equal to $B$. In other words, $(1-a_n)$ is a unit. Thus $M$ is generated by $x_1, \ldots, x_{n-1}$. Iterating this procedure we find that $M \cong 0$. 
 \end{proof}
 
 \begin{lemma}\label{lemma:local_ring_free}
 If $(B,\mathfrak{m})$ is local, then every finitely generated projective $B$-module is free.
 \end{lemma}
 
 \begin{proof}
 The quotient $B \slash \mathfrak{m} B$ is a graded field: each non-zero element is a unit. Note that every finitely generated module $N$ over a graded field is free. Indeed, if $x_1, \ldots, x_n$ is a minimal generating set of $N$ corresponding to the epimorphism $p \colon \oplus_{k=1}^n L_{i_k} \rightarrow N$, then $\mathrm{ker}(p)=0$ since $\sum_k \lambda_k x_k=0$ implies $\lambda_k=0$ for all $k$ by minimality.
 
 Let $M$ be a finitely generated projective $B$-module, $\bar{x}_1, \ldots \bar{x}_n$ a minimal generating set of $M \slash \mathfrak{m} M$. The preimages $x_1, \ldots, x_n$ in $M$ correspond to a morphism $p \colon \oplus_{k=1}^n L_{i_k} \rightarrow M$ which is an epimorphism by Nakayama. Since $M$ is projective, $p$ is split, so the kernel $K$ of $p$ is preserved by $B \slash \mathfrak{m} \ten{B} -$. The minimality of the generating set $\bar{x}_1, \ldots \bar{x}_n$ implies that $K \slash \mathfrak{m} K \cong 0$, so $K \cong 0$ by Nakayama.
 \end{proof}
 
 As in the ungraded case, we can form the localization $S^{-1} B$ of a graded ring $B$ at a multiplicative graded set $S$. Given $f \in B$, we write $B_f$ for the localization at $\{1,f,f^2,\ldots\}$.
 
 \begin{lemma}\label{lemma:free_spread_out}
 Let $M$ be a finitely generated projective $B$-module. Then for each maximal ideal $\mathfrak{m}$ in $B$, there exists $f \notin \mathfrak{m}$ such that $B_f \ten{B} M$ is free.
 \end{lemma}
 
 \begin{proof}
 The graded set $B \setminus \mathfrak{m}$ is multiplicative and $S^{-1} \mathfrak{m} \subseteq S^{-1} B$ is the unique maximal ideal of $S^{-1} B$. Thus $S^{-1} B \ten{B} M$ is free by Lemma~\ref{lemma:local_ring_free}. The ring $S^{-1} B$ is the filtered colimit of the localizations $T^{-1} B$ where $T \subseteq S$ is finitely generated. The data of an isomorphism $\oplus_{k=1}^n L_{i_k} \rightarrow S^{-1}B \ten{B} M$ involves only finitely many denominators, so the isomorphism is defined over $T^{-1} B$ for some finitely generated multiplicative set $T$. Taking $f$ to be the product of the generators of $T$ yields the desired element in $B \setminus \mathfrak{m}$ such that $B_f \ten{B} M$ is free.
 \end{proof}
 
 Let $\alpha \colon \mathbb{Z} \slash 2 \rightarrow \{1, -1\}$ be the canonical isomorphism. Then the homomorphism $\alpha^{-1} \circ \varepsilon \colon G \rightarrow \mathbb{Z} \slash 2$ induces a symmetric strong monoidal left adjoint
 \[
 (\alpha^{-1} \circ \varepsilon)_{\ast} \colon \ca{K}=(G,\varepsilon)\mbox{-}\Ab \rightarrow (\mathbb{Z}\slash 2,\alpha)\mbox{-}\Ab
 \]
 which sends $(M_i)_{i \in G}$ to $(\oplus_{\varepsilon(i)=1} M_i, \oplus_{\varepsilon(j)=-1} M_j)$. We write $\bar{\U}$ for the object $(0,\mathbb{Z})$ in $(\mathbb{Z}\slash 2, \alpha)\mbox{-}\Ab$.
 
 \begin{lemma}\label{lemma:locally_Z_2_graded}
 There exists an Adams algebra $A$ in $\ca{K}$ and an equivalence $\ca{K}_A \simeq (\mathbb{Z}\slash 2,\alpha)\mbox{-}\Ab$ of symmetric monoidal categories such that the triangle
 \[
 \xymatrix{ & \ca{K} \ar[rd]^{(\alpha^{-1} \circ \varepsilon)_{\ast}} \ar[ld]_{(-)_{A}} \\ \ca{K}_A \ar[rr]_-{\simeq} && (\mathbb{Z}\slash 2, \alpha)\mbox{-}\Ab }
 \]
 commutes up to monoidal natural isomorphism.
 \end{lemma}
 
 \begin{proof}
 The right adjoint sends $(N_0,N_1)$ to $(M_i)_{i \in G}$ where $M_i=N_0$ if $\varepsilon(i)=1$ and $M_i=N_1$ if $\varepsilon(i)=-1$. The right adjoint is thus faithful and exact. Since $\ca{K}$ is a Grothendieck tensor category, there exists an algebra $A$ and an equivalence $\ca{K}_A \simeq (\mathbb{Z}\slash 2,\alpha) \mbox{-}\Ab$ with the desired property by \cite[Proposition~3.8]{SCHAEPPI_GEOMETRIC} (note that Grothendieck tensor categories are called pre-geometric in \cite{SCHAEPPI_GEOMETRIC}).
 
 The explicit description of the right adjoint shows that the underlying object of $A$ is $(A_i)_{i \in G}$ where $A_i = \mathbb{Z}$ if $\varepsilon(i)=1$ and $A_i=0$ if $\varepsilon(i)=-1$. Moreover, the functor $(\alpha^{-1} \circ \varepsilon)_{\ast}$ is faithful and exact, so $A$ is an Adams algebra by Proposition~\ref{prop:Adams_iff_faithfully_flat}.
 \end{proof}
 
 The upshot of the above lemma is that each $L_i$ is locally isomorphic to $\U$ if $\varepsilon(i)=1$, respectively to $\bar{\U}$ if $\varepsilon(i)=-1$.
 
 In order to state the next result, we need the notion of the prime spectrum of a graded ring with its Zariski topology. A \emph{prime ideal} of a graded ring is an ideal $\mathfrak{p}$ such that $xy \in \mathfrak{p}$ implies $x \in \mathfrak{p}$ or $y \in \mathfrak{p}$. The set of all prime ideals is denoted by $\Spec(B)$. For $f \in B$, we write $\mathrm{D}(f)$ for the set $\{\mathfrak{p} \in \Spec(B) \vert f \notin \mathfrak{p} \}$. Clearly $\mathrm{D}(f) \cap \mathrm{D}(g)=\mathrm{D}(fg)$, so these sets form a basis of a topology, which we call the \emph{Zariski topology}. If $\Spec(B)=\bigcup_{j \in J} \mathrm{D}(f_j)$, then no prime ideal contains all the $f_j$, so the ideal generated by the $f_j$ must be all of $B$. It follows that there exists $b_1, \ldots b_n$ and $j_1, \ldots j_n \in J$ such that $1=\sum_{k=1}^n b_k f_{j_k}$, hence that $\Spec(B)=\bigcup_{k=1}^n \mathrm{D}(f_{j_k})$, so $\Spec(B)$ is quasi-compact. As for the ordinary Zariski topology of ungraded rings, we have the following connection between idempotents and connected components.
 
 \begin{lemma}\label{lemma:connected}
 If the degree zero component $B_0$ of $B$ has no non-trivial idempotents, then $\Spec(B)$ is connected.
 \end{lemma} 
 
 \begin{proof}
 This follows verbatim the proof in the ungraded case, for example \cite[\href{https://stacks.math.columbia.edu/tag/00EE}{Tag 00EE}]{stacks-project}.
 \end{proof}
 
 We call a graded ring \emph{connected} if it has no non-trivial idempotents.
 
 \begin{lemma}\label{lemma:dimension_well_defined}
  Let $B$ be a $\mathbb{Z} \slash 2$-graded ring such that $2$ is a unit in $B$. If there exists an isomorphism
  \[
  B^p \oplus (B \otimes \bar{\U})^{q} \cong   B^{p^{\prime}} \oplus (B \otimes \bar{\U})^{q^{\prime}} \smash{\rlap{,}}
  \]
  then $p=p^{\prime}$ and $q=q^{\prime}$.
 \end{lemma}
 
 \begin{proof}
 We can mod out by any maximal ideal and reduce to the case where $B$ is a field. Since $2$ is a unit, all odd degree elements are nilpotent, hence zero, so $B=B_0$. Therefore the isomorphism is given by a block diagonal matrix with blocks of size $p \times p^{\prime}$ and $q \times q^{\prime}$. Since $B_0$ is a field, we must have $p=p^{\prime}$ and $q=q^{\prime}$.
 \end{proof}
 
 Given a prime ideal $\mathfrak{p} \subseteq B$, we write $B_{\mathfrak{p}}$ for the localization of $B$ at the multiplicative set $B \setminus \mathfrak{p}$. This is a local ring with maximal ideal $(B \setminus \mathfrak{p})^{-1} \mathfrak{p}$.
 
\begin{prop}\label{prop:projective_module_locally_free}
 Let $B$ be a commutative algebra in $\ca{K}$ such that $2$ is a unit in $B$. If $B$ is connected, then every finitely generated projective $B$-module is locally free: there exists a faithfully flat $B$-algebra $B^{\prime}$ such that $M_{B^{\prime}}$ is free.
\end{prop}

\begin{proof}
 Let $A$ be the faithfully flat algebra of Lemma~\ref{lemma:locally_Z_2_graded} with $\ca{K}_A \simeq (\mathbb{Z}\slash 2,\alpha)\mbox{-}\Ab$. For each $\mathfrak{p} \in \Spec(B)$, the module $A \otimes B_{\mathfrak{p}} \otimes M$ is isomorphic to the base change of $\U^p \oplus \bar{\U}^q$ for some $p$ and $q$ by Lemma~\ref{lemma:local_ring_free}. Moreover, the function which sends $\mathfrak{p}$ to $(p,q) \in \mathbb{N} \times \mathbb{N}$ is well-defined by Lemma~\ref{lemma:dimension_well_defined} and continuous by Lemma~\ref{lemma:free_spread_out}. Since $\Spec(B)$ is connected by Lemma~\ref{lemma:connected}, this function is therefore constant.
 
 For each $\mathfrak{p} \in \Spec(B)$, we can find $f \notin \mathfrak{p}$ such that $B_{f} \ten{B} M$ is free by Lemma~\ref{lemma:free_spread_out}. The ideal generated by all these $f$ is not contained in any maximal ideal, so it is all of $B$. Thus there exist $f_1, \ldots, f_n$ and $b_1, \ldots, b_n$ such that $1=\sum_{i=1}^n b_i f_i$ and $B_{f_i} \ten{B} M$ is free. Let $B^{\prime}=\prod_{i=1}^n A \otimes B_{f_i}$. This is flat as a finite product of flat $B$-algebras. If $B^{\prime} \ten{B} N \cong 0$, then we have in particular $B_{f_i} \ten{B} N \cong 0$ for all $i$, hence $B_{\mathfrak{p}} \ten{B} N \cong 0$ for all $\mathfrak{p} \in \Spec(B)$. As in the ungraded case, this implies that $N \cong 0$, so $B^{\prime}$ is faithfully flat.
 
 For each $i=1, \ldots, n$, the module $A \otimes B_{f_i} \ten{B} M$ is isomorphic to the scalar extension of $\U^{p} \oplus \bar{\U}^q$, so the $B^{\prime}$-module $B^{\prime} \ten{B} M$ is free.
 \end{proof}
 
 As in \S \ref{section:existence_theorem}, we assume that $\ca{P}$ is the free $\Ab$-category on some abelian group $G$, with strict monoidal structure given by addition in $G$ and symmetry given by the Koszul sign rules according to some homomorphism $\varepsilon \colon G \rightarrow \{1,-1\}$. We write $\ca{K}$ for the coefficient category $[\ca{P}^{\op},\Ab]$. Furthermore, we assume that the tensor triangulated category $\ca{T}$ is monoidal closed and that the internal hom objects are compatible with the triangulated structure.
 
 We fix a symmetric lax monoidal homological functor $F \colon \ca{T} \rightarrow \Ab$ and we let $H \defl \bar{F}^{\mathrm{norm}} \colon \ca{T} \rightarrow \ca{K}_B$, where $B=\bar{F} \U$ is the coefficient ring of $F$. We write $\ca{A}$ for the additive site of $H$-duals.
 
 \begin{thm}\label{thm:refined_existence_of_Adams_replacement}
 Suppose that $2$ is a unit in $B$ and $B$ is connected. Then there exists a commutative algebra $B^{\prime}$ in $\ca{K}$ and a faithful and exact symmetric strong monoidal functor
 \[
 W \colon \sh(\ca{A}) \rightarrow \ca{K}_{B^{\prime}}
 \]
 which is compatible with the $\ca{K}$-enrichment. In particular, the homological functor $F$ has a flat replacement $F(W)$ (see Theorem~\ref{thm:Adams_replacement}).
 \end{thm}
 
 \begin{proof}
 This is immediate from Theorem~\ref{thm:existence_of_Adams_replacement} and Proposition~\ref{prop:projective_module_locally_free}.
 \end{proof}

 If $\ca{T}$ is a triangulated category which satisfies Brown representability, then the above theorem assigns a $\ca{T}$-flat quasi-ring $E^{\prime} \in \ca{T}$ to each quasi-ring $E \in \ca{T}$ satisfying the two conditions that $2$ is a unit in $E_{\ast}$ and that $E_{\ast}$ is connected.